\documentclass{amsart}
\usepackage[utf8]{inputenc}
\usepackage[dvipsnames]{xcolor}
\usepackage{amsmath, amscd, bm}
\usepackage{amssymb} 
\usepackage{amsthm}
\allowdisplaybreaks

\usepackage{cancel}
\usepackage{geometry}
\usepackage[all]{xy}
\usepackage{mathrsfs}
\usepackage{mathtools}
\usepackage{tikz, tikz-cd}
\usepackage{multicol}
\usepackage{soul}
\usepackage{subcaption}

\numberwithin{equation}{section}

\newtheorem{theorem}{Theorem}[section]
\newtheorem{lemma}[theorem]{Lemma} 
\newtheorem{proposition}[theorem]{Proposition}
\newtheorem{corollary}[theorem]{Corollary}
\theoremstyle{definition} 
\newtheorem{definition}[theorem]{Definition}

\newtheorem{remark}[theorem]{Remark}
\newtheorem{example}[theorem]{Example}
\newtheorem{claim}[theorem]{Claim}

\DeclareMathOperator{\Ext}{Ext}
\DeclareMathOperator{\Hom}{Hom}

\DeclareMathOperator{\Trop}{Trop}

\newcommand{\Hor}{\mathrm{Hor}}
\newcommand{\Coh}{\mathrm{Coh}}

\newcommand{\bR}{\mathbb R}
\newcommand{\bC}{\mathbb C}

\newcommand{\bZ}{\mathbb Z}
\newcommand{\bD}{\mathbb D}
\newcommand{\bP}{\mathbb P}

\newcommand{\cA}{\mathcal A}

\newcommand{\cD}{\mathcal D}

\newcommand{\cL}{\mathcal L}
\newcommand{\cM}{\mathcal M}
\newcommand{\cO}{\mathcal O}

\newcommand{\mc}[1]{\mathcal{#1}}

\newcommand{\Spec}{\text{Spec}}

\newcommand{\dd}{\partial}

\newcommand{\pr}{\text{pr}}

\usepackage{enumitem}

\title{Semi-orthogonality in Fukaya-Seidel mirrors to blowups of abelian varieties}

\author[C.~Cannizzo]{Catherine Cannizzo}
\address{C.~Cannizzo, Department of Mathematics, Columbia University, New York, NY, 10027}
\thanks{C.~Cannizzo was partially supported by NSF DMS-2316538. She also thanks the Collaborative Research Centre TRR 326 \emph{Geometry and Arithmetic of Uniformized Structures (GAUS)} at Goethe Universit\"at-Frankfurt for hosting her during the completion of this work.}

\author[S.~Venkatesh]{Sara Venkatesh}
 \email{}
 \address{S.~Venkatesh, Palo Alto, CA}
 \date{}
\makeatletter
\@namedef{subjclassname@2020}{\textup{2020} Mathematics Subject Classification}
\makeatother



\usepackage[colorlinks=true, linkcolor=blue, filecolor=blue,citecolor=blue]{hyperref}

\begin{document}
\begin{abstract}
    We prove evidence of Kontsevich's homological mirror symmetry conjecture (HMS) for a blow-up of an abelian surface times the complex plane, on the complex side, and its symplectic Landau-Ginzburg mirror. Specifically, the first author proved evidence of HMS for a 1-parameter family of genus 2 curves on the complex side, as a hypersurface in an abelian surface. The generalized SYZ mirror to the hypersurface is then the SYZ mirror to the Landau-Ginzburg model given by the blow-up of the abelian surface times the complex plane, along the hypersurface times zero, with superpotential given by projection to the complex plane. The mirror to the blow-up - without the superpotential - is obtained by removing a generic smooth fiber from the generalized SYZ mirror superpotential. We prove a categorical HMS result for the latter pair, between categories expected to split-generate. To do so, we equip the punctured superpotential with a Fukaya category which involves both partial and full wrapping in the base of the symplectic Landau-Ginzburg model due to the removal of a generic fiber. Semi-orthogonality appears in the categorical invariants on both sides of HMS. 
\end{abstract}

\maketitle

\tableofcontents

\section{Introduction}

\subsection{Background and motivation}

Homological mirror symmetry (HMS) was conjectured by Kontsevich \cite{hms} as an equivalence of the Fukaya category \cite{fooo1, fuk_intro} of a symplectic manifold $Y$ and the bounded derived category of coherent sheaves of a mirror complex manifold $X$, 
\begin{equation}
    D^b \text{Coh}(X) \cong D^\pi \text{Fuk}(Y).
    \label{eq:hms}
\end{equation}
The symplectic side is referred to as the ``A-side" and the complex side as the ``B-side," so here $X$ is on the B-side and $Y$ is on the A-side. One geometric way to construct the mirror is Strominger-Yau-Zaslow \cite{syz} mirror symmetry. If one can find a special Lagrangian torus fibration on the symplectic manifold, the SYZ mirror is the T-dual torus fibration, which has the same base and dual fibers. Thus, SYZ mirror symmetry provides a candidate pair to prove HMS for. In this paper, we prove evidence in a new example that HMS holds for an SYZ mirror pair.

Calabi-Yau and Fano varieties, and their mirrors, were the first examples of HMS \cite{fuk_abel, ab_sm,zp,sherid_CY,sherid_Fano,seidel,ueda_del_pezzo, Ab06, Ab09, FLTZ12}. Less is known about general type varieties and their mirrors. In particular, hypersurfaces of abelian and toric varieties have been studied more recently. To find a mirror to a hypersurface $H$, the \emph{generalized SYZ} construction can be used \cite{HV00, AAK}. This is because a hypersurface doesn't have an evident Lagrangian torus fibration, the way a toric or abelian variety $V$ does. So we instead consider the intermediary manifold $X$ given by a deformation to the normal cone, because that does admit such a fibration. The superpotential, also known as a \emph{Landau-Ginzburg} (LG) model, given by projection to the deformation parameter, has $H$ as the critical locus so the SYZ mirror superpotential is generalized SYZ mirror to $H$ \cite{AAK, Ca20}. In this paper, we prove a categorical HMS result for $X$ without the superpotential, and its mirror. 

{\bf Related work.} Evidence for categorical HMS of a 1-parameter family of complex genus 2 curves as a hypersurface in the Jacobian torus was proved in \cite{Ca20} for a subcategory of the Fukaya category. SYZ mirror symmetry was proved in \cite{KL}. HMS for complex hypersurfaces of $(\bC^*)^n$ was proved in \cite{AA24} and SYZ mirror symmetry for hypersurfaces of toric varieties in \cite{AAK,CLL}. Other related work includes \cite{Aur18}, relating the Fukaya category of Landau-Ginzburg models $((\bC^*)^n, f)$ to the wrapped Fukaya category of a smooth fiber and of its complement in $(\bC^*)^n$; note that the hypersurface is on the A-side while here it is on the B-side. HMS for toric blow-ups has been studied in \cite[\textsection 6.3]{HH21}, blowing up at toric fixed points, in contrast to hypersurfaces away from toric divisors as here. The works \cite{VWX23, Kerr08} consider semi-orthogonal decompositions in the toric setting, as we do in the abelian variety setting here. Furthermore, \cite{GKR17} discuss other versions of mirror symmetry for general type varieties.

This overview of the generalized SYZ construction is based on \cite[\textsection 3]{AAK}. Let $V$ be a toric variety or a principally polarized abelian variety of complex dimension $n$. Let $\cL \to V$ be an ample holomorphic line bundle with section $s$, and $H=s^{-1}(0)$ its divisor, a complex hypersurface of $V$. Let $X:=\text{Bl}_{H \times \{0\}} V \times \bC$. 

\begin{remark}\label{rem:deform_cone}
    $X$ is known as the deformation to the normal cone because the blow-up map composed with the projection to the $\bC$ factor $y:X \to V \times \bC \to \bC$ parametrizes a deformation of a smooth generic fiber to a singular fiber containing the exceptional divisor $E$ over $0 \in \bC$. Recall that the exceptional divisor is the projectivization of the normal bundle to the submanifold blown-up, here $\mathbb{P}(\mathcal{N}_{H\times\{0\}\backslash V \times \bC})$ where $\mathcal{N}_{H\times\{0\}\backslash V \times \bC} \cong (\mathcal{L} \oplus \mathcal{O}_\bC)|_{H \times \{0\}} \to H \times \{0\}$.
\end{remark}

{\bf The geometry of the blow-up $X$.}

\begin{equation}
    \begin{aligned}
    & X=\text{Bl}_{H \times \{0\}} V \times \mathbb{C} = \{(x_1,\ldots, x_n, y, (u_1:u_2)) \in \mathbb{P}(\mathcal{L} \oplus \mathcal{O}_{\bC}) \mid s(x_1,\ldots, x_n)u_2 = u_1y\}\\
    & \pi: X \xrightarrow{\text{blow-up }} V \times \bC, \quad y: X \xrightarrow{\text{pr}_\bC \circ \pi} \bC.
    \end{aligned}
\end{equation}
Fibers $y^{-1}(t)\cong V$ for $0 \neq t \in \bC$ because ${x}:=(x_1,\ldots,x_n) \in V$ is uniquely determined by $s(x) = t \cdot \frac{u_1}{u_2}$. As outlined in Remark \ref{rem:deform_cone}, over 0 this fiber deforms to $y^{-1}(0)\cong V \cup_H E$ intersecting along $H$  \cite[Proposition 7.3]{AAK}, as follows. First, the exceptional divisor $E$ is a $\bP^1$-bundle; letting $p:=\pi|_E$ we have 
$$
\text{pr}_V \circ p: E \to H.
$$
Then
 \begin{equation}
E =\pi^{-1}(H \times \{0\}) = \{({x},0, (u_1:u_2)) \mid s(x) u_2=u_1\cdot 0\} \ni ({x},0, (u_1:u_2)) \mapsto {x} \in H
 \end{equation}
because $s(x) u_2=u_1\cdot 0$ for all $(u_1:u_2) \in \bP^1\cong \bP(\cL_{{x}}\oplus \cO_{\{0\}})$ implies $s(x)=0$ therefore $x \in H$. Second, the fiber $y^{-1}(0)$ over 0 also contains the proper transform $\tilde V$ of $V \times \{0\}$, which is 
$$
\tilde V= \overline{\pi^{-1}((V\backslash H) \times \{0\})} = \{({x},0,(1:0)) \mid {x} \in V \} \cong V.
$$
Intersecting these two sets
$$
\tilde V \cap E = \{({x},0,(1:0) \mid {x} \in H \} \cong H.
$$

To obtain a Lagrangian torus fibration on $X$, start with the moment map on $V \times \mathbb{C}$  obtained from the torus action $(e^{i\theta_1}x_1,\ldots,e^{i\theta_n}x_n, e^{i\theta_{n+1}}y)$. Then use cut and paste operations \cite{4from2} on the moment polytope in the base to blow up $V\times \mathbb{C}$ along $H \times \{0\}$. The resulting polytope is the base of a Lagrangian torus fibration on $X$, with coordinates
\begin{equation}\label{eq:mom_map_coords}
    (\xi_1,\ldots,\xi_n,\eta) = \frac{1}{2\pi}(\log|x_1|,\ldots,\log|x_n|,\mu_X(x_1,\ldots,x_n,y))
\end{equation}
for moment map $\mu_X$ from the $S^1$-action on the $y$ coordinate of the blow-up $X$ {\cite[\textsection 4]{AAK}}.

{\bf Generalized SYZ mirror symmetry.} Here is an outline of the generalized SYZ procedure \cite{AAK}, which we will later describe in more detail for the setting of this paper. Let
$$
\Omega_{V \times \bC} = i^{n+1} \left(\bigwedge_{j=1}^n d \log x_j \right) \wedge d \log y
$$
be a meromorphic $(n+1)$-form on $V \times \bC$ with simple poles on the toric divisors $D_{V \times \bC} = (V \times 0) \cup (D_V \times \bC)$ where the term $D_V$ is not there if $V$ is an abelian variety. Then we have a canonical form on $X$ given by
$$
\Omega = \pi^* \Omega_{V \times \bC}
$$
with anticanonical divisor
$$
D=\tilde{V} \cup \pi^{-1}(D_V \times \bC).
$$
Let $X^0 = X\backslash D$ be the open Calabi-Yau manifold obtained by removing that anticanonical divisor. Let $Y^0$ be the SYZ mirror of $X^0$ obtained by T-dualizing the Lagrangian torus fibration of \cite[Definition 4.4]{AAK}. The prescription of SYZ produces complex analytic charts, as described by \cite[Equation (2.3)]{AAK}, as well as transition functions described by wall-crossing data \cite[Proposition 5.8]{AAK}, which together define the complex manifold $Y^0$.  Compactifying $X^0$ to $X$, the mirror to $X$ is an LG model $(Y^0, W_0)$ with superpotential $W_0:Y^0 \to \bC$. Compactifying $Y^0$ to $Y$, the mirror to $(Y,-v_0)$ is LG model $(X,y)$ with superpotential $y:X \to \mathbb{C}$. The relation between $W_0$ and $v_0$ for $V$ an abelian variety is given in Equation \eqref{eq:superpot_relns}. In particular, by \cite[Definition 1.2, Theorems 1.5, 1.6, 8.4, 10.4]{AAK}, LG model $(Y,-v_0)$ is a \emph{generalized SYZ mirror} of the hypersurface $H$. 

The manifold $Y$ can be defined by a polytope if $V$ is toric, or a polyhedron if $V$ is abelian. It is defined in terms of the tropicalization of the defining section $s$ of the hypersurface $H$:
$$
\Delta = \{(\xi_1,\ldots,\xi_n,\eta) \in \bR^{n+1} \mid \eta \geq \text{Trop}(s)(\xi_1,\ldots,\xi_n)\}.
$$
For $V$ an abelian variety there are additional steps to quotient by its defining lattice. Both $X$ and $Y$  admit K\"ahler structures and the holomorphic functions $y:X \to \bC$, $W_0:Y^0 \to \bC$, and $-v_0:Y \to \mathbb{C}$ are  Landau-Ginzburg models. In this paper, we take $X$ without superpotential to be the complex side (B-side) of mirror symmetry and $(Y^0,W_0)$ to be a symplectic (A-side) Landau-Ginzburg model, see \cite[Definition 1.1]{ACLL_LG}.

\begin{example} Take $H=\{1\} \subset V=\mathbb{C}^*/\bZ$ to be a point in an elliptic curve. Then $(Y,-v_0)$ is the total space of the Tate curve over a disc, of real dimension 4. The structure sheaf on the point $H$ is mirror to the Lagrangian thimble to the critical point of the Tate curve.
\end{example}

\subsection{Statement of main theorem}\label{section: main_result}

Now we assume $V$ is an abelian surface of complex dimension 2 and $H$ is a genus 2 curve of complex dimension 1. Also, the symbol $n$ will now be used to denote an element in $\mathbb{Z}^2$, so we let the dimension of $V$ be denoted by $g$. The $g=2$ HMS result of \cite{Ca20} was between $H$ and $(Y,-v_0)$. Here it is between $X$ and $(Y^0,W_0)$.

Let $V = (\bC^*)^2/\tau \bZ^2$ be the principally polarized abelian surface, topologically a 4-torus $T^4$, corresponding to the complex modular parameter $\tau = \frac{i}{2\pi} \log t \left( \begin{matrix} 2 & 1\\ 1 & 2 \end{matrix} \right)$, $0<t \in \bR_{\ll 1}$, defined by quotienting by 
\begin{equation}
 (\tau n)\cdot (x_1,x_2) = (e^{2\pi i \sum_{j=1}^2\tau_{1j} n_j} x_1, e^{2\pi i \sum_{j=1}^2 \tau_{2j} n_j} x_2), \quad (x_1,x_2) \in (\bC^*)^2, \quad (n_1,n_2) \in \mathbb{Z}^2.
\end{equation}
Then $V$ admits a degree 1 ample line bundle $\cL = (\bC^*)^2 \times \bC/\tau \bZ^2$ defined by quotienting by
$$
(\tau n) \cdot (x_1,x_2 , v ) = (e^{2\pi i \sum_{k=1}^2 \tau_{1k}n_k}x_1, e^{2\pi i \sum_{k=1}^2 \tau_{2k} n_k }x_2,  e^{-\pi i(n^T \tau n)} 
\prod_{j=1}^2 x_j ^{-n_j}v)
$$
with section given by the theta function
$$
\vartheta(\tau,x) := \sum_{n\in  \bZ^2} e^{\pi i n^T \tau n} x_1^{n_1} x_2^{n_2}, \quad x=(x_1,x_2) \in (\mathbb{C}^*)^2.
$$ 
The 0-set of the theta function in $V$ is a genus 2 curve 
$$
H  :=\{ x \in V \mid \vartheta(\tau,x)=0\}=\Sigma_2 \subset V.
$$
Define $X:=\text{Bl}_{H\times \{0\}} V \times \bC$ with blow-up map $\pi: X \to V \times \bC$, exceptional divisor $p:E \to H \times \{0\}$, and inclusion $i: H  \times \{0\} \to V \times \bC$. To obtain the generalized SYZ mirror, we have an additional step of quotienting by $\tau \bZ^2$ after restricting $Y$ to a neighborhood where the group acts freely \cite[Eqs. (3.10), (3.24), (3.25)]{Ca20}. For small $\varepsilon <1$, 
\begin{equation}
\label{eq:polytope_def_Y} 
\begin{aligned}
\Delta_{\tilde{Y}} &:= \{(\xi_1,\xi_2,\eta) \in \bR^3 \mid \eta \geq \Trop(\vartheta)(\xi_1,\xi_2)\}\\
\Delta_Y &:= (\Delta_{\tilde{Y}})|_{|\chi^{(0,0,1)}| \leq \varepsilon}/\tau \bZ^2 \cong T^2 \times [0,\varepsilon']
\end{aligned}
\end{equation}
where $\chi^{(0,0,1)}$ is the complex affine toric coordinate corresponding to the vector in the $\eta$ direction, $(0,0,1)$. We will denote it by $v_0$ following \cite{Ca20,AAK}. The polytope $(\Delta_{\tilde{Y}})|_{|\chi^{(0,0,1)}| \leq \varepsilon}$ looks like a bowl tiled by infinitely many hexagons and thickened by an amount $\varepsilon'$ that depends on $\varepsilon$. After the quotient, the hexagon is identified to a $T^2$ and $\Delta_Y \cong T^2 \times [0,\varepsilon']$. 

Recall from toric geometry that each vertex $v$ corresponds to a chart $\Spec 
(\bC[y_1^v,y_2^v, y_3^v])\cong \bC^3$ in local inhomogeneous coordinates $y_1^v,y_2^v, y_3^v$, and edges define transition functions. The reason this polytope defines the SYZ mirror is because its vertices and edges match the SYZ charts and transition functions. In particular, 
$$
v_0=y_1^v y_2^v y_3^v.
$$ 
Then $(Y,-v_0)$ is the SYZ mirror to $(X,y)$ and the generalized SYZ mirror to $H$. The SYZ mirror to $X$ is $(Y^0,W_0)$ defined by 
\begin{equation}\label{eq:superpot_relns}
   Y^0 = Y \backslash v_0^{-1}(1), \quad W_0=T^\epsilon(v_0-1)
\end{equation}  
where $T$ is the Novikov parameter and $\epsilon >0$ is small (different from the $\varepsilon$ in the domain of definition in Equation \eqref{eq:polytope_def_Y}). The Novikov parameter parametrizes complex structures on $Y$ mirroring symplectic structures on $X$, but here we consider complex structures on $X$ mirroring symplectic structures on $Y$, which is parametrized by $\tau$. In general, an SYZ complex mirror is defined over a Novikov parameter $T$ so that structure maps converge in the Fukaya category of the original symplectic manifold. Here, structure maps converge over $\mathbb{C}$ and we do not consider the A-side of the genus 2 curve (which was considered in \cite{genus2}), so we may assume 
\begin{equation}
    T=e^{-2\pi}.
 \end{equation}

Here $W_0: Y^0 \to \bC^*$ is a $T^4$-fibration over $\bC^*$ with one singular fiber over $-T^\epsilon$. We use the symplectic form on $Y$ constructed in \cite[\textsection 3.5]{Ca20}, restricted to $Y^0$. The total space of $Y$ is smooth, and $\omega$ is modeled on the standard form for $\mathbb{C}^3$ near critical points of $v_0$ (the north and south pole of the banana manifold of three $\mathbb{P}^1$'s identified at their poles), while it is modeled on the standard product form for $\mathbb{CP}^2(3) \times \mathbb{D}$ away from the critical points.

\begin{remark}\label{rem:toric_coords}
    What was called $x,y,z$ in \cite[Figure 10, \textsection 3.3]{Ca20} is $y_1^{(-1,-1,0)}, y_2^{(-1,-1,0)}, y_3^{(-1,-1,0)}$ here. And $\tau$ in \cite{Ca20} has been renamed $t$ here, in order to reserve $\tau$ for the more standard notation of the complex modular parameter, here $\frac{i}{2\pi} \log t \left( \begin{matrix} 2 & 1\\ 1 & 2 \end{matrix} \right)$.
\end{remark}

To recap, here are the three forms of SYZ mirror symmetry in \cite{AAK} in the case of $V$ an abelian surface. 

\begin{theorem}[{\cite[Theorem 10.4]{AAK}}]\mbox{}

\begin{enumerate}
    \item $Y^0$ is SYZ mirror to $X^0$. 
    \item $(Y^0, W_0)$ is SYZ mirror to $X$. 
    \item $(Y,-v_0)$ is generalized SYZ mirror to $H$.
\end{enumerate}

\end{theorem}

HMS for (3) is examined in \cite{Ca20}. HMS for (2) is the subject of this paper.

\begin{remark} The symplectic mirror to the hypersurface of an abelian variety cannot be exact. Note that $W_0$, like $v_0$, is a symplectic fibration with fibers given by closed, compact tori with volume. Therefore, the symplectic form $\omega$ cannot be exact by Stokes' theorem. For if $[\omega]=0$, then so does $[\omega|_{W_0^{-1}(w)}]=0$ and thus $\int_{W_0^{-1}(w)} \omega^2=0$ where $W_0^{-1}(w) \cong T^{4}$ for $w \in \mathbb{C}\backslash\{0,-T^\epsilon\}$, contradiction.
\end{remark}

To state our main theorem involving the Fukaya category of $(Y^0,W_0)$, we first describe the B-side category.

{\bf Semi-orthogonal decomposition for blowups.} We adapt the diagram at the start of \cite[\textsection 4]{Or92} where their $\tilde{X} \to X$ is our blow up $\pi: \text{Bl}_{H  \times \{0\}} V \times \bC \to V \times \bC$, their $p: \tilde{Y} \to Y$ is the exceptional divisor which is our $\bP^1$-bundle $p: E \to H  \times \{0\}$, and $j:E \to X$ is inclusion. That is, the following commutes:
\begin{equation}\label{eq:blowup_diag}
\begin{tikzcd}
E \arrow[r, "j"] \arrow[d, "p"] & X=\text{Bl}_{H \times \{0\}} V \times \mathbb{C} \arrow[d, "\pi"] \\
H \times \{0\} \arrow[r, "i"] & V \times \mathbb{C}
\end{tikzcd}  
\end{equation}
Then \cite[Assertion 4.2(b), Theorem 4.3]{Or92} assert that there is a semi-orthogonal decomposition
\begin{equation}\label{eq:sod_Bside_no_wrap}
    D^b\Coh(X) = \langle j_*(p^* D^b\Coh(H  \times \{0\})\otimes \cO_p(-1)) , \pi^* D^b\Coh(V \times \bC) \rangle
\end{equation}
where $\cO_p(-1) = \cO_E(E)$. In the first factor, this means that a sheaf on $H $ is pulled back to the exceptional divisor and then extended by 0 to all of $X$. 

This result holds, roughly speaking, because of Beilinson's spectral sequence in \cite[Proposition 8.28]{Huy06} and the semi-orthogonality of $\cO_{\bP^n}$ with $\cO(1)_{\bP^n}$. Specifically, $H^0(\bP^n,\cO(1)) \cong \bC[x_0,\ldots,x_n]_1$, degree 1 homogeneous polynomials in $(n+1)$-variables, but $H^0(\bP^n, \cO(-1)) = 0$.\newline

{\bf Main Theorem.} First we set some notation on the B-side.

\begin{definition}[notation $D^b_\cL$, {\cite{Ca20}}]
Let $\pr_V: V \times \bC \to V$ and $\pr_\bC: V \times \bC \to \bC$ be the projections to each factor. We define the following subcategories of $D^b\text{Coh}$.

\begin{enumerate}
    \item $D^b_\cL\Coh(V)$: full subcategory of objects $\{\cL^{\otimes j}[k]\}_{j,k \in \bZ}$ 
     \item $D^b_\cL\Coh(V\times \bC)$: full subcategory of objects $\{(\pr_V^*\cL^{\otimes j}\otimes \pr_\bC^*\cO_\bC)[k] \}_{j,k \in \bZ}$ 
    \item $D^b_\cL\Coh(H)$: full subcategory of objects $\{\cL^{\otimes j}[k]|_H\}_{j,k \in \bZ}$
    \item $D^b_\cL\Coh(X)$: full subcategory of objects in $j_*(p^*D^b_{\cL} \Coh(H \times \{0\}) \otimes \cO_p(-1))$ and $ \pi^* D^b_{\cL} \Coh(V \times \bC)$
\end{enumerate}
\end{definition}

Now we describe the A-side category. On $(Y^0,W_0)$ we take the Fukaya-Seidel category $FS(Y^0,W_0)$, after Fukaya \cite{fooo1} for compact manifolds, adapted to exact Lefschetz fibration symplectic LG models in \cite{seidel}, and generalized to non-exact, non-Lefschetz LG models in \cite{ACLL_LG}. Then our main theorem is:

\begin{theorem} \label{thm}
There exist two full subcategories $\cA_L$ (Lagrangians $L$ fibered over U-shapes) and $\cA_K$ (Lagrangians $K$ fibered over cotangent fibers) of the Fukaya category $H^* FS(Y^0, W_0)$ which are equivalent to $D^b_\cL\Coh(H)$ and $D^b_\cL\Coh(V\times \bC)$, respectively. Furthermore, for the full subcategory $\cA_{L,K}:= \cA_L \cup \cA_K$, there is an equivalence  
\begin{equation}\label{eq:main_result}
    \cA_{L,K} \xrightarrow[]{\cong} D^b_{\cL} \Coh(X).
\end{equation}
In particular, the semi-orthogonality in $D^b_{\cL} Coh(X)$ is respected.
\end{theorem}

\begin{remark}
    Note that $D^b\text{Coh}_\mathcal{L}(V)$ is referred to as $\mathcal{B}_\mathcal{L}$ in \cite[\textsection 1.2.2]{ACLL_tori}.
\end{remark}

\begin{remark}
    We can generalize the definition of $X$ and $(Y^0,W_0)$ by taking $H$ to be a theta divisor in a principally polarized abelian variety $V_\tau$ of any dimension, for a generic parameter $\tau$ in the Siegel upper half space. We expect the results of this paper can then be generalized to HMS between $X$ and $(Y^0,W_0)$, assuming HMS for theta divisors. This will involve generalizing the symplectic form in \cite{Ca20}, including a B-field, equipping Lagrangians with connection 1-forms with curvature determined by the B-field, replacing $\frac{i}{2\pi} \log t \left( \begin{matrix} 2 & 1\\ 1 & 2 \end{matrix} \right)$ with a general $\tau$, and generalizing calculations from $g=2$ here for an abelian surface to $g$ for an abelian variety of arbitrary dimension.
\end{remark}

\begin{remark}\label{rem:loc_sys_assump}
    In this paper, we assume trivial local systems on the Lagrangians as in \cite{Ca20}. Global HMS for principally polarized abelian varieties of any dimension and with the most general connection 1-forms on Lagrangians is proved on expected generating sets in \cite{ACLL_tori}. Here and in \cite{Ca20}, we consider fiber Lagrangians $\ell_{k,[0]}$ and $\omega_B=0$ in the notation of \cite[\textsection 1.2.3 and Equation (2.36)]{ACLL_tori}.
\end{remark}

\begin{remark}
    The Lagrangians in $\cA_L$ and $\cA_K$ do not bound discs. To define $FS(Y^0,W_0)$ one must account for the Lagrangians that do, with bounding cochains as in \cite[\textsection 3]{fooo1} for compact symplectic manifolds without superpotential. For a suitable definition of $FS(Y^0,W_0)$ with bounding cochains, upgrading to an $A_\infty$-result, and for a suitable adaptation of core HMS \cite{PS23}, we expect this theorem can be upgraded to the originally conjectured equivalence in Equation \eqref{eq:hms} which would be $D^b\text{Coh}(X) \cong D^\pi FS(Y^0,W_0)$.
\end{remark}

\begin{remark} 
\label{rem:glue}
We directly construct Lagrangians and calculate their morphisms in $(Y^0, W_0)$. Another approach could be to combine the left diagram of Figure \ref{fig:diagrams} for $H^0 FS(Y,-v_0)$ from \cite{Ca20}, prove a similar diagram for $H^0 FS(V^\vee \times \bC^*_{y^\vee}, y^\vee)$ as in the right diagram of Figure \ref{fig:diagrams} where $V^\vee$ is the SYZ mirror of $V$, and then combine them as suggested in \cite[Remark 1.5]{Aur18} and \cite[\textsection 3.3: Gluing]{sylvan_orlov}.

\begin{figure}[h]
\begin{multicols}{2}
\begin{tikzcd}
D^b_{\mathcal{L}} \mathrm{Coh}(V) \arrow[r, "\iota^*"] \arrow[d, hook, "\text{HMS on } V"'] & D^b_{\mathcal{L}} \mathrm{Coh}(H) \arrow[d, hook, "\text{HMS on } H=\Sigma_2"'] \\
H^0\mathrm{Fuk}(V^\vee) \arrow[r, "\cup"] & H^0 {FS}(Y,-v_0)
\end{tikzcd}

\vspace{1cm}

\begin{tikzcd}
D^b_{\mathcal{L}} \mathrm{Coh}(V) \arrow[r, "\mathrm{pr}_V^*"] \arrow[d, hook, "\text{HMS on } V"'] & D^b_{\mathcal{L}} \mathrm{Coh}(V \times \mathbb{C}) \arrow[d, hook] \\
H^0\mathrm{Fuk}(V^\vee) \arrow[r,"\cup"] & H^0 {FS}(V^\vee \times \mathbb{C}^*_{y^\vee}, y^\vee)
\end{tikzcd}
\end{multicols}
\caption{HMS squares for two components of Remark \ref{rem:glue}}
\label{fig:diagrams}
\end{figure}

\end{remark}

{\bf Outline of the paper.} In Section \ref{sec:Bside} we describe the B-side and the semi-orthogonality of $D^b \text{Coh}(X)$. In Section \ref{sec:Aside} we describe the A-side. In Subsection \ref{sec:objs} we define the mirror Lagrangian objects.  In Subsection \ref{sec:mors} we define the morphisms in the Fukaya subcategories of $(Y^0, W_0)$ by showing the continuation maps are well-defined. We define how Lagrangians wrap in \textsection \ref{subsec:wrap_conv} and outline the rest of Subsection \ref{sec:mors} with an overview of how to pass from the cochain complex to the localized category in \textsection \ref{sec:alg}. Then in \textsection \ref{sec:qisos} we define the quasi-units we localize at, and in \textsection \ref{sec:contn} we define their inverses, the continuation maps. Finally in \textsection \ref{sec:mors_defn} we define the morphisms by taking a homotopy colimit over multiplication by quasi-units, using a model from \cite{AA24} that amounts to wrapping one way on the input Lagrangian and in the opposite direction on the output, in Equation \eqref{eq:we_can_wrap}. In Subsection \ref{sec:grading} we define the Lagrangian gradings, which are part of the data of a D-brane in the Fukaya category. Then in Subsection \ref{sec:calcs} we calculate morphisms in Lemmas \ref{lem: F1_mors} (within $\mathcal{A}_L$), \ref{lem:relate_to_Ca} ($\mathcal{A}_L$ to $\mathcal{A}_K$), \ref{lem:orthoA} ($\mathcal{A}_K$ to $\mathcal{A}_L$), and \ref{lem:morsF2} (within $\mathcal{A}_K$). In Subsection \ref{sec:comp} we calculate the product structure on both sides. This requires an analysis of gradings to determine what intersection points appear in the product, which is Lemma \ref{lem:compn_final}. We then prove the main theorem, Theorem \ref{thm}, in Section \ref{sec:proof}.

{\bf Acknowledgements.} We thank Denis Auroux for helpful comments in formulating the direction of this project and for explaining his joint work \cite{AA24}. We thank Hiro Lee Tanaka, Sheel Ganatra, and Zack Sylvan for helpful discussions on wrapping in Fukaya categories, Christian Schnell and Alexander Kuznetsov for helpful discussions on semi-orthogonal decompositions of $D^b\text{Coh}$, and Katrin Wehrheim for helpful discussions on moduli spaces. Thanks to Haniya Azam, Hiro Lee Tanaka, and Sheel Ganatra for helpful comments on an earlier draft.

\section{B-side semi-orthogonal decomposition}\label{sec:Bside}
Let $\cL^j \boxtimes \cO_\bC:= \pr_V^* \cL^j \otimes \pr_\bC^* \cO_\bC$. Morphisms in the bounded derived category of coherent sheaves are Ext groups. Recall Equation \eqref{eq:sod_Bside_no_wrap} that 
$$
D^b\Coh(X) = \langle j_*(p^* D^b\Coh(H  \times \{0\})\otimes \cO_p(-1)) , \pi^* D^b\Coh(V \times \bC) \rangle
$$
where
\begin{center}
\begin{tikzcd}
E \arrow[r, "j"] \arrow[d, "p"] & X=\text{Bl}_{H \times \{0\}} V \times \mathbb{C} \arrow[d, "\pi"] \\
H \times \{0\} \arrow[r, "i"] & V \times \mathbb{C}
\end{tikzcd}
\end{center}
commutes.
\subsection{Morphisms from $V \times \mathbb{C}$ to $H$}
By semi-orthogonality, we know that for $k_1, k_2 \in \mathbb{Z}$,
\begin{equation}\label{eq:orthoB}
\Ext(\pi^*(\cL^{k_1} \boxtimes \cO_\bC), j_*(p^* i^*(\cL^{k_2} \boxtimes \cO_\bC) \otimes \cO_p(-1)))=0.
\end{equation}

\subsection{Morphisms from $H$ to $V \times \mathbb{C}$}

By the adjunction formula, $K_X = \pi^*K_{V \times \bC} \otimes \cO(E)$. As in \cite[Proposition 2.5.5]{Huy05}, we define the relative canonical bundle as follows,
$$
K_{E/X}:=K_E \otimes j^* K^{-1}_{X}= j^*(K_X \otimes \cO(E))  \otimes j^* K^{-1}_{X}= \cO_E(E)
$$
where $\cO_E(E)=j^*\cO(E) = \cO_p(-1)$. Furthermore, $j^!$ admits a left adjoint $j_*$ for $j:E \to X$ where 
$$
j^!(\cdot) = j^*(\cdot) \otimes K_{E/X}[\dim_\bC E - \dim_\bC X]=j^*(\cdot) \otimes \cO_E(E)[-1].
$$
Thus
\begin{equation}\label{eq:1stpart-nonorthoBside}
    \begin{aligned}
    & \Ext_X(j_*(p^*i^*(\cL^{k_2} \boxtimes \cO_\bC)\otimes \cO_E(E)),\pi^*(\cL^{k_1} \boxtimes \cO_\bC))\\
    &=\Ext_E(p^*i^*(\cL^{k_2} \boxtimes \cO_\bC) \otimes \cO_E(E), j^!\pi^*(\cL^{k_1} \boxtimes \cO_\bC))\\
    & = \Ext_E(p^*i^*(\cL^{k_2} \boxtimes \cO_\bC) \otimes \cO_E(E), j^*\pi^*(\cL^{k_1} \boxtimes \cO_\bC)\otimes \cO_E(E)[-1]).
    \end{aligned}
\end{equation}
Note that $\otimes \cO_E(E)$ is an auto-equivalence of $D^b\text{Coh}(X)$, thus its contribution cancels. Also, the statement that Equation \eqref{eq:blowup_diag} commutes is that $\pi \circ j = i \circ p$. Thirdly, $p: E \to H  \times \{0\}$ is surjective and $p^*$ is fully faithful, an isomorphism on hom sets. Hence we can simplify to obtain
\begin{equation}\label{eq:nonorthoB}
    \begin{aligned}
    & \Ext_E(p^*i^*(\cL^{k_2} \boxtimes \cO_\bC), j^*\pi^*(\cL^{k_1} \boxtimes \cO_\bC)[-1])\\
    &=\Ext_E(p^*i^*(\cL^{k_2} \boxtimes \cO_\bC), p^*i^*(\cL^{k_1} \boxtimes \cO_\bC)[-1])\\
    & \cong \Ext_H(i^*(\cL^{k_2} \boxtimes \cO_\bC),i^*(\cL^{k_1} \boxtimes \cO_\bC)[-1])\\
    & \cong \Ext_H(\cL^{k_2}|_H, \cL^{k_1}|_H)[-1]
    \end{aligned}
\end{equation}
where in the last step we restrict to $H \times \{0\}$. Now we show this matches with the symplectic side, including the shift $[-1]$ which corresponds to a choice of Lagrangian grading on the mirror.

\section{A-side semi-orthogonal decomposition}\label{sec:Aside}

The Fukaya category roughly consists of Lagrangian submanifolds as objects and cochain complexes generated by their intersection points as morphisms, graded by a Maslov index. The differential counts bigons between input and output intersection points, weighted by area (and also holonomy around the boundary of the disc, which we do not consider here, see Remark \ref{rem:loc_sys_assump}). Hamiltonian-equivalent Lagrangians should be isomorphic in the derived Fukaya category. This has led to increasingly more general definitions of the Fukaya category, as more mirrors are constructed.

In the Fukaya category of a closed symplectic manifold (compact with no boundary), no wrapping is needed for closed Lagrangians. This is the original $A_\infty$-category \cite{fooo1}. For a non-closed symplectic manifold, wrapping is needed for non-closed Lagrangians. So non-compact manifolds and those with boundary have wrapping. For example, the Fukaya category in \cite{AS10, GPS20}, which wraps with stops, applies to non-compact exact Liouville manifolds. \cite[\textsection 3.3]{GPS24} describe wrapping at a puncture as a correspondence between stop removal and localization at continuation maps.

Non-compact manifolds that appear as SYZ mirrors in many known examples are fibrations over  $\mathbb{C}$ or over a subset with boundary; these are the LG models. Seidel \cite{seidel} pioneered the first constructions of their Fukaya categories, specifically for non-compact exact Lefschetz fibrations. Therefore more generally for a symplectic fibration LG model, we refer to its Fukaya category as the Fukaya-Seidel category. To equip an LG model with a Fukaya-Seidel category, the base and the fiber determine the wrapping needed. For example, fibers given by closed manifolds have no wrapping, $(\mathbb{C}^*)^n$-fibers have wrapping, the base $\mathbb{C}$ of an LG model wraps but not fully due to one stop, and punctures in the base or fiber introduce full wrapping around the puncture. Once the wrapping behavior is determined, the Fukaya-Seidel category of the symplectic fibration LG model can be constructed by defining the continuation maps and showing they are well-defined.

In \cite{AA24} a fiberwise wrapped Fukaya-Seidel category is defined for non-compact non-exact LG models. In our LG model $(Y^0,W_0)$, there is one puncture in the base and generic fibers are closed $T^4$. Since the fibration is not Lefschetz, we avoid the singular fiber. So the projection under $W_0$ of Lagrangians in $\cA_L$ are U-shapes around the critical value. And in $\mathcal{A}_K$ they project under $W_0$ to a curve that can fully wrap in the base around the puncture. All projected Lagrangians still wrap partially at infinity due to the one stop for LG models 
\cite[Appendix A]{ushape}, \cite[\textsection 1.2, Example 1]{GPS20}, \cite{AA24}. Our continuation maps are defined in Definition \ref{def:contn_maps}.

Note that the complex dimension of the generalized SYZ mirror $(Y,-v_0)$ to $\Sigma_2$ is two larger than that of $\Sigma_2$; if one takes $\text{Crit}(v_0)$ as the mirror, which is complex dimension 1, then a Fukaya category on that has been constructed in \cite{AEK24}.

\subsection{Objects in $\mathcal{A}_L$, $\mathcal{A}_K$}\label{sec:objs}

The base of $W_0$ is $\bC^*$ which is topologically a cylinder, and a generic fiber is $T^4$. So there are two pictorial representations for the base: $\bC^*$ and the cylinder.  Curves drawn in a plane represent projections of Lagrangians under $W_0$, while curves drawn on a cylinder represent projections of Lagrangians under $W_0$ composed with $\bC^* \to \bR \times S^1$ given by $e^{2\pi (\mu+i \theta)} \mapsto (\mu,\theta)$. Lagrangians are ``thimble-like" in the sense of \cite[Definition 8.19]{GPS24}. 

\begin{figure}[h]
    \centering
   {\fontsize{59pt}{61pt}\selectfont
\resizebox{80mm}{!}{
\begingroup%
  \makeatletter%
  \providecommand\color[2][]{%
    \errmessage{(Inkscape) Color is used for the text in Inkscape, but the package 'color.sty' is not loaded}%
    \renewcommand\color[2][]{}%
  }%
  \providecommand\transparent[1]{%
    \errmessage{(Inkscape) Transparency is used (non-zero) for the text in Inkscape, but the package 'transparent.sty' is not loaded}%
    \renewcommand\transparent[1]{}%
  }%
  \providecommand\rotatebox[2]{#2}%
  \newcommand*\fsize{\dimexpr\f@size pt\relax}%
  \newcommand*\lineheight[1]{\fontsize{\fsize}{#1\fsize}\selectfont}%
  \ifx\svgwidth\undefined%
    \setlength{\unitlength}{637.18484233bp}%
    \ifx\svgscale\undefined%
      \relax%
    \else%
      \setlength{\unitlength}{\unitlength * \real{\svgscale}}%
    \fi%
  \else%
    \setlength{\unitlength}{\svgwidth}%
  \fi%
  \global\let\svgwidth\undefined%
  \global\let\svgscale\undefined%
  \makeatother%
  \begin{picture}(1,0.35795928)%
    \lineheight{1}%
    \setlength\tabcolsep{0pt}%
    \put(0,0){\includegraphics[width=\unitlength,page=1]{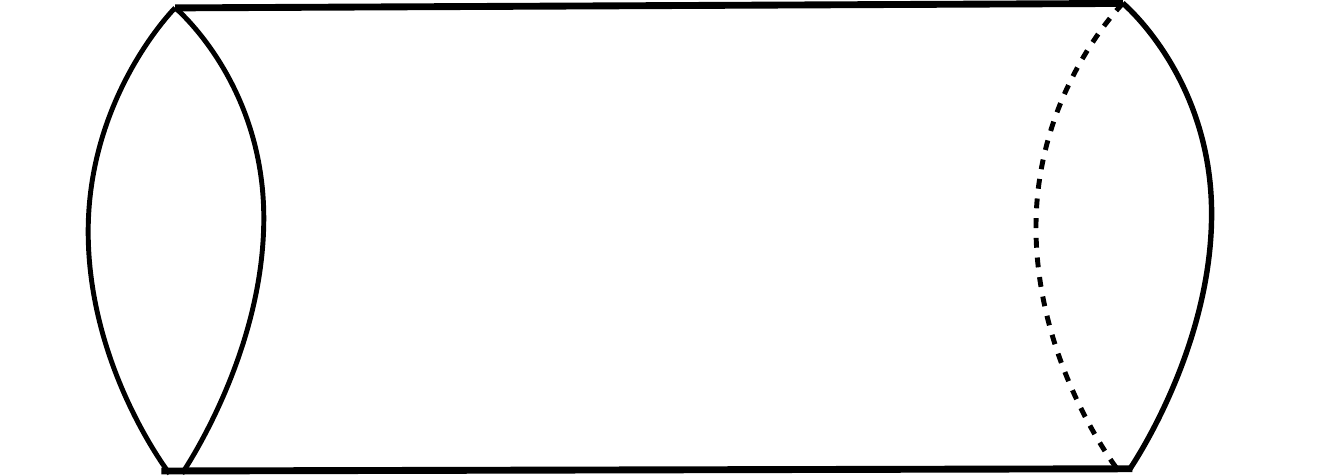}}%
    \put(-0.00213801,0.16324769){\color[rgb]{0,0,0}\makebox(0,0)[lt]{\lineheight{1.25}\smash{\begin{tabular}[t]{l}$0$\end{tabular}}}}%
    \put(0.92051293,0.1678199){\color[rgb]{0,0,0}\makebox(0,0)[lt]{\lineheight{1.25}\smash{\begin{tabular}[t]{l}$\infty$\end{tabular}}}}%
    \put(0.49010153,0.18178587){\color[rgb]{0,0,0}\makebox(0,0)[lt]{\lineheight{1.25}\smash{\begin{tabular}[t]{l}x\end{tabular}}}}%
    \put(0,0){\includegraphics[width=\unitlength,page=2]{cylinder_latex.pdf}}%
    \put(0.63672148,0.22560327){\color[rgb]{0,0,0}\makebox(0,0)[lt]{\lineheight{1.25}\smash{\begin{tabular}[t]{l}$\frac{1}{2\pi}\log(W_0(L_i))$\end{tabular}}}}%
    \put(0,0){\includegraphics[width=\unitlength,page=3]{cylinder_latex.pdf}}%
    \put(0.39235491,0.08740976){\color[rgb]{0,0,0}\makebox(0,0)[lt]{\lineheight{1.25}\smash{\begin{tabular}[t]{l}$\frac{1}{2\pi}\log(W_0(K_j))$\end{tabular}}}}%
  \end{picture}%
\endgroup%
}
    \caption{Projections of the two types of Lagrangians to the cylinder model for the base of $W_0:Y^0 \to \bC$}
    \label{fig:base_cylinder}
    }
\end{figure}

\begin{definition}[Parallel transport, {\cite[Equation (6.2)]{symp_intro}}]\label{def:par_transport}
Let $\Phi_{\gamma(0) \to \gamma(t)}: W_0^{-1}(\gamma(0)) \to W_0^{-1}(\gamma(t))$ denote parallel transport in $(Y^0,W_0)$. It is taken with respect to $\omega|_{Y^0}$, for $\omega$ on $Y$ defined in \cite[\textsection 3.5]{Ca20}, over an embedded path $\gamma: [0,t] \to \bC^*$ avoiding the critical value $\text{critv}(W_0) = -T^\epsilon$.
\end{definition}

We define two subcategories $\mathcal{A}_L, \mathcal{A}_K$ of the Fukaya-Seidel category of $(Y^0, W_0)$. Define a point 
\begin{equation}\label{eq:starting_pt}
b_S=-S \in \bC^*, \quad S \in (T^\epsilon, T^\varepsilon).
\end{equation} 
Let $\gamma: \bR \to \bC^* \backslash\{-T^\epsilon\}$ be an embedded path as above, avoiding $\text{critv}(W_0)$ and passing through $\gamma(0):=b_S$. We parallel transport slope $k$ Lagrangians in the $b_S$-fiber over arcs in the base to obtain 3-dimensional Lagrangians in the 6-dimensional total space of $Y^0$.

\begin{definition}[Fiber Lagrangians]
    Define slope $k$ Lagrangians in the fiber $W_0^{-1}(b_S)\cong T^4$ by
    $$
    \ell_k := \{(\xi_1,\xi_2,\theta_1,\theta_2)\in T^4 \mid (\theta_1,\theta_2) = -k\left( \begin{matrix} 2 & 1 \\ 1 & 2 \end{matrix} \right)^{-1}\left( \begin{matrix} \xi_1\\ \xi_2 \end{matrix} \right) \}
    $$
    where $(\theta_1,\theta_2) = \frac{1}{2\pi}(\arg(t_1),\arg(t_2))$ for $(t_1,t_2) \in (\bC^*)^2$ the affine coordinates corresponding to monomials $\chi^{(1,0,0)}$ and $\chi^{(0,1,0)}$ on the polytope $\Delta_Y$ of Equation \eqref{eq:polytope_def_Y}, and $(\xi_1,\xi_2)$ are defined in Equation \eqref{eq:mom_map_coords}.
\end{definition}

\begin{definition}[Objects of $\mathcal{A}_L$]\label{def:U_shape} 
Objects of $\mathcal{A}_L$ are $\{L_j[k]\}_{j,k \in \mathbb{Z}}$ defined as 
$$
L_j := \bigcup_{-\infty < t < \infty} \Phi_{\gamma_{L_j}(0) \to \gamma_{L_j}(t)}(\ell_j)
$$
where 
$$
\gamma_{L_j}(t) =r(t) e^{2\pi \mathbf{i} \theta(t)}: (-\infty,\infty) \to \bC^*
$$  
such that $r(t), \theta(t): \bR \to \bR$ are smooth with
\begin{equation}\label{eq:U-shape_base}
(r(0),\theta(0))=(S, -\frac{1}{2}) , \qquad r(t) \propto |t|, \dot\theta(t) = 0 \;\; \forall t \notin r^{-1}([R_2,R_4]). 
    \end{equation}
    Let $\gamma_{L_j}(t)$ initially near $t=0$ stay within rings $R_2$ and $R_4$ in Figure \ref{fig:concentric_rings}, and take a U-shape around the critical value at $\times = -T^\epsilon$ and below the puncture at 0. In particular, $\gamma_{L_j}(0)=-S$. When $r(t)$ reaches $R_4$, the ends are at constant angles away from $-\pi$
    \begin{equation}\label{eq:radial_ends}
     \theta(t) = \begin{cases}
            \theta_{h_1}(L_j)& \forall t \gg 0\\
            \theta_{h_2}(L_j) & \forall t \ll 0 
            \end{cases}
    \end{equation}
    and strictly increasing radial direction. See $W_0(L_{j_1})$ in Figure \ref{fig:1cot_2U}, where the $\theta_{h_1}(L_j)$ end corresponds to the lower end of the U-shape. 

    \end{definition}
    
    \begin{remark}
    U-shaped Lagrangians have appeared in the literature such as \cite[\textsection A.2]{ushape}, \cite{ACLL_LG}, and \cite{AA24}. Here we use the model that ends of the U-shape project to angular rays away from 0 in $\mathbb{C}^*$. 
    \end{remark}

    \begin{definition}[Objects of $\mathcal{A}_K$]\label{def:F2objs}

Objects of $\mathcal{A}_K$ are $\{K_j[k]\}_{j,k \in \mathbb{Z}}$ defined as 
$$
K_j := \bigcup_{1 < t < \infty} \Phi_{\delta_{K_j}(0) \to \delta_{K_j}(t)}(\ell_j)
$$
where 
$$
\delta_{K_j}(t) =r(t) e^{2\pi \mathbf{i} \theta(t)}: [0,\infty) \to \bC
$$  
such that $r(t), \theta(t): [0,\infty) \to \bR$ are smooth with
\begin{equation}\label{eq:arc-shape_base}
(r(0),\theta(0))=(S, -\frac{1}{2}), \quad r(1)=0,  \quad (r(t), \theta(t)) =(t-1, \theta_h(K_j)) \;\; \forall t > 1
\end{equation}
    so that $\delta_{K_j}$ is an arc from $-S$ to 0 avoiding the critical value at $-T^\epsilon$ by going above it, and is then a single arc at fixed angle $\theta_h(K_j)$. To be the correct mirror, we take the arc to lie above any U-shape ends in the $\mathbb{C}^*$ model, so it lies counterclockwise from the U-shape ends at angles $\theta_{h_1}(L_{j'}),\theta_{h_2}(L_{j'})$  for all $j' \in \mathbb{Z}$, and clockwise from the ray at $-\pi$. Parallel transport starts from the fiber over $\delta_{K_j}(0)=-S$, but the Lagrangian is defined only over the part of the arc for $t>1$ because in $Y^0$ we puncture the base at 0, which is where the curve passes through at $t=1$. See $W_0(K_j)$ in Figure \ref{fig:plane_KL}. In the cylinder model, $W_0(K_j)$ is a cotangent fiber of $T^*S^1$ at height $\theta = \theta_h(K_j)$, see $W_0(K_j)$ in Figure \ref{fig:cylinder_KL}.
\end{definition}
    
We will define the gradings on the Lagrangians in Subsection \ref{sec:grading}.

\begin{remark}
We fix the arc to go above the U-shaped ends when viewed in the base of $W_0$. Had we taken it to go below, it would swap the order of the semi-orthogonal decomposition and be mirror to a different embedding of the components on the B-side too.
\end{remark}

\subsection{Definition of morphisms}\label{sec:mors}

In this section we define the morphisms. We give some motivation for the process. The morphisms are given by the wrapped Lagrangian intersection Floer cohomology theory, described on the chain-level in \textsection \ref{subsec:wrap_conv} as intersections of Lagrangians and denoted by $CF(L_0,L_1)$. Furthermore, Hamiltonian-isotopic Lagrangians should be quasi-isomorphic, meaning isomorphic when we derive the category. There is a convenient tool from homotopy theory that does the job for non-compact Lagrangians, called a homotopy colimit \cite{AS10,AA24}. As explained in the introduction to \cite{Hi03}, a goal of homotopy theory is to identify homotopic maps, but a homotopy on topological spaces may not make sense in other categories. To identify homotopic morphisms $f$ and $g$, we define a set called the \emph{weak equivalences} to be where $f-g$ lands in, and then localize the category at weak equivalences. To take a homotopy colimit of a diagram one needs a model category: a category with three sets of morphisms called weak equivalences, fibrations, and cofibrations. Quillen axiomatized such a category as a place where one can do homotopy theory, calling them ``a category of models
for a homotopy theory" \cite{Qu67}. Note that more modern approaches in homotopy theory use $\infty$-categories instead of model categories. In the wrapped Fukaya category, weak equivalences will be multiplication by quasi-units. Furthermore, we will take the images of objects in the Fukaya category under the Yoneda embedding, and work inside the resulting triangulated category. 

We motivate how to glue together equivalent objects, that is, how to identify homotopic morphisms. An example \cite[\textsection 1.5]{We94} is that of making a simplicial complex $X$ contractible by gluing it to a point. Taking its cone $CX$ then $C_j(CX) = C_{j}(X) \oplus C_{j-1}(X),\; j \geq 1$ from simplices fully in the base and those through the cone vertex $s$. Checking the differential and quotienting by $C_0(s)$ generated by the cone vertex, this is just $C_j(CX)/C_j(s) = \text{Cone}_j(\mathbf{1}_\bullet: C_\bullet(X) \to C_\bullet(X))$. So the notion of topological and algebraic cone coincide.

In practice when computing morphisms of Lagrangians in a Landau-Ginzburg model, the informal way of doing so is to ``push the ends of the input Lagrangian, when projected to the base, to lie above the ends of the output Lagrangian." In other words, if $L_0^+$ is $L_0$ with the ends pushed up, we would like a morphism $L_0 \to L_1$ to be a \emph{roof} 
\begin{tikzcd}
    & L_0^+ \arrow[dl, "e"'] \arrow[dr, "p"] & \\
    L_0 & & L_1
\end{tikzcd}
where $e \in CF(L_0^+,L_0)$ is a quasi-unit (defined in Equation \eqref{eq:quasi-unit}) and $p \in L_0^+ \cap L_1$ is an intersection point, a morphism of $CF(L_0^+,L_1)$. This roof is the standard construction of a left fraction \cite[Prototype 10.3.5]{We94}, a morphism in a category localized at $e$. To compose roofs would involve a pullback along $L_0^+ \rightarrow L_1 \leftarrow L_1^+$ \cite[Definition 10.3.4, Ore condition]{We94}. However, it is not evident what this pullback would be in practice. 

We instead take a homotopy colimit over Yoneda modules $\mathcal{Y}_{L^k}=\cO(-,L^k)$ in the abelian category of $\mathcal{O}$-modules for a category $\mathcal{O}$ directed by amount wrapped, defined in Equation \eqref{eq:dir_cat}. The objects $L^k$ are indexed by wrapping an amount $k$ defined in Equation \eqref{eq:Lag_k_wrapped}. The homotopy colimit identifies objects that have a homotopy between them, preserving weak equivalences which are when a map of spaces induces an isomorphism on homotopy groups (so a particular inverse isn't needed, as for a strong equivalence). This is similar to how a colimit glues along connecting maps. However we need the homotopy version, for example a limit such as the pullback of $0 \hookrightarrow [0,1] \hookleftarrow 1$ is empty even though $0,1,[0,1]$ are all homotopy equivalent, while the homotopy limit would be $\{\gamma: [0,1] \to [0,1] \mid \gamma(0)=0, \gamma(1)=1\}$, \cite[Exercise 6.5.4]{Ri14}. Weak equivalences here are multiplcation by quasi-units, and we localize at quasi-units in order to identify non-compact Lagrangians equivalent under wrapping in the derived Fukaya-Seidel category.

We follow \cite[\textsection 3.5]{AA24}, where they use standard constructions similar to \cite[\textsection (3g)]{AS10}, \cite[\textsection 3.1.3]{GPS20}. These refer to \cite{LO06} for  localization of $A_\infty$-categories, a process that generalizes localization of categories \cite[\textsection 10.3]{We94}, \cite[\textsection 1]{Ne01}, \cite[Example 11.5.11]{Ri14}. 

\subsubsection{Wrapping conventions and the cochain complex}\label{subsec:wrap_conv}

For two transversely intersecting Lagrangians $L_0,L_1$, the Floer cochain complex without local systems is
\begin{equation}\label{eq:chain_level_mor}
    CF^m(L_0,L_1) = \bigoplus_{\substack{p \in L_0 \pitchfork L_1\\ \deg(p)=m}} \bC \cdot p.
\end{equation}
If they are not transverse, we perturb by positive and negative isotopies. We describe what we mean by positive. Let $L$ be a Lagrangian of the form $L_j$ or $K_j$. Recall that a Lagrangian isotopy is a homotopy 
   $$
    \phi: L \times [0,1] \to Y^0
    $$
which is a smooth embedding. We have a horizontal distribution, where $Y^0_b$ denotes the fiber $W_0^{-1}(b)$, given by
$$
\Hor_p=\left(\ker d_pW_0 \right)^{\perp_\omega}=(T_pY^0_{W_0(p)})^{\perp_\omega}\subset T_pY^0
$$
for the symplectic form $\omega$ because $W_0$ is a symplectic fibration, that is, $\omega$ restricts to a symplectic form on smooth fibers. Isotopies in a Landau-Ginzburg model \cite[\textsection 2]{ACLL_LG} are constructed by parallel transport over homotopies $\delta(s,t)$ of curves in the base $\bC$.

Let $\rho^t$ be the flow of the horizontal lift of the Hamiltonian vector field for the Hamiltonian $H$ defined as follows. Let $w$ be the complex coordinate on the base $\bC^*$ of $W_0$. There are several commonly used real coordinate systems on the base. We express them here for completeness. We have 
\begin{equation}
    w=w_x+iw_y=re^{2\pi i\theta}=e^{2\pi(\mu+i \theta)}
\end{equation}  
where the last expression is in action-angle coordinates $(\mu,\theta)$ for $\mu(r) = \frac{1}{2\pi} \log r$. By \cite[Remark 3.2]{ACLL_LG}, the standard form $\omega_\bC$ is related to $\omega$ by
\begin{equation}
({W_0}^*\omega_\bC)|_{\text{Hor}} = |dW_0|^2_{J,\omega} \omega|_{\text{Hor}}.
\end{equation}
Here on the base we have $\mathbb{C}^*$ instead of $\mathbb{C}$, so for $w$ the coordinate on $\mathbb{C}^*$ and under the map $\exp(2\pi \cdot): \mathbb{C} \to \mathbb{C}^*$ the standard form on $\mathbb{C}$ becomes the following on $\mathbb{C}^*$  
\begin{equation}
    \omega_{\bC^*}=\frac{i}{2} d \left(\frac{1}{2\pi}\log w\right) \wedge d\left(\frac{1}{2\pi}\log \overline{w}\right) =\frac{i}{2} \frac{dw \wedge d \overline{w}}{(2\pi)^2|w|^2} =d \mu \wedge d \theta.
\end{equation}

For $\mu \neq 0$ and $\chi(\theta):[-\pi,\pi) \to [0,1]$ a bump function supported away from $-\pi$, consider concentric rings in the base around 0 defined by $0<R_1 < R_2<T^\epsilon < S< R_3 < R_4< T^\varepsilon$ as in Figure \ref{fig:concentric_rings}. 
\begin{figure}
    \centering
    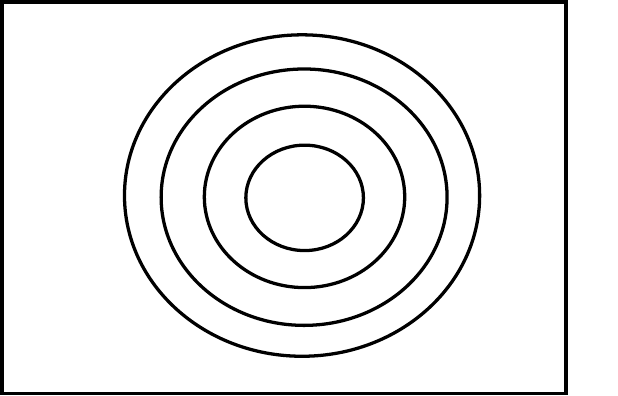
    \caption{$Y^0$ is a fibration over the $T^\varepsilon$ disc about 0. U-shape curves $W_0(L_j)$ pass through $b_S=-S$. Arcs $W_0(K_j)$ start from, but don't include, 0. The singular fiber of $W_0$ is above $-T^\epsilon$ marked by an $\times$.}
    \label{fig:concentric_rings}
\end{figure}
Define
\begin{equation}\label{eq:Hamilt}
   H(w) :=
\begin{cases}
\chi(\theta) \cdot \mu &  R_4 < |w| < T^\varepsilon\\
\text{interpolate} & |w| \in [R_3,R_4]\\
0 & R_2< |w| < R_3 \\
\text{interpolate} & |w| \in [R_1, R_2]\\
\mu & |w| < R_1
\end{cases}
\end{equation}
where recall $|W_0|: Y_0 \to [0,T^\varepsilon]$ because $W_0$ is a fibration only over a small disc. So within this domain but outside of the annulus 
\begin{equation}\label{eq:annulus}
    A=\{w \in \bC \mid R_1 < |w| < R_4 \}
\end{equation}
Lagrangians project under $W_0$ to arcs which wrap slightly for large $|w|$ and fully around 0. Within the smaller annulus $A=\{w \in \bC \mid R_2 < |w| < R_3 \}$ there is no movement, including at $-S$ where Lagrangians are defined as parallel transport from in Equation \eqref{eq:starting_pt}.
Define
\begin{equation}\label{eq:mov_Lag_isotopy}
\rho^t: Y^0 \to Y^0
\end{equation}
to be the flow of a vector field $\xi^\#$ which is the horizontal lift of $\xi$ defined by $\iota_{\xi}\omega_{\bC^*} =-dH$. (Note that $\xi$ is sometimes called $X_H$ in the literature, but in \cite{AA24} $X_H$ is used for a wrapping Hamiltonian $H$ in the fiber while $\xi$ realizes the wrapping in the base. Since we are wrapping in the base and not the fiber, to be consistent with them, we use $\xi$.) 

The relation between positively or negatively wrapping with respect to $\lambda:=\mu d \theta$ (where $\omega_{\bC^*} = d \lambda$) and counterclockwise/clockwise wrapping around 0 and infinity is as follows.
\begin{enumerate}
    \item Near $\infty$, by which we mean in the region $R_4 < r < T^\varepsilon$ near the boundary of $W_0(Y^0)$,
    \begin{equation}
        \iota_{-\chi' \mu \dd_\mu + \chi \dd_\theta} (d\mu \wedge d\theta) = -dH=-d(\chi(\theta) \mu)
    \end{equation}
    so on $\xi|_{\{R_4 < r < T^\varepsilon\}}=-\chi' \mu \dd_\mu + \chi \dd_\theta$ we find that
    \begin{equation}\label{eq:pos_wrap_partial}
        \lambda(-\chi' \mu \dd_\mu + \chi \dd_\theta) = \mu d \theta(-\chi' \mu \dd_\mu + \chi \dd_\theta) = \chi \mu >0
    \end{equation}
    and wrapping is positive and counterclockwise.
    \item Around 0 we flow clockwise. This is because under the conformal coordinate change $w \mapsto \tilde w:=w^{-1}$ on $\mathbb{C}^*$, then $(\mu,\theta)\mapsto (\tilde \mu, \tilde \theta) = (-\mu,-\theta)$ and
    \begin{equation}
    \omega_{\bC^*}=d\tilde \mu \wedge d \tilde \theta.
    \end{equation}
    Thus for $\tilde r=r^{-1} >1/R_1$
    \begin{equation}
    -dH = d \tilde \mu =  \iota_{-\dd/\dd \tilde \theta} d\tilde \mu \wedge d \tilde \theta
    \end{equation}
    while 
    \begin{equation}\label{eq:pos_wrap_full}
    \lambda\left(-\frac{\dd}{\dd \tilde \theta}\right) = \tilde \mu d \tilde \theta\left(- \frac{\dd}{\dd \tilde \theta}\right)= -\tilde \mu<0
    \end{equation}
    so we wrap negatively/counterclockwise in $\tilde \theta$. Inverting back to $w$, we wrap negatively/clockwise in $\theta$.  
\end{enumerate}

Now let 
\begin{equation}\label{eq:Lag_k_wrapped}
L^k := \rho^{-\tilde \epsilon k}(L)    
\end{equation}
for small positive number $\tilde \epsilon$. So larger $k$ means more negatively wrapped. Following \cite[Equation (3.40)]{AA24}, we define a directed category $\mc{O}$ with objects the Lagrangians $\{L^j\}_{j \in \bZ}$ so they form an exceptional collection 
$$
\cO(L^{j_1},L^{j_2}) := CF(L^{j_1},L^{j_2}), \quad j_1 < j_2.
$$
In other words, when the ends of the input projected Lagrangian near infinity lie above/counterclockwise from the ends of the output, there is a Floer complex over their intersections. In the other direction it is 0. 

The differential $\dd: CF^m(L_0,L_1) \to CF^{m+1}(L_0,L_1)$ on the cochain complex of transversely intersecting Lagrangians in Equation \eqref{eq:chain_level_mor} is
\begin{equation}\label{eq:usu_HF}
    \dd(p) = \sum_{q \in L_0 \cap L_1} \sum_{\substack {[u]:[u]=\beta \in \pi_2(Y^0,L_0 \cup L_1) \\ \text{ind}(u)=1}} \# \cM(p, q;[u],J)/(u \sim u_a) e^{-2\pi \int_\beta \omega} \cdot  q
\end{equation}
where 
\begin{equation}\label{eq:diff_mod_space}
\cM(p, q;[u],J):=\left\{u: (\bR \times [0,1], \bR \times \{0\} \cup \bR \times \{1\} )\to (Y^0,L_0 \cup L_1) \middle\vert \begin{array}{l} \overline\partial_J(u)=0,\ 
 E(u)<\infty,\\ u(\bR,0) \subset L_0, \ u(\bR,1) \subset L_1, \\\lim_{s \to -\infty} u(s,t) = q,\\\  \lim_{s \to \infty} u(s,t) = p
\end{array}\right\}.
\end{equation}
We quotient by $u \sim u_a$ where $u_a(s,t) = u(s-a,t)$ for each $a \in \bR$, the translation in the $s$ direction, and $J$ is an almost complex structure compatible with $\omega$ which is also regular, meaning the linearization of $\overline{\dd}_J$ is surjective for all $u \in \cM(p, q;[u],J)$. In order for this to be well-defined, meaning in order for us to be able to count the moduli space, it suffices for it to be a 0-dimensional compact smooth manifold. The 0-dimensional part is ensured by choosing $p$ and $q$ so that $\deg(p) +1= \deg(q)$. The compact part is ensured by the curves $u$ satisfying a maximum principle, the way holomorphic curves do, and by excluding bubbling. And the smooth manifold part can be ensured by the existence of regular $J$.  

This was proven in \cite{Ca20} in the setting for $(Y,-v_0)$. The difference between the set-up in \cite{Ca20} and here is that we have a puncture at $0$ in the base so therefore we have additional Lagrangians fibered over single arcs that can wind infinitely many times around 0 in the base. So we have several cases to consider.

\subsubsection{Algorithm to derive the category}\label{sec:alg}

As outlined at the start of this chapter, here is the algorithm used to invert the quasi-units:

\begin{enumerate}
   \item For chain complexes bounded below, over an abelian category with enough projectives, there is a model category structure where the weak equivalences are the quasi-isomorphisms, the fibrations are the epimorphisms, and the cofibrations are injective maps ${\displaystyle i:C_{\bullet }\to D_{\bullet }}$ where the cokernel complex ${\displaystyle {\text{Coker}}(i)_{\bullet }}$ is a complex of projective objects \cite[\textsection 1.2 Example B]{Qu67}. It follows that the cofibrant objects are the complexes whose objects are all projective.  In our case (\cite[Equation (3.40)]{AA24}, \cite[Equation (4.2)]{Ca20}) the model category will be a category of modules over another category $\mc{O}$, 
\begin{equation}\label{eq:dir_cat}
    \mathcal{O}^*(L_0^{k_0},L_1^{k_1}) = \begin{cases}
        CF^*(L_0^{k_0}, L_1^{k_1}) & k_0 < k_1\\
        \bC \cdot \text{id} & k_0=k_1, L_0 = L_1 \\ 
        0 & \text{otherwise}
    \end{cases}
\end{equation}
where id acts as the identity morphism when composed with any $CF^*(L_0^{k_0},L_1^{k_1})$. Then the model category $\mathcal{A}$ is the category of $\mc{O}$-modules with objects $\mc{Y}_{L}$ given by Yoneda modules $X \mapsto \mc{O}(X,L)$ and quasi-isomorphisms are the quasi-units, which are defined in the next subsection, \textsection \ref{sec:qisos}. Counts of curves between multiple intersection points define the $A_\infty$-structure $\mu^k$.
    
    \item Construct the homotopy category by deriving $\mc{A}$ so that quasi-isomorphisms become isomorphisms. An object is therefore an equivalence class of a chain complex, up to quasi-isomorphism, and a morphism is an equivalence of a chain map, up to chain homotopy. This is a localization process \cite[\textsection 10.3]{We94} which has an $A_\infty$-version in \cite{LO06}.
    \begin{enumerate}
    \item The directed category $\bZ$ indexes the diagram in $\mathcal{O}$ for each Lagrangian $L$, given by $\ldots \to \mc{Y}_{L^k} \xrightarrow[]{\mu^2(e_{L^k},-)}\mc{Y}_{L^{k+1}} \to \ldots $ for $k \in \bZ$, with multiplication by quasi-units $e_{L^k}$. Define
    \[
    \mc{Y}_{L^\infty}:= \underset{k \in \mathbb{Z}}{\text{hocolim}}(\mc{Y}_{L^k}).
    \]
    
    \item To compute homs with $\mc{Y}_{L^\infty}$ we use the standard model in wrapped Floer theory
    $$
\mathcal{Y}_{L^\infty}=\text{Cone}\left(\bigoplus_{k=0}^\infty \mathcal{Y}_{L^k} \xrightarrow[]{\text{id}-e_{L^k}}\bigoplus_{k=0}^\infty \mathcal{Y}_{L^k}\right). 
$$
We will need the continuation maps of \textsection \ref{sec:contn} to relate Floer complexes of different amounts of wrapping, allowing us to derive and localize the morphisms in \textsection \ref{sec:mors_defn}.
\end{enumerate}
\end{enumerate}

\begin{remark}\label{rem:wrap}
The homotopy direct limit over a collection of isotopies avoids the definition depending on the amount of wrapping. Note that in $\mathbb{C}^*$, using a quadratic Hamiltonian to wrap as in \cite[\textsection 4.2]{fuk_intro} is equivalent to a colimit by linear Hamiltonians, the original definition of the wrapped Fukaya category \cite[\textsection 5.1]{AS10}, by \cite[Proposition 6.1.9]{OT24}. We will use this below in  Subsection \ref{sec:grading}: Lagrangian gradings.
\end{remark}

\subsubsection{Quasi-isomorphisms: quasi-units}\label{sec:qisos}

Following the convention of \cite[Equation (3.42)]{AA24}, a quasi-unit roughly counts certain curves $\Sigma:=\mathbb{D}\backslash \{1\} \to Y^0$ with one moving Lagrangian boundary condition. The boundary condition is defined by $\rho^{-\tilde \epsilon g_e(s)}(L)$ for $g_e:\bR \cong \partial \Sigma \to [k,k+1]$ monotonically increasing smooth and constant near the ends, meaning there is some $K\in \bR$ so 
\begin{equation}
    g_e(s) = \begin{cases}
        k, & s<-K\\
        k+1, & s>K\\
        \text{interpolate increasing}, & s \in [-K,K]
    \end{cases}
\end{equation}
In particular, the Lagrangian isotopy is traversed in reverse along the boundary from $L^k$ to $L^{k+1}$.  

The quasi-unit is the image of 1 under the PSS isomorphism $PSS: H^*(L) \to HF^*(L^k,L^{k+1})$. We define it on the chain-level by a disc count, 
\begin{equation}\label{eq:quasi-unit}
    e_{L^k} := PSS(1)= \left[\sum_{p \in L^k \cap L^{k+1}} \sum_{\substack{[u]:[u]=\beta \in \pi_2(Y^0,\bigcup_{s \in \bR} \rho^{-\tilde \epsilon g_e(s)}(L^k))\\ \text{ind}(u)=0}} \# \mathcal{M}(p;[u],J)  e^{-2\pi \int_\beta \omega} \cdot p\in CF(L^k, L^{k+1})\right], \\
\end{equation}
\begin{equation}
\begin{aligned}
    & \mathcal{M}(p;[u],J) = \{u: (\bD\backslash \{1\}, \dd \bD\backslash \{1\}) \to (Y^0, \bigcup_{s \in \bR} \rho^{-\tilde \epsilon g_e(s)}(L^k))\mid \overline{\dd}_J(u)=0,\\
    & u(s,0) \in \rho^{-\tilde \epsilon g_e(s)}(L^k), \lim_{s \to \pm\infty} u(s,0) = p \}
    \end{aligned}
\end{equation}
where we use coordinates $s+it$ on the closed upper half plane $\overline{\mathbb{H}}$ under the biholomorphism $\overline{\mathbb{H}}\to \mathbb{D}\backslash \{1\}$ given by $z=s+it \mapsto \frac{z-i}{z+i}$.

\subsubsection{Inverses: continuation maps}\label{sec:contn}

Since we will see in Claim \ref{claim:q-iso_wrap_limit} that the input and output in morphisms wrap in opposite directions, we need an inverse to quasi-units to connect Floer complexes in Lemma \ref{lem:contn_inverse}. These will be the continuation maps. The algorithm to define and use them is as follows. We follow the methodology of \cite[Chapter 3]{AA24}, noting that we have no $X_H$ term since we have no wrapping in the fiber. Another related setting is that of \cite[\textsection 8.6]{GPS24}, since curves projected to the base are then in an exact symplectic manifold. 

\begin{enumerate}
\item Define the isotopy of the moving Lagrangian boundary according to the wrapping needed. We use that in Equation \eqref{eq:mov_Lag_isotopy} as a lift of the Hamiltonian vector field associated to the Hamiltonian in the base in Equation \eqref{eq:Hamilt}, since we have no wrapping in the fiber. The moving Lagrangian boundary condition introduces a perturbation into the Cauchy-Riemann equation \cite[Remark 1.10]{fuk_intro}, known as Floer's equation. Here, it is
    \begin{equation}\label{eq:perturbed_Floer}
        (du+ (\xi^\tau)^\# \otimes d \tau)^{0,1}=0
    \end{equation}
    for a suitable $\tau$ depending on two parameters (different from the $\tau$ used to denote the complex parameter on $X$). Continuation maps count these solutions.
    \item Show that the continuation map moduli space in Equation \eqref{eq:strip_moduli} can be counted.
    \begin{enumerate}
        \item A 0-dimensional manifold: By Fredholm theory, it suffices to show that a regular $J$ exists. This means that the linearization of the Floer operator in Equation \eqref{eq:perturbed_Floer} is surjective. This is a standard argument provided curves $u$ in the moduli space are shown to have a dense set of somewhere injective points, which follows from the choice of perturbation $\tau$. 
        \item Compact: this follows from Gromov compactness. This is a standard argument provided there's a maximum principle that the curves $u$ in the moduli space stay in a compact subset for a given relative homotopy class $[u]$. This follows from rescaling the curve to obtain a solution to the elliptic Cauchy-Riemann equation  $(du)^{0,1}=\overline{\dd}_J(u) = \frac{1}{2}(du + J \circ du \circ j)=0$ which satisfies a maximum principle on the interior, and then a boundary analysis for the boundary.
    \end{enumerate}
\end{enumerate}

Let $\Sigma = \bD \backslash\{p,q\}$, the complement of two points on the boundary of a disc, which is biholomorphic to $\bR \times [0,1]$. Let $s$ be the coordinate on $\bR$. A continuation map, roughly, counts pseudo-holomorphic strips $\Sigma \to Y^0$ with moving Lagrangian boundary conditions that, under $W_0$, have prescribed wrapping behavior according to the piecewise-defined Hamiltonian in Equation \eqref{eq:Hamilt}. 
\begin{definition}\label{def:lambda}
    The moving Lagrangian boundary used to define the continuation maps is
\begin{equation}\label{eq:Lambda}
\Lambda_s(L) = \rho^{-\tilde\epsilon \tilde\chi(s)}(L), \quad s \in \bR
\end{equation}
where $\tilde\chi: \bR \to [0,1]$ is a smooth monotonically decreasing function, constant near $-\infty$ and $+\infty$ at 1 and 0 respectively, and $\tilde \epsilon$ is small. 
\end{definition}

This isotopy moves in the opposite direction of the moving boundary condition on the quasi-unit. Because $L^k=\rho^{-\tilde \epsilon k}(L)$, the isotopy $\Lambda_s$ moves $L^{k+1}$ to $L^k$. That is, there exists $\tilde K$ so that 
\begin{equation}
    \Lambda_s(L^k) = \begin{cases}
         L^{k+1}, & s< -\tilde K\\
        L^k, & s > \tilde K
    \end{cases}.
\end{equation}
So in words, $\Lambda_s(L): \bR \to Y^0$ rotates Lagrangian ends of $W_0(L)$ counterclockwise outside a compact set around 0, meaning $\frac{d}{ds}\theta_h(\Lambda_s(L))>0$ for $\theta_h$ as in Definitions \ref{def:U_shape} and \ref{def:F2objs}, by Equation \eqref{eq:pos_wrap_partial}. When $L=K_j$, $W_0(\Lambda_s(L))$ also fully wraps clockwise around 0 by Equation \eqref{eq:pos_wrap_full}. 

We introduce a perturbation term moving in the direction of $\rho^{\tilde \epsilon \tilde \chi(s)}$ on the boundary:
\begin{equation}\label{eq:perturb_term}
(\xi^{\tau})^\# \otimes d\tau
\end{equation}
for $\tau(s,t):\bR \times [0,1] \to \bR$ a smooth function such that
\begin{equation}\label{eq:tau}
\tau(s,j)=\tilde \epsilon \tilde \chi(s), \quad j=0,1.
\end{equation}
Thus $\tau$ is a way to extend the Lagrangian perturbation on the boundary of the disc to the whole disc, providing enough parameters in the perturbation to ensure transversality of the moduli space of solutions to the perturbed equation in Equation \eqref{eq:perturbed_Floer}, while also allowing us to cancel the moving Lagrangian boundary condition in $\Lambda_s(L)$ by rescaling the map $u$. The perturbation term $(\xi^{\tau})^\# \otimes d\tau$ is valued in the vector $(\xi^\tau)^\#$ on the target and $d\tau$ inputs vectors on the domain, so tensored together it is the same type as $du$. It describes how $du$ should be perturbed infinitesimally over the domain of $u$. 

To ensure enough perturbation parameters to prevent solutions to the Floer equation that have a lot of symmetry because they are nowhere somewhere injective, we need $\tau$ to depend on the full set of two parameters on the domain $\Sigma$, so both $s$ and $t$. We also perturb with the correct orientation for a maximum principle to hold, that is, in directions out of the strip on the boundary. This is related to \cite[\textsection (8k), Equation (8.21)]{seidel} for moving boundary conditions as in continuation maps. (In cases such as for a differential on $CF(W_0(K_j),W_0(K_j))$, one can perturb in the $t$-direction only. That is, one can use $(\xi^\tau)^\# \otimes dt$, which is 0 on $\dd_s$ tangent to the boundary and $(\xi^\tau)^\#$ on $\dd_t$ normal to the boundary.)

Analogous to \cite[Equation (3.44)]{AA24}, define
\begin{definition}\label{def:contn_maps}
    The continuation maps $F_{L_0^k,L_1^j}: \mathcal{O}^*(L_0^k, L_1^j) \to \mathcal{O}^*(L_0^{k+1},L_1^{j+1})$ are
\begin{equation}\label{eq:contn_map}
    F_{L_0^k,L_1^j}(p)  = \sum_{q \in L_0^{k+1} \cap L_1^{j+1}}  \sum_{\substack {[u]:[u]=\beta \in \pi_2(Y^0,\cup_{s \in \bR} \Lambda_s(L_0^k) \bigcup \cup_{s \in \bR} \Lambda_s(L_1^j))\\ \text{ind}(u)=0}} \#\mathcal{M}(p,q;[u],J) e^{-2\pi \int_{\beta}  \omega} \cdot q 
    \end{equation}
\begin{equation}\label{eq:strip_moduli}
    \begin{aligned}
   \mathcal{M}(p,q;[u],J) & =\{u: (\bR \times [0,1], \bR \times \{0\} \cup \bR \times \{1\} ) \to (Y^0, \cup_{s \in \bR} \Lambda_s(L_0^k) \bigcup \cup_{s \in \bR} \Lambda_s(L_1^j) )  \mid \\ & (du+ (\xi^\tau)^\# \otimes d \tau)^{0,1}=0,
     u(s,0) \subset  \Lambda_{s}(L_0^k), u(s,1) \subset  \Lambda_{s}(L_1^j),\\
     & \lim_{s \to \infty}u(s,t)=p, \lim_{s \to -\infty}u(s,t)=q \}
    \end{aligned}
\end{equation}
where recall $\xi^t$ is the Hamiltonian vector field on the base dual to $dH$. Since $\rho^t$ is the flow of the horizontal lift $(\xi^t)^\#$, then $(\xi^\tau)^\#$ is a section of Hom($\Sigma, TY^0)$ which at a point $z_0 \in \Sigma$ is $(\xi^{\tau(z_0)})^\# = \frac{d}{dt} \rho^t\mid_{t=\tau(z_0)}$. 
\end{definition}

\begin{lemma}\label{lem:mod_mfld}
    The moduli space $\mathcal{M}(p,q;[u],J)$ is a smooth 0-dimensional manifold.
\end{lemma}

\begin{lemma}\label{lem:mod_cpct}
    The moduli space $\mathcal{M}(p,q;[u],J)$ is compact.
\end{lemma}

\begin{corollary}\label{cor:well_def}
    The moduli space $\mathcal{M}(p,q;[u],J)$ is a finite set of points which can be counted.
\end{corollary}

\begin{proof}[Proof of Corollary \ref{cor:well_def}]
By Lemmas \ref{lem:mod_mfld} and \ref{lem:mod_cpct}, the moduli space is a compact 0 dimensional manifold.
\end{proof}

\begin{proof}[Proof of Lemma \ref{lem:mod_mfld}] This is a standard result in the literature (\cite[\textsection 8]{AD14} for moduli spaces of cylinders between Reeb orbits, \cite[\textsection 6.7]{jholbk} for domain dependent $J$, and \cite[\textsection (9k), Equation (9.27)]{seidel}, \cite[Lemma 8.5]{AS10}, \cite[\textsection 4]{Se12} for discs). We apply them to our setting. The moduli space is 0-dimensional because we count index 0 discs. 

The moduli space is the vanishing of the differential operator $u \mapsto (du+ (\xi^\tau)^\# \otimes d \tau)^{0,1}$, thought of as a section of a Banach bundle, so by the Sard-Smale theorem it suffices to show the derivative (also known as the linearization) of the differential operator at each map $u$ in the moduli space, is surjective. (A finite-dimensional example of this process is the vanishing of $f(x,y,z) = x^2+y^2+z^2-a: \bR^3 \to \bR$, which cuts out a smooth manifold for $a \neq 0$, a sphere, because its linearization $df = [2x\; 2y\; 2z] \neq \underline{0}$ at each such $(x,y,z)\in f^{-1}(0)$.) We have three choices of datum to vary: $u$, $J$, and $\tau$. The linearization of $(du+ (\xi^\tau)^\# \otimes d \tau)^{0,1}$ is 
$$
(d\mathcal{F})_u(\dot u, \dot J, \dot \tau) = D_u \dot u + \frac{1}{2}\dot J \circ (du+ (\xi^\tau)^\#\otimes \tau) \circ j + \dot \tau^{0,1}
$$
where the first two terms are justified in \cite[Equation (4.15)]{Ca20} and the third term is from an infinitesimal perturbation of $\tau$. (In \cite[\textsection 8.4]{AD14} their Floer equation involves a perturbation from $JX(s,t)$ instead of $\tau$. Our $\dot \tau^{0,1}$ is similar to $(\delta Y)^{0,1}$ in \cite[Equation (9.26)]{seidel}, and see \cite[Equation (8.10)]{AS10} for perturbation term $-X \otimes \gamma$ and linearization with respect to $u$ and $J$.)

The argument to show that the image of $(d\mathcal{F})_u$ is surjective relies on $u$ containing a dense set of somewhere injective points. (See \cite[\textsection 8.6]{AD14}, \cite[Proposition 2.5.1]{jholbk} for spheres, or \cite[\textsection 4.5: Key regularity argument]{Ca20} for an example with disc configurations.) Then we have surjectivity because the set of points perpendicular to the image of $(d\mathcal{F})_u$ is zero. For if a nonzero $\eta$ is 0 when paired with image$(d\mathcal{F})_u$, then use an open set of somewhere injective points around a point $z$ where $\eta_z \neq 0$ to construct an element in image$(d\mathcal{F})_u$ that would pair nontrivially with that element. This contradicts that they paired to 0. So it's enough to show there is a dense set of somewhere injective points. 

Recall that an injective point $z \in \bD$ is one such that $d_zu \neq 0$ and $u^{-1}(u(z)) = \{z\}$. We make use of the perturbation $\tau$ to ensure that discs are somewhere injective. Decomposing discs with Lagrangian boundary condition to somewhere injective discs has been studied in \cite{KO00,lazz1,lazz2}. Here we show discs are already somewhere injective. Note that the boundary conditions prevent multiply covered discs.

\begin{lemma}\label{lem:inj}
    The set of injective points for a curve $u$ counted in Equation \eqref{eq:strip_moduli} is dense in $\mathbb{D}$.
\end{lemma}

\begin{proof}
To show non-injective points are isolated, first select $u: \Sigma \to Y^0$ a curve in the moduli space in Equation \eqref{eq:strip_moduli}. Suppose by contradiction that we can pick a $z_0 \in \Sigma$ such that there exists $z_0 \neq z_0'$ in $\Sigma$ and disjoint open neighborhoods of each, call them $U, U'$ so that there is a biholomorphism $\varphi: U \to U'$ where $u(z) = u\circ \varphi(z)$ for all $z \in U$ but $z \neq \varphi(z) \in U'$. Because $u|_{U} = u\circ \varphi$, we have $du|_{U} = du \circ d\varphi$ on $U$. But there exist vector $v$ and point $z\in U$ where $(\xi^{\tau(z)})^\# d\tau _{z}(v) \neq (\xi^{\tau(\varphi(z))})^\# d\tau _{\varphi(z)}(d\varphi(v))$. Then $u$ couldn't have solved the Floer equation in Equation \eqref{eq:perturbed_Floer} at $\varphi(z)$ if it did at $z$. The perturbation $\tau$ distinguishes that $z \neq \varphi(z)$ in the domain, even though from the perspective of the image of $u$ they are the same. Thus non-injective points are isolated and injective points are dense.
\end{proof}

This concludes the proof of Lemma \ref{lem:mod_mfld} that $\mathcal{M}(p,q;[u],J)$ is a smooth 0-dimensional manifold.
\end{proof}

Now we prove Lemma \ref{lem:mod_cpct} that the moduli space is compact. Since the moduli space of discs is a metric space \cite[p47]{jholbk}, it suffices to show it is sequentially compact, that every sequence has a convergent subsequence whose limit is in the space. This is ensured by Gromov compactness, named after \cite{Gr85}. To prove Lemma \ref{lem:mod_cpct}, we split into two parts. 
\begin{enumerate}[label=(\alph*)]    
    \item We exclude discs escaping to infinity, in Lemma \ref{lem:max}, by proving a maximum principle. Then Gromov compactness applies and states there are only certain types of bubbling that could be in the limit, Proposition \ref{prop:poss-configs}.
    \item We exclude the possible degenerations of curves, in the proof of Lemma \ref{lem:mod_cpct}.
\end{enumerate} 

\begin{proposition}\label{prop:poss-configs}
    A sequence of pseudo-holomorphic discs with Lagrangian boundary conditions, and satisfying a maximum principle, has a convergent subsequence, which limits to a pseudo-holomorphic curve that may degenerate. Possible degenerations are 
    \begin{enumerate}
        \item sphere bubbles (energy concentrates at an interior point)
        \item disc bubbles (energy concentrates at a boundary point)
        \item strip-breaking (energy concentrates at a limiting intersection point).
    \end{enumerate} 
\end{proposition}

\begin{proof}
    This is proved for compact symplectic manifolds and discs with Lagrangian boundary conditions in \cite[\textsection 4.6, Lemma 4.6.5]{jholbk}, which applies once we have the maximum principle in a compact set. See also \cite[Example 2.4.14, Theorem 2.4.36]{fooo1}.
\end{proof}

Consider the annulus $A$ from Equation \eqref{eq:annulus}. Outside of $A$, projected Lagrangian ends wrap positively under $\rho^t$ for increasing $t$. Recall that in an LG model, curves counted in structure maps are sections of the superpotential $W_0$ \cite[Equation (17.2)]{seidel}. In other words, because $u$ is 1-1 there exists $z(s,t) \in \bD$ uniquely defined by
\begin{equation}
(W_0 \circ u)(z(s,t))=A(s,t)
\end{equation}
where $A(s,t)$ parametrizes
\begin{equation}
    D:=(W_0 \circ u)(\bD).
\end{equation}
Then we can define $u_D:D \to Y^0$  
\begin{equation}\label{eq:make_section}
    u_D(A(s,t)) = u(z(s,t))
\end{equation}
as a $((W_0\circ u)_*j,J)$-holomorphic section of $W_0$ if $u$ was a $(j,J)$-holomorphic curve. 

\begin{lemma}\label{lem:max}
For all $u \in \mathcal{M}(p,q;[u],J)$, the image of $u$ is contained in a compact subset of $Y^0$.
\end{lemma}

We adapt the methodology of \cite[Proposition 3.10]{AA24}. 

\begin{proof}  By Equation \eqref{eq:make_section}, we may assume $u$ is a section over a bigon in the base of $W_0$, and we identify $(s,t) \in \bR \times [0,1]$ with $A(s,t) \in \mathbb{C}^*$. First, let's define the region we expect the curve to stay in within the base, because fibers are already compact. Enlarge $A$ to a larger annulus $\tilde A :=\{\tilde{R}_1 \leq  |w| \leq \tilde{R}_2\} \ni W_0(p), W_0(q)$ containing the points the projected strip limits to. Also let $(R_1,R_4) \subset (\tilde{R}_1,\tilde{R}_2)$ so outside of $\tilde A$, projected Lagrangians have constant angle and isotopies only rotate them by Equation \eqref{eq:Hamilt}. We'd like to show $u(\mathbb{R} \times [0,1]) $ is contained in the compact subset $\Omega:={W_0^{-1}(\tilde A)}$. We will do this by showing that maximums of $|W_0 \circ u|$ and $|\frac{1}{W_0} \circ u|$, if they exist, must be attained on the boundary. Then we show that any boundary maximum, if it exists, must lie within $\Omega$. If it doesn't exist, then a maximum is approached as we limit to $p$, $q$, and their images are also in $\Omega$. Therefore the image of $u$ is contained in $\Omega$. 

Recall that away from the singular fiber, $J|_{Y^0 \backslash \Omega}=J_0$ is multiplication by $i$ in the affine toric coordinates $t_1,t_2,t_3$ which are related to $y_1,y_2,y_3$ of Remark \ref{rem:toric_coords} by the process described in \cite[Equation (72)]{ACL21}. In particular, since $v_0 = t_3$ we have $W_0 = T^\epsilon(t_3-1)$ and so $W_0|_\Omega$ is $(J_0,i)$-holomorphic because $dW_0 \circ J_0 = T^\epsilon dt_3 \circ i = i \circ d(T^\epsilon(t_3-1))=i\circ dW_0$. Similarly $\frac{1}{W_0} : \mathbb{R} \times [0,1] \to \mathbb{C}^*$ is $(J_0,i)$ holomorphic by composing with the $i$-holomorphic map $w \mapsto 1/w$.

We project $u$ to the base and absorb the isotopy $\xi^\tau$ so it satisfies the elliptic Cauchy-Riemann operator. Let $\tilde \rho^t$ be the flow of $\xi^t$, which is equal pointwise to $W_0 \circ \rho^t \circ W_0^{-1}$ for any choice of $W_0^{-1}$. Then $\tilde u(s,t):=\tilde \rho^{\tau(s,t)} \circ W_0 \circ u(s,t)$ is a $(j,(\rho^\tau)_*j)$-holomorphic curve satisfying
\begin{equation}\label{eq:CR_mod}
(d\tilde u)_{(j,(\rho^\tau)_*j)}^{0,1}=   \dd_s \tilde u + (\rho^\tau)_* j \dd_t \tilde u=0
\end{equation}
where $j=i$ is the standard complex structure. Thus $\tilde u:\bR \times [0,1] \to \mathbb{C}^*$ satisfies the maximum principle for holomorphic functions between open subsets of $\mathbb{C}$. If it has a maximum it must be on the boundary $\bR \times \{0\} \cup \bR \times \{1\}$. 

Let $j_1:=j$ and $j_2:=(\rho^\tau)_* j$ so $\tilde u$ is $(j_1,j_2)$-holomorphic on domains in $\mathbb{C}$. In particular, it is conformal for the metrics $g_1,g_2$ compatible with $\omega_{\bC^*}$ defined by $g_k(-,-) = \omega_{\bC^*}(-,j_k(-))$ for $k=1,2$. Let $\tilde u :=r(s,t)e^{i \theta(s,t)}$ so $r(s,t) = r \circ \tilde u$ and $\theta(s,t) = \theta \circ \tilde u$. 

We need to check that the disc boundary doesn't stretch out arbitrarily far along the Lagrangians under $\tilde u$. We start by showing boundedness away from infinity. Suppose there's a maximum $r(s_0,t_0)$ on the boundary and outside of $\tilde A$, at some $(s_0,t_0)\in W_0(Y^0) \backslash \{|w| \leq \tilde{R}_2\}$ where $t_0$ is 0 or 1. Then 
\begin{equation}\label{eq:1st_deriv_test}
\dd_{s_0} (r\circ \tilde u)=0    
\end{equation} 
there. By the Hopf lemma, at a maximum, the derivative in the other direction $\dd_t$ moving outside the disc, must be positive. Because $e^{\pi (s+it)}$ is a biholomorphism from $\bR \times (0,1)$ to the interior of the upper-half plane which extends to mapping $\bR \times \{0\} \to \bR_+$ and $\bR \times \{1\} \to \bR_-$, we can take coordinates on the strip to be parametrized by points on the upper-half plane (UHP) and then $-\dd_{t^{UHP}}$ points outwards of the upper-half plane where $s^{UHP}+i t^{UHP}$ are the UHP coordinates. The derivative of $r$ on $\tilde u$ must point outwards at $s_0^{UHP}+i t_0^{UHP}:=e^{\pi (s_0+it_0)}$ in the direction of $-\dd_{t^{UHP}}$ so 
\begin{equation}\label{eq:r_inward}
\frac{\dd (r \circ \tilde u)}{\dd t^{UHP}}\mid_{(s_0,t_0)} <0.
\end{equation}
Therefore since $(\dd_s,\dd_t)$ is a positively-oriented $g_1$-orthogonal frame and $\tilde u$ is conformal then $(d\tilde u(\dd_s), d \tilde u(\dd_t))$ is a positively-oriented $g_2$-orthogonal frame and therefore 
\begin{equation}\label{eq:theta_CCW}
\frac{\dd (\theta \circ \tilde u)}{\dd s^{UHP}} \mid_{(s_0,t_0)} >0.
\end{equation}

In words, the image $\tilde u$ moves counterclockwise  in the base $\mathbb{C}^*$ as we move in the direction of increasing $s$ on the boundary of the upper-half plane in the domain. From the definition of $\tilde u$, this movement has boundary on a fixed Lagrangian, either $W_0(L_0^k)$ or $W_0(L_1^j)$. This Lagrangian has ends which have strictly increasing radial coordinate outside the disc $\{|w|\leq \tilde{R}_2\}$ by Definitions \ref{def:U_shape} and \ref{def:F2objs}. The tangent to the curve along the Lagrangian then must have a positive radial component. Equations \eqref{eq:r_inward} and \eqref{eq:theta_CCW} imply $d\tilde u(\dd/\dd s)$ along the boundary is rotated counterclockwise from the radial direction, and hence also counterclockwise from the Lagrangian tangent vector at the same point. But, $\tilde \rho^\tau$ moves the Lagrangian clockwise by Equation \eqref{eq:tau}. Since the two are going in opposite directions, this is a contradiction, and the curve $\tilde u$ could not achieve its maximum on the boundary outside $W_0^{-1}(\tilde A)$. For boundedness away from 0, we apply the above argument to $1/W_0$ on the inverted annulus $1/\tilde A$ and use the counterclockwise rotation indicated by Equation \eqref{eq:pos_wrap_full}.
\end{proof}

\begin{proof}[Proof of Lemma \ref{lem:mod_cpct}]

We exclude the possible types of bubbling described in Gromov compactness in Proposition \ref{prop:poss-configs}. We have the assumptions needed for Gromov compactness to hold by \cite[Proposition 3.13]{AA24} where their $X_H$ and $H$ terms are 0 in our setting since we don't have wrapping in the fibers. The assumptions of their lemma hold because we proved in Lemma \ref{lem:max} above that $[u]$ is contained in a compact subset which we labeled $\Omega$ to match their notation. 

We can exclude sphere bubbles by the dimension argument  \cite[Corollary 4.53]{Ca20}, because discs are somewhere injective by Lemma \ref{lem:inj}. As $W_0(L)$ is a U-shape or arc, it cannot bound a disc in the base. And in a fiber, a linear Lagrangian cannot bound discs so $[u]\in \pi_2(T^4,\ell_k)=0$ must be constant. Therefore there can be no disc bubbles. Strip-breaking, from reparametrization by translation, is excluded because the prescribed moving Lagrangian boundary condition prevents translation as an automorphism; the continuation maps count index 0 strips which limit to fixed Lagrangian boundaries. So the moduli space cannot limit to any configurations not already in it and therefore is compact.

This concludes the proof of Lemma \ref{lem:mod_cpct} that the moduli space $\mathcal{M}(p,q;[u],J)$ is compact.

\end{proof}

To justify the title of this subsection that continuation maps are inverses, we have the following lemma from \cite{AA24}, on cohomology. Composition in $\mu^2$ goes from right to left.

\begin{lemma}[{\cite[Proof of Lemma 3.23]{AA24}}]\label{lem:contn_inverse}
    For $p_k \in  HF^*(L_0^{k}, L_1^j)$ and $p_{k+1} \in  HF^*(L_0^{k+1}, L_1^j)$ 
    \begin{equation}\label{eq:commute}
    \begin{aligned}
    \mu^2(e_{L_1^j},p_{k+1}) &= F_{L_0^k,L_1^j}(\mu^2(p_{k+1},e_{L_0^k}))\\
\mu^2(F_{L_0^k,L_1^j}(p_k),e_{L_0^k})&= \mu^2(e_{L_1^{j}},p_k).
    \end{aligned}
    \end{equation}
    More specifically, taking $p_k$ and $p_{k+1}$ to represent elements in the quotients $\underset{j \to \infty}{\text{colim }} HF^*(L_0^{k}, L_1^j)=$ 
    
    $\coprod_j HF^*(L_0^{k}, L_1^j)/\sim$ and $\underset{j \to \infty}{\text{colim }} HF^*(L_0^{k+1}, L_1^j)=\coprod_j HF^*(L_0^{k+1}, L_1^j)/\sim $ respectively, the induced continuation maps on cohomology $F_{L_0^k,L_1^j}: HF(L_0^k, L_1^j)\to HF(L_0^{k+1}, L_1^{j+1})$ form an inverse to  
    $$
    \mu^2(-,e_{L_0}^k):\underset{j \to \infty}{ \text{colim }} HF^*(L_0^{k+1}, L_1^j)\to \underset{j \to \infty}{\text{colim }} HF^*(L_0^{k}, L_1^j).
    $$
\end{lemma}

\begin{remark}\label{rem:param_moduli_1}
    The proof uses \cite[Lemma 3.19]{AA24} that Equation \eqref{eq:commute} holds up to homotopy on the chain-level by constructing \emph{parametrized moduli spaces} \cite[p. 241]{seidel}, \cite[Lemma 5.1]{Ca20} varying choices in 1 dimensional families and taking the boundary. It still applies in our setting because we have a maximum principle. 
\end{remark}

\subsubsection{Morphisms in the derived category}\label{sec:mors_defn}
As motivated at the start of this chapter we take the following

\begin{definition}[Definition of morphisms]
    Define 
\begin{equation}
\mc{Y}_{L^\infty}:= \underset{{k \to \infty} }{\text{hocolim}}\; \mc{Y}_{L^k}
\end{equation}
over the sequential diagram
\begin{equation}
\ldots \mc{Y}_{L^{k-1}} \xrightarrow[]{\mu^2(e_{L^{k-1}},-)} \mc{Y}_{L^{k}} \xrightarrow[]{\mu^2(e_{L^k},-)} \mc{Y}_{L^{k+1}} \to \ldots.
\end{equation}
Define morphisms
\begin{equation}\label{eq:define_mors}
    \mathcal{W}(L_0,L_1):= \text{Hom}_\mathcal{O}(\mc{Y}_{L_0^\infty},\mc{Y}_{L_1^\infty}).
\end{equation}
\end{definition}
So instead of Floer theory with $L$, we use $\mathcal{Y}_{L^\infty}$. Now we justify why this definition for morphisms means we can push the ends of the input Lagrangian up. The following isomorphism implies that after taking cohomology of the chain level construction of  $\mathcal{W}(L_0,L_1)$, we can take a sufficiently large wrapping of $k$ on the input Lagrangian, then take the usual Floer theory of Equation \eqref{eq:usu_HF}.
\begin{claim}[{\cite[Corollary 3.24]{AA24}}]
    For each integer $k \in \bZ$ and each pair $L_0^k$ and $L_1$ of objects of $\mathcal{O}$, there is a natural isomorphism
\begin{equation}\label{eq:we_can_wrap}
    \underset{j \to \infty}{\text{colim}} HF^*(L_0^k,L_1^j) =\coprod_{j \in \mathbb{N}} HF^*(L_0^k,L_1^j)/(p \sim \mu^2(e_{L_1^j},p)) \to H \mathcal{W}(L_0,L_1).
\end{equation}
\end{claim}
The first step to justify this is true is to note that the following is a model for the homotopy colimit.
\begin{claim}[{\cite[Equation (3.48)]{AA24}}]
\begin{equation}\label{eq:hocolim_model}
\mathcal{Y}_{L^\infty}=\text{Cone}\left(\bigoplus_{k=0}^\infty \mathcal{Y}_{L^k} \xrightarrow[]{\oplus_{k=0}^\infty \mathbf{1}_{\mathcal{Y}_{L^k}} \oplus -\mu^2(e_{L^k},-)}\bigoplus_{k=0}^\infty \mathcal{Y}_{L^k}\right)
\end{equation}
\end{claim}

This is a standard model from topology called the mapping telescope, adapted to abelian categories as described in \cite[\textsection (3g)]{AS10}, \cite{BN93}. One can use it as the definition of the homotopy colimit of a sequential diagram as in \cite[Definition 1.6.4]{Ne01}. More recent homotopical algebra developments, for example \cite[Example 11.5.11]{Ri14}, explain why this agrees with the general definition of homotopy colimit as the left derived functor of the colimit constructed as the geometric realization of the cobar construction \cite[Corollary 5.1.3]{Ri14}. (A motivating example is the double mapping cylinder \cite[Example 6.4.5]{Ri14}.) 

To justify Equation \eqref{eq:we_can_wrap}, which is how we will compute morphisms in the next chapter, we see that it is a corollary of the following two results:
\begin{claim}[{\cite[Lemma 3.22]{AA24}}]\label{claim:q-iso_wrap_limit}
$$
\text{holim}_{k \to \infty} \text{hocolim}_{j \to \infty} \mathcal{O}(L_0^k,L_1^j) \to \mathcal{W}(L_0,L_1)
$$
is a quasi-isomorphism. 
\end{claim}
The proof starts with the mapping telescope model, and then relies on homological algebra properties which still apply in our setting. This leads to the following
\begin{lemma}[{\cite[Lemma 3.23]{AA24}}]\label{lem:mult_e_qiso}
    For all $L_0,L_1$, and $k$, the map
    $$
\Hom_\cO(\mc{Y}_{L_0^{k+1}}, \mc{Y}_{L_1^\infty}) \to \Hom_\cO(\mc{Y}_{L_0^{k}}, \mc{Y}_{L_1^\infty})
    $$
    induced by multiplication by quasi-units is a quasi-isomorphism.
\end{lemma}

The proof follows again from homological algebra, the definition of $\mathcal{O}$, and Equation \eqref{eq:commute}. These all still apply in our setting. Since the bonding maps in the holim of Claim \ref{claim:q-iso_wrap_limit} are isomorphisms on cohomology by Lemma \ref{lem:mult_e_qiso}, taking a projection we find that for any $k$, Equation \eqref{eq:we_can_wrap} holds.

\subsection{Lagrangian gradings}\label{sec:grading} Before we compute the morphisms we need to discuss Lagrangian gradings. Here we assume $j>i+1$ so $CF^*(\ell_i,\ell_j)$ and $CF^*(\ell_{i+1},\ell_j)$ are supported in degree 0 by \cite[Example 4.6]{ACLL_Bfield}, where $\ell_k$ is linear of slope $k$ and the Lagrangian grading is $\frac{1}{\pi}\arctan(k)$. The calculations for the morphisms in Subsection \ref{sec:calcs} can be extended to other $i,j$ using Serre duality on the A- and B-side \cite[Lemma 4.7]{ACLL_Bfield}, done in the proof of Corollary \ref{thm:factor1}. 

We start with the standard grading in $\mathbb{C}$. The typical gradings in a bigon between curves $l_1,l_2$ in $\bR^2$ have degree 0 on the right intersection point and degree 1 on the left intersection point, as in Figure \ref{fig:usu_bigon} and \cite[Example 1.9]{fuk_intro}. On a surface, a grading on a Lagrangian $l_j$ is a continuous choice of angle $\alpha_j(t)$ along the Lagrangian for the tangent vector field $\dot l_j(t)$, relative a line field. Then intersection points are graded as follows.

\begin{figure}[h]
    \centering
\begin{tikzpicture}[scale=1.2]
  \tikzset{
    interarrow/.style={
      decoration={markings,
                  mark=at position #1 with {\arrow{>}}},
      postaction={decorate}
    }
  }
\begin{scope}
\draw[thick, black, -] (-0.3,-0.2) to[out=55,in=125] (2.3,-0.2);
\draw[thick, red, -] (-0.3,0.2) to[out=-55,in=-125] (2.3,0.2);

  \draw[black, ->, thick] (-0.1,0) -- ++(0.5,0.5);   
  \draw[black, ->, thick] (2.12,0) -- ++(0.5,-0.5);   
  \draw[red, ->, thick] (-0.1,0) -- ++(0.5,-0.5); 
  \draw[red, ->, thick] (2.12,0) -- ++(0.5,0.5); 

  \node[black] at (1,.8) {$l_2$};
  \node[red] at (1,-.8) {$l_1$};

  \filldraw[magenta] (-.1,0) circle (1pt);
  \node[magenta, left] at (-.2,0) {\text{deg = 1}};
    \filldraw[magenta] (2.12,0) circle (1pt);

\end{scope}

  \node[above right] at (-1.3,.5)
    {$\alpha_2(0)=\tfrac18$};
  \node[below right] at (-1.3,-.5)
    {$\alpha_1(0)=-\tfrac18$};

  \begin{scope}[xshift=4.5cm]
    \node[draw, rectangle] at (1,2) {$t=0$};

    \node at (.9,1.2)
      {Step 1: $\displaystyle \alpha_2(0)-\alpha_1(0)=\frac{1}{4}$};
\node at (1,0.4) {Step 2: rotate 
  \raisebox{-0.5ex}{\tikz{\draw[thick, black] (0,0) -- (0.3,0.3);}} 
  by $\tfrac14$ to 
  \raisebox{-0.5ex}{\tikz{\draw[thick, red] (0,0) -- (-0.3,0.3);}}};
    \node at (1.2,-0.4)
      {Step 3: $\displaystyle 2\times\bigl(\tfrac{1}{4}+\tfrac{1}{4}\bigr) =1=\text{deg}$};
  \end{scope}
\end{tikzpicture}
\caption{Angles depicted as fractions of $2\pi$ relative the positive horizontal line field}
\label{fig:usu_bigon}
\end{figure}

\begin{itemize}
    \item[Step 1:] Find $\alpha_2(t)-\alpha_1(t)$ at values of $t$ where the paths intersect (direction of the arrows on the base Lagrangians matter). 
    \item[Step 2:] At that same value of $t$, we rotate the line spanned by $\dot l_2(t)$ counterclockwise until it coincides with that spanned by $\dot l_1(t)$ (direction of the arrows on the base Lagrangians don't matter). Since counterclockwise is always positive, we add that amount of rotation to $\alpha_2(t)-\alpha_1(t)$. This is referred to as the reverse canonical short path. \item[Step 3:] Twice this winding number is the degree, because rotating a line back to itself, through an angle $\pi$, should count as one rotation. 
\end{itemize}

In an LG model, fiber and base contributions add, both for Lagrangian gradings and degrees of intersection points, by \cite[Lemma 3.11 and Corollary 3.12]{ACLL_LG}. To grade Lagrangians, a choice of holomorphic volume form $\Omega$ on $Y^0$ must be made, so locally 
\begin{equation}
    \Omega = a(z_1,z_2,z_3)dz_1 \wedge dz_2 \wedge dz_3
\end{equation}
for local holomorphic coordinates $z_1,z_2,z_3$ and non-vanishing $a(z_1,z_2,z_3)$. From there, we define a quadratic complex volume form $\Theta:=(\Omega/|\Omega|)^{\otimes 2}$ and this then defines a squared phase map $\alpha_{\Theta}: LGr(TY^0,\omega)\to U(1)$ by 
\begin{equation}\label{eqn:alpha}
\alpha_{\Theta}(p, V) = \frac{\Omega(e_1,e_2, e_{3})^2}{ |\Omega(e_1,e_2, e_{3})|^2}
\end{equation}
where $p\in Y^0$, $V=T_pL\cong \bR^{3}$ is a linear Lagrangian subspace of $(T_pY^0, \omega(p))$, and $(e_1,e_2,e_{3})$ is an ordered $\bR$-basis of $V$. 

Taking $e_1,e_2$ to span the tangent space of the fiber Lagrangian $\ell_k$, we find its squared phase function is $e^{-2\pi i \cdot 2\phi_k}$ where $\phi_k = \frac{1}{\pi}\arctan(k)$. Its squared phase map is the square of the determinant in coordinates $\left( \begin{matrix} r_1 \\ r_2 \end{matrix}\right)=\left( \begin{matrix} 2 & 1 \\ 1 & 2 \end{matrix} \right)^{-1}\left( \begin{matrix} \xi_1\\ \xi_2 \end{matrix} \right)$, \cite[Example 4.6]{ACLL_Bfield}. On the base take $e_3=\frac{\dot \gamma_{L}(t)}{|\dot \gamma_{L}(t)|}$ to be the unit tangent to the projected curve in the base. The choice of $\Omega$ in the base is then equivalent to a choice of line field with respect to which we measure angles of the tangent vector field $\dot \gamma_L$ from, as mentioned above for surfaces. We interpolate between two models, a positive horizontal line field away from 0 in $\bC^*$, and an angular line field around 0 in $\bC^*$ (positive horizontal line field in the cylinder).

\begin{definition}[Definition of Lagrangian grading]\label{def:grading}
Fiber: the relative squared phase map is
$$
\alpha_{\Theta^{\text{Vert}}}: LGr=T^{4}\times U(2)/O(2) \mapsto U(1), \qquad (x,aO(2)) \mapsto \det(a)^2.
$$

Base: For $|w|>R_2$, $L=L_j$:    Note that $(\exp(2\pi i \alpha_j(t)))^2=\alpha_L^{\text{Hor}}(p)$ the horizontal squared phase function of $L_j$ of \cite[Lemma 3.11, Eq.~(3.8)]{ACLL_LG}. For horizontal lift $(\dot \gamma_{L_j}(t))^\#$ and recalling $r(t)$ and $\theta(t)$ parametrize $L_j$ as in Definition \ref{def:U_shape}, define
\begin{equation}
    \alpha_L^{\text{Hor}}(p) = \frac{(dW_0(\dot \gamma_{L_j}(t)^\#))^2}{|dW_0(\dot \gamma_{L_j}(t)^\#)|^2}= \frac{(\dot \gamma_{L_j}(t))^2}{|\dot \gamma_{L_j}(t)|^2} = \frac{(\dot r(t) + r(t) 2\pi \mathbf{i} \dot \theta(t))^2 e^{2\pi \mathbf{i}\cdot 2 \theta(t)}}{|\dot r(t) + r(t) 2\pi \mathbf{i} \dot \theta(t)|^2}=e^{2\pi \mathbf{i}\cdot  2\alpha_j(t)}.
\end{equation}
For the part of the U-shape around the critical value we take a circle $\dot{r}(t)=0$, $\dot \theta (t)=1$, then
$$
-e^{2\pi \mathbf{i}\cdot 2\theta(t)}=e^{2\pi \mathbf{i}\cdot  2\alpha_j(t)} 
$$
so we can take the choice for $\alpha_j(t)$ to be
\begin{equation}
\alpha_j(t) = \theta(t) + 1/4.
\end{equation}
In other words, $\alpha_j(t)$ is the angle that the tangent vector $\dot \gamma_{L_j}(t)$ makes with respect to the horizontal line field. So we recover the fact that the tangent to the circle is rotated $90$ degrees from the radial vector, when measured with respect to the positive horizontal vector field. In particular, the phase function is $-1/4$ at $b_S$ by Equation \eqref{eq:U-shape_base}. For the ends of the U-shape we have $r(t) = R_4'|t|$ for constant of proportionality $R_4'$ in Equation \eqref{eq:U-shape_base}, $\dot \theta(t)=0$, and $\theta(t) = \theta_h(L)$ so 
$$
e^{2\pi \mathbf{i}\cdot 2\theta_h(L)}=e^{2\pi \mathbf{i}\cdot  2\alpha_j(t)} 
$$
and we can take the choice for $\alpha_j(t)$ to be the choice of angle in $(-\frac12, \frac12)$
\begin{equation}
\alpha_j(t) = \begin{cases}
    \theta_h(L)- \lfloor \theta_h(L) \rfloor & \text{ if } <\frac12\\
    \theta_h(L)- \lfloor \theta_h(L) \rfloor -1 & \text{ if } >\frac12
\end{cases}.
\end{equation}
This is again the angle the tangent vector $\dot \gamma_{L_j}(t)$ makes with respect to the horizontal line field.

For $|w|>R_2$, $L=K_j$:  For mirror symmetry reasons to be explained below, we take negative the above choice of horizontal squared phase function  
\begin{equation}
    \alpha_K^{\text{Hor}}(p) = -\alpha_L^{\text{Hor}}(p)
\end{equation}
so the phase function is shifted by $-1/2$
\begin{equation}\label{eq:K_ends}
    \alpha_j(t) = \begin{cases}
    \theta_h(L)- \lfloor \theta_h(L) \rfloor - \frac12 & \text{ if } <\frac12\\
    \theta_h(L)- \lfloor \theta_h(L) \rfloor -\frac32 & \text{ if } >\frac12
\end{cases}.
\end{equation}

For $|w| < R_1$, $L=K_j$: Note that only $W_0(K_j)$ passes through this region, and not $W_0(L_j)$. Here we take
\begin{equation}
    \alpha_K^{\text{Hor}}(p) = \frac{(d\log W_0(\dot \gamma_{K_j}(t)^\#))^2}{|d \log W_0(\dot \gamma_{K_j}(t)^\#)|^2} = \frac{(\dot \gamma_{K_j}(t))^2}{|\dot \gamma_{K_j}(t)|^2} \bigg / \frac{\gamma_{K_j}(t)^2}{|\gamma_{K_j}(t)|^2}.
\end{equation}
Without wrapping, $W_0(K_j)$ can be parametrized by $\gamma_{K_j}(t) = te^{2\pi \mathbf{i} \theta_h(K_j)}$ and the standard convention would be $\alpha_j(t) = 0$. We instead take $\alpha_j(t)=-1$ to be consistent with $|w|>R_2$. With wrapping, by Remark \ref{rem:wrap}, instead of a colimit we can consider a quadratically-wrapped Lagrangian around 0 with $H(\mu,\theta) = \frac{1}{2}\mu^2$. Then under the time-1 flow of the corresponding Hamiltonian vector field $\mu \frac{\dd}{\dd \theta}$, we find that near 0, $W_0(K_j)$ wraps. The new curve is $\{(\mu, \theta(\mu)) \mid \theta(\mu)= \theta_h(K_j)+\mu,\;  \mu \in \mathbb{R}_+\}$  in action-angle coordinates. In coordinate $w$ on $\mathbb{C}^*$, this can be parametrized as $\gamma(t) = te^{ \mathbf{i} (2\pi\theta_h(K_j) + \log t)}$, by taking $t=r$ and using that the action coordinate $\mu = \frac{1}{2\pi} \log r$. Thus 
\begin{equation}\label{eq:perturb_K_grading}
\begin{aligned}
    \alpha_K^{\text{Hor}}(p) &= \frac{(\dot \gamma_{K_j}(t))^2}{|\dot \gamma_{K_j}(t)|^2} \bigg / \frac{\gamma_{K_j}(t)^2}{|\gamma_{K_j}(t)|^2} = \frac{(1 +  \mathbf{i} )^2 e^{\mathbf{i}\cdot 2 (2\pi\theta_h(K_j) + \log t)}}{|1 +   \mathbf{i}|^2}\bigg / e^{ \mathbf{i}\cdot 2 (2\pi\theta_h(K_j) +  \log t)}\\
    &=\mathbf{i}=e^{2\pi \mathbf{i}\cdot  2\alpha_j(t)}.
    \end{aligned}
\end{equation}
So one choice is $\alpha_j(t) = 1/8$ but we again shift by $-1$ to be consistent so take $\alpha_j(t) = 1/8-1 = -7/8$. In words, this measures that the universal cover of a curve wrapping around the 0 end of the cylinder is a curve of slope 1 that makes an angle $\pi/4$ with the horizontal line field in $(\mu,\theta)$ coordinates. In the complex plane in the base, it wraps around 0.

For $R_1<|w|<R_2$, $L=K_j$: Again only $W_0(K_j)$ passes through this region. We interpolate between the two choices for $\alpha_K^{\text{Hor}}$. Visually, interpolating between the two different choices of horizontal holomorphic one form, $dw$ and $d \log w$, is as follows. Draw a line field on the base, with respect to which tangent vectors of the projected Lagrangians are measured. Outside of the disc of radius $R_2$ it is the horizontal line field and inside the disc of radius $R_1$ it spirals around the origin. In between, rotate the horizontal lines to the angular ones.
\end{definition}

Summary of the base gradings: The Lagrangian U-shapes in the base $W_0(L_j)$ are all graded the same, with phase $-1/4$ at $b_S$. The grading on $W_0(K_j)$ is shifted by $-1$ from that of the ends of the U-shapes because the phase is shifted by $-1/2$. This is done to match morphisms on the B-side in the non-orthogonal direction, see Equation \eqref{eq:deg_shft_homs}.

\subsection{Calculations of morphisms}\label{sec:calcs}
Now we calculate the morphisms $HF(\mathbf{L}_{i}, \mathbf{L}_{j})$ on the A-side where $\mathbf{L}$ is $L$ or $K$. We start with the differential between two U-shapes $L_{i}$ and $L_{j}$, which involves counting bigons over the one depicted in Figure \ref{fig:b1} from $b_1$ to $b_0=b_S$. Note the difference from Figure \ref{fig:usu_bigon}; bigons between U-shapes instead go from degree $-1$ to degree 0 in the base. In a fiber, by \cite[\textsection 4.2, Equation (4.21)]{ACLL_tori} with $a_i=b_i=0$, $g=2$,
\begin{equation}\label{eq:mors_Aside_fiber}
    HF^d(\ell_i,\ell_j) = \begin{cases}
        \delta_{d0} \cdot \mathbb{C}^{(i-j)^2} & i < j\\
        \mathbb{C}^{2 \choose d} & i = j\\
        \delta_{d2} \cdot \mathbb{C}^{(i-j)^2} & i > j
    \end{cases}
\end{equation}
where $\delta$-function $\delta_{d0}$  is 0 unless $d=0$ in which case it is 1, and similarly for $\delta_{d2}$.
\begin{example}
    $CF^{d}(\ell_{i+1}, \ell_{i})=HF^{d}(\ell_{i+1}, \ell_{i})$ only has intersection points in degree $d=2$ because the input Lagrangian has larger slope than the output Lagrangian \cite[Example 4.6]{ACLL_Bfield}. 
\end{example} 

Then because under monodromy $\Phi(\ell_i)$ is Hamiltonian isotopic fo $\ell_{i+1}$ \cite[Lemma 4.24]{Ca20}, $HF^d(L_i,L_j)$ is defined by the cohomology in degree $d$ of
  \begin{equation}\label{eq:generalHFses}
  \begin{aligned}
         & CF^d_{b_1}({\ell}_{i+1},{\ell}_{j})[-1] \oplus CF^{d-1}_{b_0}({\ell}_{i},{\ell}_{j})\xrightarrow[]{\dd^{d-1}} CF^{d+1}_{b_1}({\ell}_{i+1},{\ell}_{j})[-1]\oplus CF^d_{b_0}(\ell_{i}, \ell_{j}) \\
         & \xrightarrow[]{\dd^d} CF^{d+2}_{b_1}({\ell}_{i+1},{\ell}_{j})[-1]\oplus CF^{d+1}_{b_0}(\ell_{i}, \ell_{j}).
         \end{aligned}
    \end{equation}
$CF^d_{b_1}({\ell}_{i},{\ell}_{j})[-1]$ means we take the degree $d$ intersection points in the fiber over $b_1$ and degree $-1$ in the base. All differentials within a fiber 
\begin{equation}\label{eq:fiberdiffl_zero}
    \dd: CF^d_{b}(\ell_i,\ell_j) \xrightarrow[]{0} CF^{d+1}_b(\ell_i,\ell_j)
\end{equation} 
are zero. There are no bigons between two linear Lagrangians in the universal cover for $i \neq j$, and when $i=j$ we use \cite[Equation (4.47)]{ACLL_tori} for local systems $a_1=a_2=0$ here. Furthermore, given the location of the Lagrangians, there can be no bigons from $b_0$ to $b_1$ as they must be orientation preserving. So
\begin{equation}\label{eq:no-opp-bigon}
    \dd: CF_{b_0}(\ell_i,\ell_j) \xrightarrow[]{0} CF_{b_1}(\ell_{i'},\ell_{j'}).
\end{equation}

As we will see below in the Proof of Corollary \ref{thm:factor1}, the possible short exact sequences computing $HF(L_i,L_j)$ are therefore the following. Each $\dd$ counts bigons over the base bigon from ${b}_1$ to ${b}_0$.
\begin{equation}\label{eq:poss_ses}
\begin{aligned}
0  \to CF^0_{{b_1}}(\ell_{i+1},\ell_{j})[-1] & \xrightarrow[]{\partial} CF^0_{{b_0}}(\ell_{i},\ell_{j}) \to  HF^0({L}_{i},{L}_{j})\to 0, \qquad & j > i+1\\
            0 & \to CF^{d+1}_{{b}_1}({\ell}_{i+1},{\ell}_{i+1})[-1] = HF^d({L}_i,{L}_{i+1})\to 0, \qquad & d = 0,1\\
            0 & \to CF^d_{{b}_0}(\ell_{i}, \ell_{i}) = HF^d({L}_i,{L}_{i})\to 0, \qquad & d = 0,1\\
            0  \to CF^2_{{b_1}}(\ell_{i+1},\ell_{j})[-1] & \xrightarrow[]{\partial} CF^2_{{b_0}}(\ell_{i},\ell_{j}) \to  HF^2({L}_{i},{L}_{j})\to 0, \qquad & j < i
         \end{aligned}
    \end{equation}

It remains to calculate the differential in the first line in Equation \eqref{eq:poss_ses} (the differential in the last line is the same). We assume $j>i+1$.

\begin{lemma}[Morphisms on $\mathcal{A}_L$]\label{lem: F1_mors}
The Floer theory on the U-shaped Lagrangian objects of $\mathcal{A}_L$ is isomorphic to the Floer theory of the U-shaped Lagrangians of $(Y,-v_0)$ in \cite[Lemma 5.14]{Ca20}. That is,
$$
HF^*(L_i, L_j) =\bC^{(i-j)^2} / C \dd  (\bC^{(i+1-j)^2})
$$
where $\dd$ is the differential between Floer complexes $CF(\ell_{i+1},\ell_j)[-1] \cong \bC^{(i+1-j)^2} \to CF(\ell_i,\ell_j) \cong \bC^{(i-j)^2}$ and $C$ is a scalar function of $(x_1,x_2,y)$ defined by an open Gromov-Witten invariant from \cite{KL15}. 
\end{lemma}

In particular \cite{Ca20} shows that $\dd$, as a linear map, is proportional to the linear map $\vartheta \otimes -$ on $\Ext(\mathcal{L}^{i+1}, \mathcal{L}^j) \to \Ext(\mathcal{L}^{i}, \mathcal{L}^j)$, for $\vartheta$ the defining theta function of $H=\Sigma_2$, thus proving homological mirror symmetry of morphisms that $\Ext^*(\cL^{i}|_H, \cL^{j}|_H) = HF^*(L_{i},L_{j})$.

\begin{proof}

\begin{figure}[h]
\centering
\begin{tikzpicture}[scale=2]

  \node[draw, rectangle] at (0,3.3) {$b_0$ at $t = t_0$};
\coordinate (vOrigin) at (-.3,2.9);
\draw[->, thick, black] (vOrigin) -- ++(-.3,-0.3);         
\draw[->, thick, red] (vOrigin) -- ++(0.28,-0.28);      

\node[align=left] at (1,2.6) {
 \textcolor{red}{$\alpha_1 (t_0) = -\frac{1}{8}$} \\
 \textcolor{black}{$\alpha_2 (t_0) = -\frac{3}{8}$} \\
  deg$(b_0) = 0$
};

  \node[align=left] at (3.3,2.6) {
    Step 1: $\alpha_2(t_0) - \alpha_1(t_0) = -\frac{1}{4}$ \\
{Step 2:  
  \raisebox{-0.5ex}{\tikz{\draw[thick, black] (0,0) -- (.3,0.3);}} 
  + $\tfrac{1}{4}$ = 
  \raisebox{-0.5ex}{\tikz{\draw[thick, red] (0,0) -- (-0.3,0.3);}}}\\
  Step 3: $\displaystyle 2 \times (-\tfrac{1}{4} + \tfrac{1}{4}) = 0$
  };

  \draw[-] (-0.5,2.1) -- (4.2,2.1);

\node[draw, rectangle] at (0,1.8) {$b_1$ at $t = t_1$};

\coordinate (vOrigin2) at (-.2,1.3);
\draw[->, thick, black] (vOrigin2) -- ++(-.3,.3);        
\draw[->, thick, red] (vOrigin2) -- ++(0.28,.28);     

\node[align=left] at (1.1,1.2) {
  \textcolor{red}{$ \alpha_1 (t_1) = -\frac{1}{8} + \frac{1}{4}^*$} \\
  \textcolor{black}{$\alpha_2 (t_1) = -\frac{5}{8}$}\\
  deg$(b_1) = -1$
};

  \node[align=left] at (3.3,1.2) {
    Step 1: $\alpha_2(t_1) - \alpha_1(t_1) = -\frac{3}{4}$ \\
{Step 2:  
  \raisebox{-0ex}{\tikz{\draw[thick, black] (0,0) -- (0.3,-.3);}} + $\tfrac{1}{4}$ to 
  \raisebox{-0.5ex}{\tikz{\draw[thick, red] (0,0) -- (-0.3,-0.3);}}}\\
  Step 3: $\displaystyle 2 \times (-\tfrac{3}{4} + \tfrac{1}{4}) = -1$
  };

\begin{scope}[yshift=3pt, scale=0.8]
{\small  

\begin{scope}[yshift=-2pt, xshift=12pt]
\draw[thick, black, -] (-0.36,0.12) to[out=55,in=125] (2.36,0.12);

\draw[black, ->, thick] (-0.36,0.12) -- ++(-0.2,-0.28);   
\draw[black, ->, thick] (2.36,0.12) -- ++(-0.2,0.28);

  \node[black] at (1,.4) {$W_0(L_2)$};
  \node[red] at (1,-1) {$W_0(L_1)$};

  \node[left] at (-.6,0) {\text{deg$(b_0)$ = 0}};
\node[right] at (2.4,0)
{\text{deg$(b_1)=-1$}};

  \node[above right] at (-1.4,.3)
    {$\alpha_2(t_0)=-\tfrac{3}{8}$};
  \node[below right, red] at (-1.3,-.5)
    {$\alpha_1(t_0)=-\tfrac18$};

\end{scope}
    
  \draw[thick, red] (0,0) arc[start angle=205, end angle=340, radius=1.5cm];
\begin{scope}[yshift=-2pt]
\draw[->, thick, red] (0,0.1) -- ++(.1,-0.3) node[below] {$@t_0$};    
\end{scope}

\draw[->, thick, red] (0.51,-.6) -- ++(0.2,-0.2) ;

\draw[->, thick, red] (1.8,-.8) -- ++(0.36,0.1);
\begin{scope}[xshift=3pt,yshift=8pt]
\draw[->, thick, red] (2.6,-.3) -- ++(0.24,0.36) node[right] {$@t_1$};
\end{scope}

\begin{scope}[xshift=50pt, yshift=16pt]
  \draw[->, thick] (3.3,-0.7) arc[start angle=-45, end angle=40, radius=0.14];
  \node at (4,-0.6) {\small $+\frac{\pi}{2}$ rotation};
\draw[->, thick, red] (3,-0.5) -- ++(.1,-0.3) node[below] {$@t_0$};
\draw[->, thick, red] (3,-0.5) -- ++(0.24,0.2) node[right] {$@t_1$};
\end{scope}
}
\end{scope}
\end{tikzpicture}

        \caption{The differential on $CF(L_i,L_j)$ counts curves covering this bigon depicted in the base, from point $b_1$ of degree $-1$ to point $b_0$ of degree 0. ${}^*$: Add $\tfrac14 \cdot 2\pi$ as the oriented tangent vector rotates counterclockwise $\frac{\pi}{2}$ from $b_0$ to $b_1$}
    \label{fig:b1}
\end{figure}

The differential on $CF(L_i,L_j)$ counts curves covering the bigon depicted in the base in Figure \ref{fig:b1}. Note that U-shaped Lagrangians stay outside the $R_2$ disc by Definition \ref{def:U_shape}, so they do not interact with the puncture at 0. The moduli space in Equation \eqref{eq:diff_mod_space} admits a regular $J$ as follows. There exists a dense set of regular $J$ on $Y$ following the roadmap from \cite{jholbk} outlined in \cite[\textsection 4.5, Lemma 4.47]{Ca20} for the moduli space
    \begin{equation}\label{eq:diffl_thesis}
\cM(p, q;\beta,J):=\left\{u: \bR \times [0,1] \to Y \middle\vert \begin{array}{l}[u] = \beta, \ \overline\partial_J(u)=0,\ 
 E(u)<\infty,\\ u(\bR,0) \subset L_i, \ u(\bR,1) \subset L_j, \\u(-\infty) = q,\  u(+\infty) = p
\end{array}\right\}/(u \sim u_a).
\end{equation} 
Then we restrict $J$ to $Y^0$ and note that the moduli space of Equation \eqref{eq:diffl_thesis} is isomorphic by a 1-1 bijection with the moduli space counted in the differential here 
\begin{equation}
\cM(p, q;\beta,J|_{Y^0}):=\left\{u: \bR \times [0,1] \to Y^0 \middle\vert \begin{array}{l}[u] = \beta, \ \overline\partial_{J|_{Y^0}}(u)=0,\ 
 E(u)<\infty,\\ u(\bR,0) \subset L_i, \ u(\bR,1) \subset L_j, \\u(-\infty) = q,\  u(+\infty) = p
\end{array}\right\}/(u \sim u_a).
\end{equation}
Then we carry out the same calculation of the differential as in \cite[Lemma 5.14]{Ca20}.

    Because we are in the non-exact setting, changes in area affect calculations. In particular, homotoping $\gamma_{L_j}$ and $\delta_{K_j}$ (staying away from the critical values) to different choices then isotopes the Lagrangians above them by parallel transport and affects disc area calculations in the Fukaya category according to the formula \cite[Proposition 6.10]{ACLL_Bfield}. However, in a symplectic Landau-Ginzburg model, the calculations split into fiber and base contributions because the Lagrangian isotopies are exact \cite[Lemma 2.1]{ACLL_LG}. Therefore, different choices can be absorbed into $C$ by \cite[Corollary 2.4]{ACLL_LG}.
\end{proof}

For morphisms from $\mathcal{A}_L$ to $\mathcal{A}_K$, recall that we perturb the ends of the input projected Lagrangian to lie above those of the output, Equation \eqref{eq:we_can_wrap}. After  wrapping, set the notation that ${L_i} \cap {K_j}$ intersect in two points $w_0,w_1$ over points $b_0,b_1 \in \gamma_{L_i} \cap \delta_{K_j}$ in the base as depicted in Figure \ref{fig:ham_inv_Floer} which are 
$$
b_0=W_0(w_0):=\gamma_{L_i}(t_0) = \delta_{K_j}(s_0), \quad b_1=W_0(w_1):=\gamma_{L_i}(t_1) = \delta_{K_j}(s_1).
$$
The differential counts bigons bounded by Lagrangians $L_i \cup K_j$ that cover the bigon in the base between these two paths and intersection points. Hamiltonian-perturb the Lagrangians via Hamiltonian $\tilde H$, so that the boundary of the bigon in the base, $W_0(\phi_{\tilde H}(L_i \cup K_j))$, becomes a circle centered at $-T^\epsilon$ described by 
$$
\zeta: S^1 \to \bC^*.
$$
Parametrize $\zeta$ so that it starts at $\zeta(1)=W_0(\phi_{\tilde H}(w_0))$. Setting the notation $Y^0_b:={W_0^{-1}(b)}$ to denote the fiber over $b$, define 
$$
\Phi^s: Y^0_{b_0} \to Y^0_{\zeta(1)} \to Y^0_{\zeta(e^{2\pi i s})} \to Y^0_{b_0'}
$$
to be the composition of the Hamiltonian isotopy on $Y^0_{b_0}$, parallel transport counterclockwise around the circle an angle of $2\pi s$, and then the inverse of the Hamiltonian isotopy, bringing us to a fiber $Y_{b_0'}$. Lastly, let
$$
\ell_i':=\Phi_{\gamma_{L_i}(t_0)}(\ell_i)=L_i|_{b_0}, \qquad \ell_j'':= \Phi_{\delta_{K_j}(s_0)}(\ell_j)=K_j|_{b_0}, \qquad \Phi:=\Phi^1
$$
so $\Phi=\Phi^1$ denotes the monodromy of $W_0$ around the singular fiber. With this set-up, we can now examine Floer theory between $L_i$ and $K_j$. We take the convention on cochain complexes that $C^n[p]=C^{n+p}$ and consider Floer theory with $K_j[2]$, which has phase 0 by Equation \eqref{eq:K_ends}. With the choice of grading on $L_i, K_j[2]$ in Definition \ref{def:grading}, the degrees of the base intersection points matches that in Figure \ref{fig:usu_bigon}, so appearing in degrees 0 and 1. 

\begin{lemma}\label{lem:chaincx_KL}
$CF^*(L_i, K_j[2]) \cong \bC^{(i-j)^2} \oplus \bC^{(i+1-j)^2}$. 
\end{lemma}

\begin{proof}
As depicted in Figure \ref{fig:ham_inv_Floer},
\begin{equation}
    \begin{aligned}
        CF_{Y^0}^*(L_i, K_j[2]) & = HF^*_{Y^0_{b_0}}(\Phi_{\gamma_{L_i}(t_0)}(\ell_i), \Phi_{\delta_{K_j}(s_0)}(\ell_j))[1] \oplus  HF^*_{Y^0_{b_1}}(\Phi^{1-s}(\ell'_i), \Phi^{-s}(\ell''_j)) \\
        & = HF^*_{Y^0_{b_S}}(\Phi_{\delta_{K_j}(s_0)}^{-1}\Phi_{\gamma_{L_i}(t_0)}(\ell_i), \ell_j)[1] \oplus  HF^*_{Y^0_{b_1}}(\Phi^{1}(\ell_i'), \ell_j')\\
        &=HF^*_{T^4}(\ell_i, \ell_j)[1] \oplus  HF^*_{T^4}(\ell_{i+1}, \ell_j) \cong \bC^{(i-j)^2} \oplus \bC^{(i+1-j)^2}
    \end{aligned}
\end{equation}
where we use $HF$ on the righthand-side because in a fiber $\dd=0$ as there are no bigons between linear Lagrangians in a 4-torus. We start with $\ell'_i, \ell''_j$ in the fiber over the left intersection point $b_0$, and parallel transport around the singular fiber. The second equality follows because parallel transport $\Phi_{\gamma(0) \to \gamma(1)}$ is a diffeomorphism so preserves intersection points, hence we can apply $\Phi_{\delta_{K_j}(s_0)}^{-1}$ and $\Phi^{s}$ to Lagrangians in the first and second factor of the direct sum, respectively. Then note that $\Phi_{\delta_{K_j}(s_0)}^{-1}\Phi_{\gamma_{L_i}(t_0)}$ is parallel transport around a loop not containing the critical value so is Hamiltonian by \cite[Theorem 6.21(ii)]{symp_intro}, and $\Phi^{1}(\ell_i)$ is Hamiltonian isotopic to $\ell_{i+1}$ by \cite[Lemma 4.24]{Ca20}. Thus the statement follows from the Hamiltonian invariance of $HF$ \cite[Theorem 1.5]{fuk_intro}. 

\begin{figure}[h]
\begin{subfigure}{0.55\textwidth}
      \centering
{\fontsize{25pt}{27pt}\selectfont
\resizebox{100mm}{!}{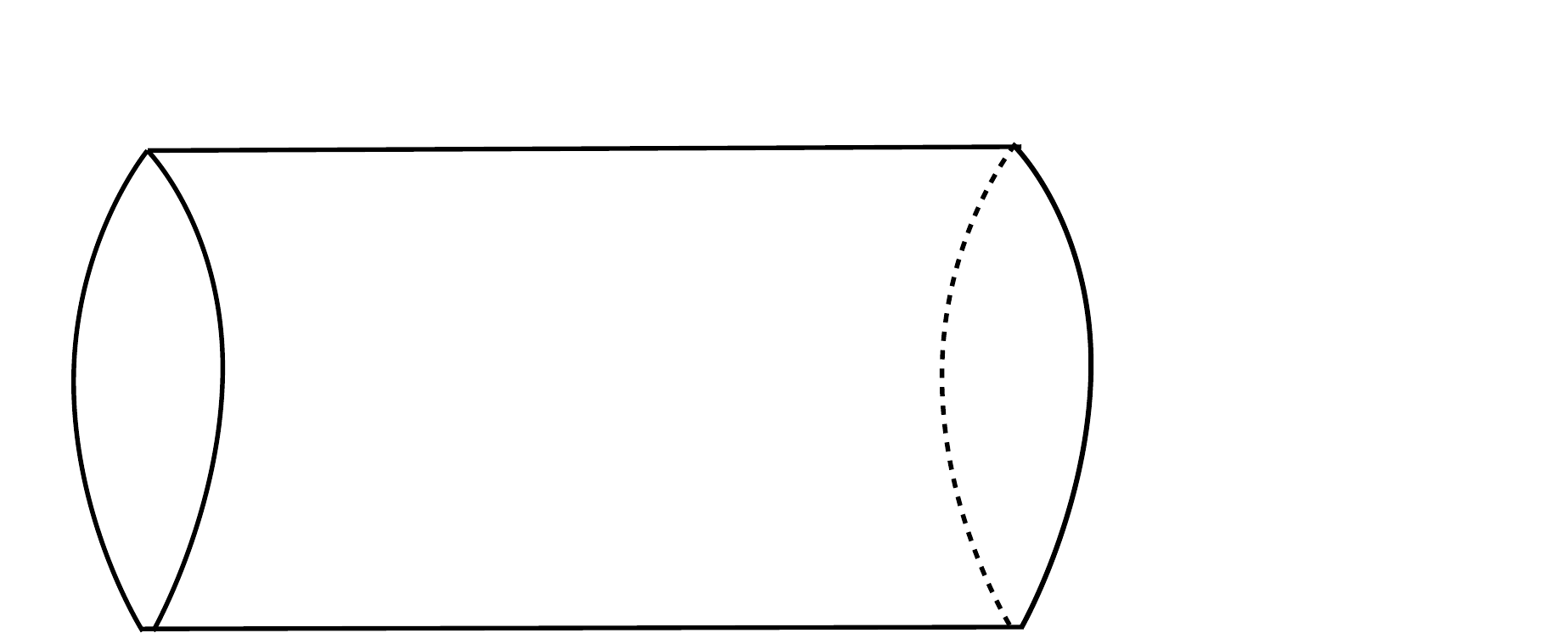}}
    \caption{The base as a cylinder. The Lagrangian intersection points of $CF^*(L_i,K_j[2])$ lie in two fibers over the base intersections, upon positively wrapping the ends of $L_i$ by $\phi$. Here from the left intersection point to the right, we move counterclockwise $2\pi (1-s)$ around $W_0(\phi(L_i))$ and clockwise/to the right $2\pi s$ along $W_0(K_j[2])$. So up to diffeomorphism the differential counts bigons between two fibers, $\ell_i \cap \ell_j$ and $\Phi(\ell_i) \cap \ell_j$ where  $\Phi$ denotes monodromy around the singularity, and $\Phi(\ell_i) \cong \ell_{i+1}$ in the Fukaya category.}
    \label{fig:cylinder_KL}
\end{subfigure}
\hspace{3pt}
\begin{subfigure}{0.43\textwidth}
    \centering
{\fontsize{10pt}{12pt}\selectfont
 \def\svgwidth{0.9\columnwidth}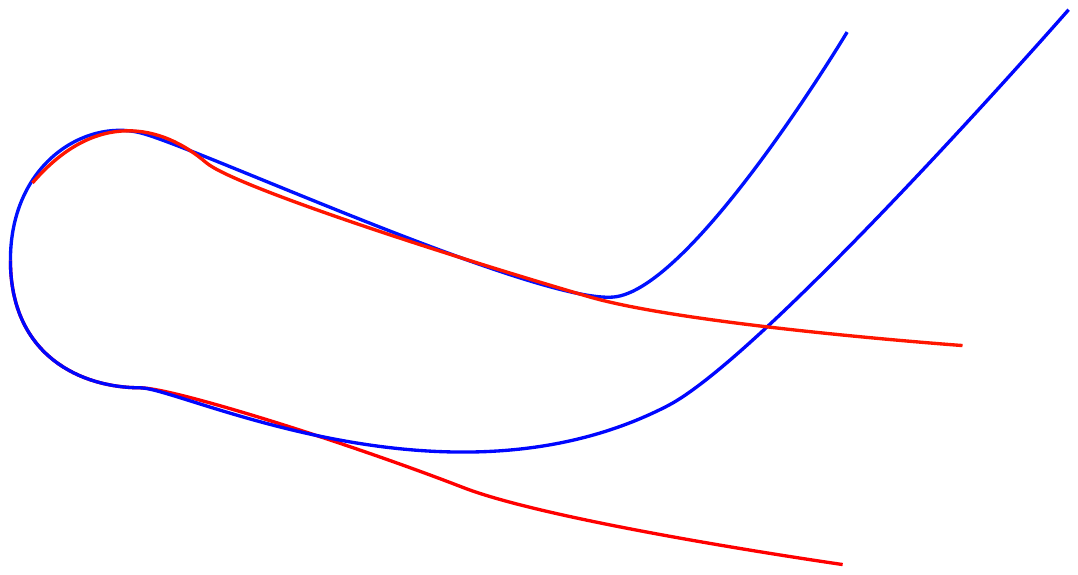}
    \caption{The base as $\bC^*$: projection of indicated Lagrangians under $W_0$}
    \label{fig:plane_KL}
    \end{subfigure}
    \caption{Projection of $K_i[2]$, $L_j$, and the positively wrapped $\phi(L_j)$ to the base of $W_0:Y^0 \to \bC$.  There is a puncture at 0, and  the singular fiber is at $-T^\epsilon$, denoted by $\times$.}
    \label{fig:ham_inv_Floer}
\end{figure}

\end{proof}

Now we compute the differential and pass to cohomology. Note that $HF(L_i,K_j[2])$ is supported in degree 1 but the mirror morphism group is supported in degree 0. We take a shift of $K_j[2]$ down by 1, accordingly, and then $K_j[1]$ has phase $-1/2$ in the base, analogous to the U-shape ends.

\begin{lemma}[$\Hom(\mathcal{A}_L, \mathcal{A}_K)$]\label{lem:relate_to_Ca} The cohomology of $CF(\ell_{i+1},\ell_j)[-1]$ $  \xrightarrow{\dd}  CF(\ell_i,\ell_j)  \to  HF(L_i, K_j[1])  \to  0
$ is  
\begin{equation}
    HF^*(L_i, K_j[1]) =HF^*(L_i, L_j)=\bC^{(i-j)^2} / C  \dd( \bC^{(i+1-j)^2}) .
\end{equation}
\end{lemma}
\begin{proof}
The moduli spaces counted in the differential are isomorphic to those counted for two U-shapes in Lemma \ref{lem: F1_mors}, as illustrated in the identification in Figure \ref{fig:betweenK_L}. Note that the grading on $K_j[1]$, under this identification, matches with the grading on the U-shaped Lagrangian $L_j$.\end{proof}

\begin{figure}[h]
      \centering
{\fontsize{10pt}{12pt}\selectfont
 \def\svgwidth{0.7\columnwidth} 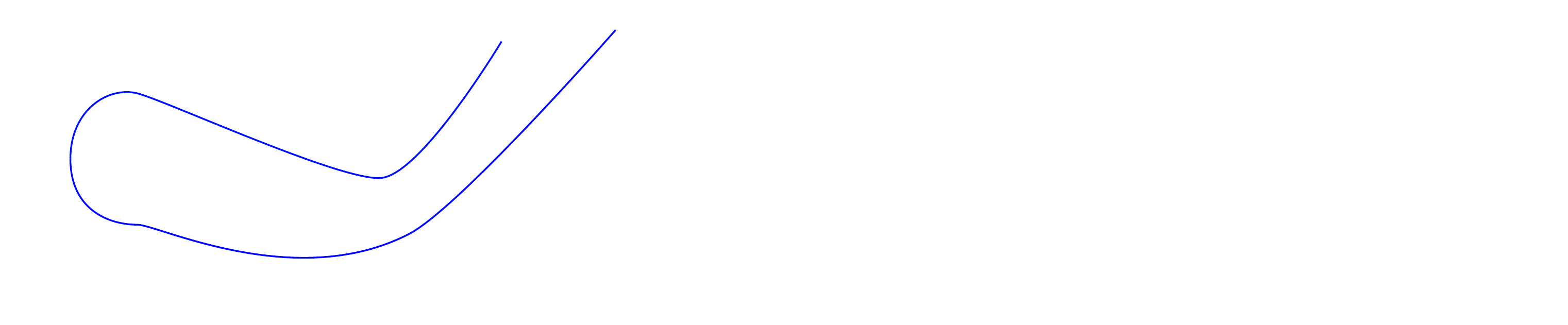}
    \caption{The differential on $CF(L_i,K_j[1])$ is equivalent to the differential on $CF(L_i,L_j)$. This can be seen by applying a Lagrangian isotopy $\phi'$ on the arc $K_j[1]$  homotoping the curve in the base as shown in the solid line, and then completing it to a U-shape as shown with the additional dotted line. The resulting moduli spaces of bigons in the differentials can then be identified. Furthermore, the gradings on $K_j[1]$ and $L_j$ agree under this identification near the intersection points.}
    \label{fig:betweenK_L}
\end{figure}

\begin{lemma}[{$\Hom(\mathcal{A}_K, \mathcal{A}_L)$}]\label{lem:orthoA}
Swapping the order of the objects, $HF(K_j, L_i)=0$.
\end{lemma}

\begin{proof}
Recall from Equation \eqref{eq:we_can_wrap} that we first wrap the ends of the projected output Lagrangian negatively arbitrarily close to $-\pi$, or equivalently we can wrap that of the input $W_0(K_j)$ positively arbitrarily close to $\pi$. Then chain-level morphisms are given by the cochain complex defined in Equation \eqref{eq:chain_level_mor} generated by intersection points between the resulting wrapped input and original output. Wrapping here does not introduce more intersection points because the end of the projected input Lagrangian $W_0(K_j)$ already lies above (in other words, counterclockwise from) the ends of the projected output Lagrangian $W_0(L_i)$, see Figure \ref{fig:ham_inv_Floer}. No intersection points in the base means no intersection points in the total space and therefore the morphism group $HF(K_j,L_i)=0$. 
\end{proof}

This exhibits the semi-orthogonality between  $\mathcal{A}_L, \mathcal{A}_K$ because there are morphisms in one direction and none in the other.

\begin{lemma}[Morphisms on $\mathcal{A}_K$]\label{lem:morsF2}
\begin{equation}
 HF^d(K_i, K_{j}) \cong \begin{cases}
     \mathbb{C}^{(j-i)^2}[x] & j>i, d=0 \text{ or } j<i, d=2\\
     \mathbb{C}^{d \choose 2}[x] & i=j\\
     0 & \text{ else }
 \end{cases}.
\end{equation}
\end{lemma}

\begin{proof} Intersection points in the total space cover intersections in the base of $W_0$. Morphisms and their composition in the base cylinder are given by 
\begin{equation}\label{eq:cyl_base_Floer}
    HF^*_{\mathbb{C}^*}(W_0(K_i),W_0(K_j))\cong\bC[x]
\end{equation}
as algebras, supported in degree 0, by counting intersections and triangles in the universal cover, using a quadratic Hamiltonian as explained in Remark \ref{rem:wrap}. Note that the differential is 0 in the base because there is no bigon between two lines of fixed slope. The details of the calculation are given in \cite[\textsection 4.2]{fuk_intro}, using a rescaling trick. Note that the coordinates $r,\theta$ in \cite[\textsection 4.2]{fuk_intro} are $\mu,\theta$ here. 

In a torus fiber over a base intersection point, $CF_{T^4}^*(\ell_i,\ell_j) = HF^*_{T^4}(\ell_i,\ell_j) $ is given by Equation \eqref{eq:mors_Aside_fiber}. Thus over each base intersection point corresponding to $x^k$ under the isomorphism in Equation \eqref{eq:cyl_base_Floer}, we have 
 \begin{equation}
 HF^d(K_i, K_{j}) =CF^d(\ell_i,\ell_j) \otimes \bC[x] \cong \begin{cases}
     \mathbb{C}^{(j-i)^2}[x] & j>i, d=0 \text{ or } j<i, d=2\\
     \mathbb{C}^{d \choose 2}[x] & i=j\\
     0 & \text{ else }
 \end{cases}.
 \end{equation}
\end{proof}

Note that by \cite{Ab10}, $W_0(K_j)$ split-generates the fully-wrapped Fukaya category of $T^*S^1$. 

\subsection{Product structure}\label{sec:comp}

We compose morphisms from inside to out, meaning from right to left. Due to orthogonality on the A- and B-sides, composition between factors is nonzero in two cases.

{\bf B-side:} Composition in the bounded derived category of coherent sheaves reduces to multiplication of theta functions, restricted to the hypersurface. 

First, we have composition within the $H$ factor and also within the $V \times \bC$ factor. Between factors, we can express the two possible cases roughly as  
$$
\text{Hom}_{D^b(X)}(D^b(H), D^b(V \times \bC))\otimes \text{Hom}_{D^b(X)}(D^b(H), D^b(H)) \to \text{Hom}_{D^b(X)}(D^b(H), D^b(V \times \bC))
$$
and
$$
\text{Hom}_{D^b(X)}(D^b(V\times \bC), D^b(V \times \bC))\otimes \text{Hom}_{D^b(X)}(D^b(H), D^b(V\times \bC)) \to \text{Hom}_{D^b(X)}(D^b(H), D^b(V \times \bC)).
$$
Other compositions involving both $H$ and $V \times \bC$ factors are 0. More specifically, we consider
\begin{equation}
\begin{aligned}
& \Ext^{d_2}(j_*(p^*i^*(\cL^{k_2} \boxtimes \cO_\bC)\otimes \cO_E(E)),\pi^*(\cL^{k_3} \boxtimes \cO_\bC)) \\
&\otimes \Ext^{d_1}(j_*(p^*i^*(\cL^{k_1} \boxtimes \cO_\bC)\otimes \cO_E(E)),j_*(p^*i^*(\cL^{k_2} \boxtimes \cO_\bC)\otimes \cO_E(E)))\\
&\hspace{1in} \to \Ext^{d_1+d_2}(j_*(p^*i^*(\cL^{k_1} \boxtimes \cO_\bC)\otimes \cO_E(E)),\pi^*(\cL^{k_3} \boxtimes \cO_\bC))
\end{aligned}
\end{equation}
which reduces to 
\begin{equation}\label{eq:compn_1}
\Ext^{d_2}(\cL^{k_2}|_H, \cL^{k_3}|_H)[-1] \otimes \Ext^{d_1}(\cL^{k_1}|_H, \cL^{k_2}|_H) \to \Ext^{d_1+d_2}(\cL^{k_1}|_H, \cL^{k_3}|_H)[-1]
\end{equation}
and
$$
\Ext^{d_2}(\pi^*(\cL^{k_2} \boxtimes \cO_\bC),\pi^*(\cL^{k_3} \boxtimes \cO_\bC)) \otimes \Ext^{d_1}(j_*(p^*i^*(\cL^{k_1} \boxtimes \cO_\bC)\otimes \cO_E(E)),\pi^*(\cL^{k_2} \boxtimes \cO_\bC))
$$
\begin{equation}
\hspace{1in} \to \Ext^{d_1+d_2}(j_*(p^*i^*(\cL^{k_1} \boxtimes \cO_\bC)\otimes \cO_E(E)),\pi^*(\cL^{k_3} \boxtimes \cO_\bC))
\end{equation}
which reduces to
\begin{equation}\label{eq:compn_2}
\Ext^{d_2}(\cL^{k_2}|_H, \cL^{k_3}|_H) \otimes \Ext^{d_1}(\cL^{k_1}|_H, \cL^{k_2}|_H)[-1] \to \Ext^{d_1+d_2}(\cL^{k_1}|_H, \cL^{k_3}|_H)[-1]
\end{equation}
where we used the isomorphisms of Equations \eqref{eq:1stpart-nonorthoBside} and \eqref{eq:nonorthoB}
$$
\Ext_X(j_*(p^*i^*(\cL^{k_1} \boxtimes \cO_\bC)\otimes \cO_E(E)),\pi^*(\cL^{k_2} \boxtimes \cO_\bC))\cong \Ext_H(\cL^{k_1}|_H, \cL^{k_2}|_H)[-1]
$$
under which we restrict to $H \times \{0\}$. Sections of $\cL^j$ for $j >0$ are multi-theta functions, infinite series, as in \cite[Equation (2.14), $g=2$]{ACLL_tori}. These are restricted to $H$. Sections of $\cO_\bC$,  polynomials, are evaluated at 0. We calculate   $\text{Ext}^*(\cL^i|_H,\cL^j|_H) \cong H^*(\cL^{j-i}|_H)$ by using the long exact sequence associated to tensoring the short exact sequence
\begin{equation}\label{eq:ses}
    0 \to \cL^{-1} \to \cO_V \to \iota_*\cO_H \to 0
\end{equation}
with $\cL^{\otimes(j-i)}$, for $\iota:H \to V$.  Taking $l=j-i$ and for brevity letting $\mathcal{L}^{\otimes l}$ be denoted as $\mathcal{L}^l$, we find that the cohomology groups of $\cL^l|_H$ can be expressed in terms of cohomology of line bundles on $V$ in \cite[Theorem 3.3]{ACLL_tori}. We have
\begin{equation}\label{eq:homs_Bside_forH}
    H^d(\cL^l|_H) \cong \begin{cases}
        \delta_{d0}\cdot \text{coker}(H^0(\cL^{l-1}) \xrightarrow[]{\otimes \vartheta} H^0(\cL^l)) & l>1\\
        H^{d+1}(\cO_V) & l=1\\
        H^d(\cO_V) & l=0\\
        \delta_{d1} \cdot\text{ker}(H^1(\cL^{l-1}) \xrightarrow[]{\otimes \vartheta} H^1(\cL^l)) & l<0
    \end{cases}, \qquad 0 \leq d \leq 1
\end{equation}
  Composition is given by multiplication of theta functions \cite[\textsection 3.3]{ACLL_tori}. 

{\bf A-side.} Composition in the Fukaya category is defined by counting triangles weighted by area, see \cite[Definition 5.13, $k=2$]{ACLL_Bfield}. 

When composing two morphisms on the A-side, either the first morphism is on the $\cA_L$ factor and then we compose with a morphism in $\text{Hom}(\cA_L,\cA_K)$, or the first morphism is in $\text{Hom}(\cA_L,\cA_K)$ and we compose with one in $\cA_K$. Other compositions, involving both factors, are zero by semi-orthogonality. In other words, the three Lagrangians involved in a given composition will either have two which cover cotangent fiber arcs and one which covers a U-shape, or two which cover U-shapes and one which covers a cotangent fiber arc. Let $\mathbf{L}_j$ denote $L_j$ (respectively $K_j[1]$). For the moduli space of discs $\cD(Y^0, (L_{j_1},\mathbf{L}_{j_2}, K_{j_3}[1]), (p_0,p_1,p_2))$ with Lagrangian boundary conditions defined in \cite[Definition 5.1]{ACLL_Bfield}, let 
\begin{equation}
    \begin{aligned}
        &\cM(Y^0, (L_{j_1},\mathbf{L}_{j_2}, K_{j_3}[1]), (p_0,p_1, p_2); \beta, J)=\\
        &        \hspace{1in}
 \{u \in \cD(Y^0, (L_{j_1},\mathbf{L}_{j_2}, K_{j_3}[1]), (p_0,p_1,p_2)) \mid [u] = \beta, \overline\partial_J(u)=0, E(u)<\infty \}
    \end{aligned}
\end{equation}
be the moduli space of $J$-holomorphic discs of class $\beta$ counted in the composition
\begin{equation}\label{eq:comp}
    HF^*(\mathbf{L}_{j_2},K_{j_3}[1]) \otimes HF^*(L_{j_1},\mathbf{L}_{j_2}) \to HF^*(L_{j_1},K_{j_3}[1])
\end{equation}
depicted under projection of $W_0$ in Figure \ref{fig:1cot_2U} when $\mathbf{L}_{j_2}=L_{j_2}$ (respectively in Figure \ref{fig:2cot_1U} when $\mathbf{L}_{j_2}=K_{j_2}[1]$).

\begin{figure}[h]
    \centering
    {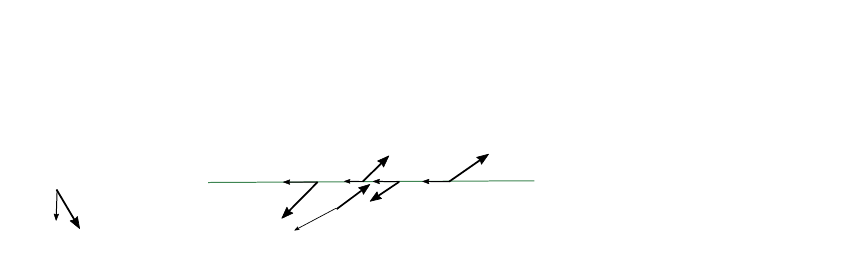}
    \caption{Composition $HF(L_{j_2},K_{j_3}[1]) \otimes HF(L_{j_1},L_{j_2}) \to HF(L_{j_1},K_{j_3}[1])$ depicted in the base. The thicker arrow corresponds to the grading on the input projected Lagrangian of the Floer group. }
    \label{fig:1cot_2U}
\end{figure}

\begin{figure}[h]
    \centering
    \resizebox{6in}{!}{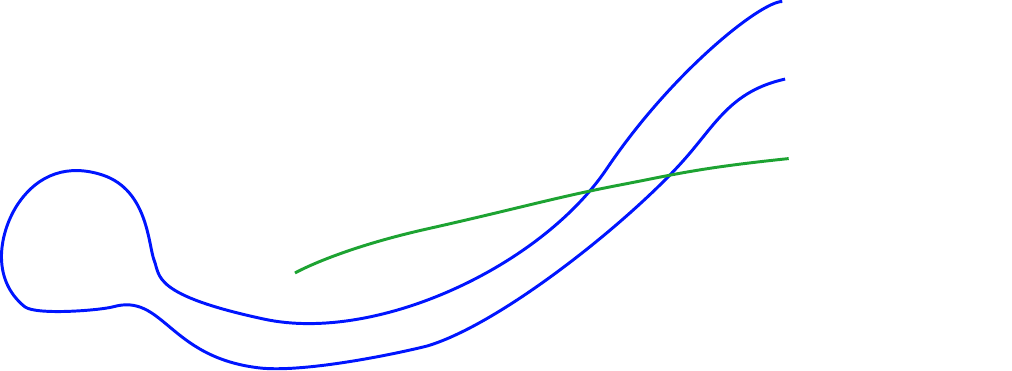}
    \caption{Composition $HF(K_{j_2}[1],K_{j_3}[1]) \otimes HF(L_{j_1},K_{j_2}[1]) \to HF(L_{j_1},K_{j_3}[1])$ depicted in the base.}
    \label{fig:2cot_1U}
\end{figure}

Then we claim that these compositions can be calculated by moving one Lagrangian so that all three points in the base coincide, see Lemma \ref{lem:prism}. Then we can count triangles in the fiber, see Lemma \ref{lem:compn_final}. 

\begin{lemma}\label{lem:prism}  
Consider a pseudo-holomorphic triangle $u \in \cM(Y^0, (L_{j_1},\mathbf{L}_{j_2}, K_{j_3}[1]), (p_0,p_1, p_2); \beta, J)$ counted in the composition 
$$
HF^*(\mathbf{L}_{j_2},K_{j_3}[1]) \otimes HF^*(L_{j_1},\mathbf{L}_{j_2}) \to HF^*(L_{j_1},K_{j_3}[1]),
$$
where $W_0(\beta)$ covers the shaded region in Figure \ref{fig:1cot_2U} for $\mathbf{L}_{j_2}=L_{j_2}$, respectively Figure \ref{fig:2cot_1U} for $\mathbf{L}_{j_2}=K_{j_2}[1]$. Let $u'$ be the new piecewise curve obtained after Lagrangian isotopy of $K_{j_3}[1]$, respectively $L_{j_1}$, so that all three intersection points pass through the same fiber. 

Then 
\begin{enumerate}[label=(\roman*)]
    \item the image of $u'$ is relatively homologous to a triangle $u''$ in a fiber, and
    \item after applying a Hamiltonian diffeomorphism to the isotoped Lagrangian, respectively also a symplectomorphism, $u''$ is bounded by $\ell_{j_1} \cup \ell_{j_2} \cup \ell_{j_3}$ in the fiber over $-S$.
\end{enumerate} 
\end{lemma}

\begin{proof}
    (i) First assume $\mathbf{L}_{j_2}=L_{j_2}$. We will prove the statement by constructing a contractible singular 3-chain $Q$ whose boundary relative the Lagrangians is $[u']-[u'']$. First we describe the construction in the base of $W_0$.
    
    Let $\delta(s,t): I \times I \to \bC$ for $I=[0,1]$ parametrize the projection of the triangle $W_0( u(\bD))$, as follows. In the shaded triangle in Figure \ref{fig:1cot_2U}, on the boundary, the top of the square $I \times \{1\}$ maps to right side of the triangle along $W_0(L_{j_2})$, the left side of the square $\{0\} \times I$ maps to top of the triangle along $W_0(K_{j_3}[1])$, the bottom side of the square $I \times \{0\}$ maps to the the left of the triangle along $W_0(L_{j_1})$, and the right side of the square $\{1\} \times I$ maps to $b_0$. On its boundary, reparametrize the base curves defining the Lagrangians to match $\delta(s,t)$ along the boundary, so that
\begin{equation}\label{eq:match_params}
\begin{aligned}
    \delta(s,1) & = \gamma_{L_{j_2}}(t_{L_{j_2}}(s)) \\
    \delta(0,t) &= \gamma_{K_{j_3}}(t_{K_{j_3}}(t))  \\
    \delta(s,0) &= \gamma_{L_{j_1}}(t_{L_{j_1}}(s)) \\
    \delta(1,t) &= b_0.
    \end{aligned}
\end{equation}
Extend to the interior so that $\delta: (0,1) \times (0,1) \to (W_0\circ u)(\text{int}(\bD))$ is a homeomorphism. Then the Lagrangian isotopy of $K_{j_3}[1]$, in the base, contracts the curve $\{\delta(s,t)\}_{0 \leq s \leq 1}$ for each fixed $t$. That is, at $s=0$, $\delta(0,t) \in W_0(K_{j_3}[1])$, while at $s=1$, $\delta(1,t) = b_0$. 
    
 Now we define the singular chain $Q$ in the total space as a continuous map from the unit cube to $Y^0$. Again we can replace $u$ with $u_D$ as in Equation \eqref{eq:make_section}, over the region in the base defined by $A(s,t)=\delta(s,t)$ here. We extend $\delta(s,t)$ to the total space by parallel transport. Recall $\Phi_{\gamma(0) \to \gamma(1)}$ denotes parallel transport along a path $\gamma$ between the fibers over $\gamma(0) \to \gamma(1)$, Definition \ref{def:par_transport}. Define the singular 3-chain $Q$ to be $u_D$ on the bottom face $I \times I \times \{0\}$, the continuous map 
\begin{equation}\label{eq:Jhol_tri_fiber}
u''=\Phi_{\delta(s,t) \to \delta(1,t)}\circ u_D    
\end{equation}
to the fiber over $b_0$ on the top face $I \times I \times \{1\}$, and interpolate in between:
    \begin{equation}
        \begin{aligned}
              Q:  I \times I \times I & \to Y^0\\
(s,t,h) & \mapsto \begin{cases}
    \Phi_{\delta(s,t) \to \delta(h,t)} u_D(\delta(s,t)),& h \geq s\\
    u_D(\delta(s,t)), & h <s
\end{cases}.
        \end{aligned}
    \end{equation} 

Then $Q$ is a continuous map by definition. The image of $Q$ is depicted in \cite[Figure 2]{ACLL_LG}, the prism over the middle image. Note that the construction there is an example of $P$ in \cite[Equation (6.11)]{ACLL_Bfield} for LG models and describes the boundary of $Q$ constructed here. The map $u'$ is defined by keeping $u$ and then ``stretching" it further as we move the Lagrangian \cite[Equation (6.11)]{ACLL_Bfield}. On $\dd Q$, (paranthetical comments refer to \cite[Figure 2]{ACLL_LG})
\begin{equation}\label{eq:boundaryofQ}
    \begin{aligned}
        h=0, \quad 0 \leq s,t \leq 1 \quad & \mapsto [u] \text{ (bottom triangle) }\\
        h=1, \quad 0 \leq s, t \leq 1 \quad & \mapsto [u''] \text{ (fiber triangle) }\\
        s=0, \quad 0 \leq h, t \leq 1 \quad & \mapsto [u'] - [u] \text{ (top quadrilateral obtained from the Lagrangian isotopy) }\\
        s=1, \quad 0 \leq h, t \leq 1 \quad & \mapsto u_D(b_0) \text{ (constant at bottom vertex) }\\
        t=0, 1, \quad 0 \leq s,h \leq 1 \quad & \mapsto L_{j_1} \cup L_{j_2} \text{ (back and front triangles) }.
    \end{aligned}
\end{equation}

(ii) By Equation \eqref{eq:Jhol_tri_fiber}, the homology class $[u''] \in H_2(T^4)$ is in the fiber over $-S$ with boundary on the following Lagrangians:
\begin{equation}
   \Phi_{ \gamma_{L_{j_1}}(t_{L_{j_1}}(s)) \to \gamma_{L_{j_1}}(0)} \circ \Phi_{\gamma_{L_{j_1}}(0)\to \gamma_{L_{j_1}}(t_{L_{j_1}}(s))}(\ell_{j_1}) = \ell_{j_1}
\end{equation}
and 
\begin{equation}
   \Phi_{ \gamma_{L_{j_2}}(t_{L_{j_2}}(s)) \to \gamma_{L_{j_2}}(0)} \circ \Phi_{\gamma_{L_{j_2}}(0)\to \gamma_{L_{j_2}}(t_{L_{j_2}}(s))}(\ell_{j_2}) = \ell_{j_2}.
\end{equation}
For the third Lagrangian, consider 
\begin{equation}
\begin{aligned}
   \phi_{H_\Delta}:= \Phi_{\delta(1,t) \to \delta(0,t)} \circ \Phi_{\delta_{K_{j_3}}(0) \to \delta_{K_{j_3}}(t_{K_{j_3}}(t))}^{-1}: \bigcup_{0 \leq t \leq 1} Y^0_{\delta_{K_{j_3}}(t_{K_{j_3}}(t))} & \mapsto \bigcup_{0 \leq t \leq 1} Y^0_{\delta_{K_{j_3}}(t_{K_{j_3}}(t))} \\
    \Phi_{\delta_{K_{j_3}}(0) \to \delta_{K_{j_3}}(t_{K_{j_3}}(t))}(p) & \mapsto  \Phi_{\delta(1,t) \to \delta(0,t)}(p), \;\; p \in Y^0_{-S}
\end{aligned}
\end{equation}
because recall that $\delta_{K_{j_3}}(0) = -S = b_0 = \delta(1,t)$ by Definition \ref{def:F2objs}. The two parallel transport maps in $\phi_{H_\Delta}$ compose to parallel transport around a closed loop, $\delta_{K_{j_3}}(t_{K_{j_3}}(t))$ to $\delta_{K_{j_3}}(0)=-S$ backwards along $\delta_{K_{j_3}}$ and then $\delta(1,t)=-S$ backwards along the homotopy back to $\delta(0,t)$, which equals $\delta_{K_{j_3}}(t_{K_{j_3}}(t))$ by Equation \eqref{eq:match_params}. As this loop does not enclose the singular fiber, the composition is Hamiltonian on $\bigcup_{0 \leq t \leq 1} Y^0_{\delta_{K_{j_3}}(t_{K_{j_3}}(t))}$, by \cite[Theorem 6.21(ii)]{symp_intro}, and corresponds to some Hamiltonian function which we call $H_\Delta$ to indicate it's used in counting triangles. We extend $H_\Delta$ to the whole space smoothly and call the resulting Hamiltonian diffeomorphism $\phi_{H_\Delta}$. The reason for applying this Hamiltonian diffeomorphism is so that under the Lagrangian isotopy over $\delta$, the third side of the fiber triangle $u''$ is again in a linear Lagrangian
\begin{equation}
       \Phi_{\delta(0,t)\to \delta(1,t)} \circ \Phi_{\delta(1,t)\to \delta(0,t)}(\ell_{j_3}) = \ell_{j_3}.
\end{equation}
Thus because of Hamiltonian invariance of $HF$, we can replace the Lagrangian in our product and instead consider $u$ to be from
$$
HF(L_{j_2},\phi_{H_\Delta}(K_{j_3}[1])) \otimes HF(L_{j_1},L_{j_2}) \to HF(L_{j_1},\phi_{H_\Delta}(K_{j_3}[1])).
$$
Then $[u''] \in H_2(T^4,\ell_{j_1} \cup \ell_{j_2} \cup \ell_{j_3})$ represents a unique flat holomorphic triangle in the universal cover of $(T^4,\ell_{j_1} \cup \ell_{j_2} \cup \ell_{j_3})$, determined by the original choice of $\beta$ in the statement of the Lemma; the 1-1 condition on the boundary of $u_D$ implies $u''$ does not wrap multiple times, and moreover $u''$ is positively oriented because $u$ was. Thus we can replace $u''$ with a relatively homologous holomorphic curve, one of the curves counted in composition in the Fukaya category of $T^4$, \cite[\textsection 4.3]{ACLL_tori}.

Thus by Equation \eqref{eq:boundaryofQ}
\begin{equation}
\dd Q = 0 = [u]-[u'']+[u']-[u]+(\text{class in Lagrangians}) 
\end{equation}
which implies, relative the Lagrangians,
\begin{equation}\label{eq:homolog_tris}
    [u']=[u''] \in H_2\left(Y^0, L_{j_1} \cup L_{j_2} \cup \bigcup_{0 \leq t \leq 1}\Phi_{\delta(0,t) \to \delta(1,t)}\phi_{H_\Delta}(K_{j_3}[1]|_{\delta(0,t)}))\right)
\end{equation}
where $u''$ is in a $T^4$-fiber bounded by the three linear Lagrangians $\ell_{j_1} \cup \ell_{j_2} \cup \ell_{j_3}$.

The proof for $\mathbf{L}_{j_2} = K_{j_2}[1]$ as in Figure \ref{fig:2cot_1U} is similar, where $\delta(s,t)$ is a homotopy which moves the right side of the triangle, $L_{j_1}$, across to $\epsilon'$. Now $\phi_{H_\Delta}$ is chosen to have $L_{j_1}$ match  $K_{j_2}$ and $K_{j_3}$. Then in a fiber the triangle is bounded by $\Phi_{\delta_{K_j}(0) \to \epsilon'}(\ell_{j_1} \cup \ell_{j_2} \cup \ell_{j_3})$, going from $-S$ to $\epsilon'$ along the path for $W_0(K_j)$ in Definition \ref{def:F2objs}. We then apply the symplectomorphism $\Phi_{\delta_{K_j}(0) \to \epsilon'}^{-1}$ to the whole set-up, obtaining a bijection between triangles over $\epsilon'$ and triangles over $-S$ bounded by linear Lagrangians, preserving areas. 

\end{proof}

Thus, we are now able to compute composition on the A-model. 

\begin{lemma}\label{lem:compn_final}
    The compositions with (i) two U-shaped Lagrangians and one cotangent fiber Lagrangian
    $$
HF(L_{j_2},K_{j_3}[1]) \otimes HF(L_{j_1},L_{j_2}) \to HF(L_{j_1},K_{j_3}[1])
$$
and (ii) with two cotangent fiber Lagrangians and one U-shaped Lagrangian
    $$
HF(K_{j_2}[1],K_{j_3}[1]) \otimes HF(L_{j_1},K_{j_2}[1]) \to HF(L_{j_1},K_{j_3}[1]),
$$
after rescaling, are products in the abelian variety $T^4$ fiber, when $|j_i-j_k| \geq 2$. 
\end{lemma}

\begin{proof} (i) Let $CF((L_{j_1})_{b_0},(L_{j_2})_{b_0})$ denote what we called $CF_{b_0}(\ell_{j_1},\ell_{j_2})$ above, so that $(L_j)_b$ denotes the fiber Lagrangian $L_j \cap Y^0_b$. (i) As in Figure \ref{fig:1cot_2U}: 
\begin{equation}
\begin{aligned}
        CF(L_{j_1},L_{j_2}) & = CF((L_{j_1})_{b_0},(L_{j_2})_{b_0}) \oplus CF((L_{j_1})_{b_1},(L_{j_2})_{b_1})[-1]\\
    CF(L_{j_1},K_{j_3}[1])&=CF((L_{j_1})_{b_2},(K_{j_3}[1])_{b_2}) \oplus CF((L_{j_1})_{b_3},(K_{j_3}[1])_{b_3})[-1]\\
    CF(L_{j_2},K_{j_3}[1]) &= CF((L_{j_2})_{b_4},(K_{j_3}[1])_{b_4}) \oplus CF((L_{j_2})_{b_5},(K_{j_3}[1])_{b_5})[-1].
    \end{aligned}
\end{equation}
By our choice of Lagrangian gradings in Definition \ref{def:grading}, the differential for each cochain complex goes from the fiber over the intersection point on $b_j$ of degree $-1$ to that over $b_{j-1}$ of degree 0, for $j=1,3,5$ respectively. We therefore have exact sequences computing Floer cohomologies:
\begin{equation}
\begin{aligned}
    0 \to CF^{d_1}((L_{j_1})_{b_1},(L_{j_2})_{b_1})[-1] \xrightarrow[]{\partial} CF^{d_1}((L_{j_1})_{b_0},(L_{j_2})_{b_0}) & \to  HF^{d_1}(L_{j_1},L_{j_2})\to 0\\
    0 \to CF^{d_1+d_2}((L_{j_1})_{b_3},(K_{j_3}[1])_{b_3})[-1] \xrightarrow[]{\partial} CF^{d_1+d_2}((L_{j_1})_{b_2},(K_{j_3}[1])_{b_2}) & \to HF^{d_1+d_2}(L_{j_1},K_{j_3}[1])\to 0\\
    0 \to CF^{d_2}((L_{j_2})_{b_5},(K_{j_3}[1])_{b_5})[-1] \xrightarrow[]{\partial} CF^{d_2}((L_{j_2})_{b_4},(K_{j_3}[1])_{b_4}) & \to  HF^{d_2}(L_{j_2},K_{j_3}[1])\to 0
    \end{aligned}
\end{equation}
where $\partial$ counts images of bigon sections of $W_0$ over the bigon in the base between $b_{j-1}$ and $b_j$ for $j=1, 3, 5$ respectively. Analogous to \cite[Remark 4.14]{ACLL_tori}, we have three cases: $j_1+1<j_2<j_3-1$ in which case the degrees are all 0, $j_2+1<j_3<j_1-1$ in which case $d_1=2$ and $d_2=0$, or $j_3+1<j_1<j_2-1$ in which case $d_1=0$ and $d_2=2$. Then the composition 
\begin{equation}
\begin{aligned}
M^2: HF^{d_2}(L_{j_2},K_{j_3}[1]) \otimes HF^{d_1}(L_{j_1},L_{j_2}) 
\to  HF^{d_1+d_2}(L_{j_1},K_{j_3}[1])
\end{aligned}
\end{equation}
can be computed by taking representatives for the cohomology classes in 
\begin{equation}
\begin{aligned}
    M^2: CF^{d_2}((L_{j_2})_{b_4},(K_{j_3}[1])_{b_4}) \otimes CF^{d_1}((L_{j_1})_{b_0},(L_{j_2})_{b_0}) \to CF^{d_1+d_2}((L_{j_1})_{b_2},(K_{j_3}[1])_{b_2})
    \end{aligned}
\end{equation}
which recall, by Lemma \ref{lem:prism}, can be computed in a fiber with linear Lagrangians. This product is well-defined on Floer cohomology as a consequence of the Leibniz rule, \cite[Equation (1.2), $d=2$, \textsection(9j)]{seidel}; for coefficient ring $\mathbb{Z}/2$ the rule is
\begin{equation}
    \partial \circ M^2(\cdot, \cdot) = M^2(\partial(\cdot), \cdot) + M^2(\cdot, \partial(\cdot)).
\end{equation}

   Now we show that we can make the products equal to products in the fiber, by rescaling the intersection points. In the case of $HF(L_{j_2},K_{j_3}[1]) \otimes HF(L_{j_1},L_{j_2}) \to HF(L_{j_1},K_{j_3}[1])$, the Lagrangian $K_{j_3}$ is moved, which moves $b_2$ (output of $M^2$) and $b_4$ (an input of $M^2$). Select intersection points over these base points, $p_1 \in Y^0_{b_4}$ the fiber over $b_4$, and $q \in Y^0_{b_2}$. Let $\rho(u)=e^{-2\pi  \int_{\mathbb{D}} u^*\omega}$ be the weight of $u$, counted in $M^2$, which depends only on the relative homology class $[u]$ because of the Lagrangian boundary conditions. Thus $\rho(u')=\rho(u'')$ by Equation \eqref{eq:homolog_tris}. Then 
    \begin{equation}\label{eq:weight_change}
    \rho(u)\rho(u')^{-1} = \rho(u) \rho(u'')^{-1} = e^{-2\pi  f(p_1)}e^{2\pi  f(q)}
    \end{equation}
    by the last equation in the proof of \cite[Theorem 2.6]{ACLL_LG}. Here $f$ is the function $f=\int_0^1 H_t dt$ where $b_t=dH_t$ for $\psi^*\omega = b \wedge dt$, which follows because Lagrangian isotopies $\psi_t$ in an LG model are exact by \cite[Lemma 2.1]{ACLL_LG}, using that fibered Lagrangians over contractible curves are homotopy equivalent to fiber Lagrangians. 
    
    By Equation \eqref{eq:weight_change}, the weight of the original triangle $u$ and that of the fiber triangle $u''$ differ by an amount that only depends on the intersection points $p_1,q$. We then use  parametrized moduli spaces, varying choices under the Lagrangian isotopy and taking the boundary of a 1-dimensional manifold to identify moduli spaces for the original choices and the new choices under Lagrangian isotopy (denoted with $''$), similar to Remark \ref{rem:param_moduli_1}. We can therefore let
    \begin{equation}
        \cM:= \cM(Y^0, (L_{j_1},{L}_{j_2}, K_{j_3}[1]), (p_0,p_1, q); \beta, J) = \cM(Y^0, (L_{j_1}'',{L}_{j_2}'', K_{j_3}''[1]), (p_0'',p_1'', q''); \beta'', J'').
    \end{equation}
    We find that for $p_0 \in CF((L_{j_1})_{b_0},(L_{j_2})_{b_0})$, by Equation \eqref{eq:weight_change},
    \begin{equation}
    \begin{aligned}
        M^2(p_1,p_0) &= \sum_{\substack{q \in L_{j_1} \cap K_{j_3}[1]\\ \text{ind}[u]=0 \\ [u]:[u]=\beta}} \# \cM \cdot \rho(u) \cdot q =\sum_{\substack{q \in \ell_{j_1} \cap \ell_{j_3}\\ \text{ind}[u'']=0 \\ [u'']:[u'']=\beta''}} \# \cM\cdot \rho(u'') \cdot e^{-2\pi  f(p_1)}e^{2\pi  f(q)} \cdot q\\
        \implies M^2(e^{2\pi  f(p_1)} \cdot p_1,p_0)  &=\sum_{\substack{q \in \ell_{j_1} \cap \ell_{j_3}\\ \text{ind}[u'']=0 \\ [u'']:[u'']=\beta''}} \# \cM\cdot \rho(u'') \cdot (e^{2\pi  f(q)} \cdot q).
        \end{aligned}
    \end{equation}
    Thus rescaling $p_1 \mapsto e^{2\pi f(p_1)}p_1$ and $q \mapsto e^{2\pi f(q)}q$, we have
    \begin{equation}
        M^2(p_1, p_0) = \sum_{q \in \ell_{j_1} \cap \ell_{j_3}} \# C_{p_0p_1q} \cdot q
    \end{equation}
    where $C_{p_0p_1q}$ is the count of triangles between linear Lagrangians in a fiber, see \cite[Equation (2.8)]{Ca20} or the more general \cite[Equation (4.81)]{ACLL_tori}. \newline
    
    (ii) The proof for the second case is similar: we have Floer cochain complexes
    \begin{equation}
\begin{aligned}
        CF(L_{j_1},K_{j_2}[1]) & = CF((L_{j_1})_{c_2},(K_{j_2}[1])_{c_2}) \oplus CF((L_{j_1})_{c_4},(K_{j_2}[1])_{c_4})[-1]\\
    CF(L_{j_1},K_{j_3}[1])&=CF((L_{j_1})_{c_1},(K_{j_3}[1])_{c_1}) \oplus CF((L_{j_1})_{c_3},(K_{j_3}[1])_{c_3})[-1]\\
    CF(K_{j_2}[1],K_{j_3}[1]) & \cong CF((K_{j_2}[1])_{c_0},(K_{j_3}[1])_{c_0})[x]  
    \end{aligned}
\end{equation}
where the differential on the first two goes from the second factor to the first factor, as above, while it is zero in the third case as there are no bigons between linear Lagrangians on a cylinder or torus:
\begin{equation}
\begin{aligned}
    0 \to CF^{d_1}((L_{j_1})_{c_4},(K_{j_2}[1])_{c_4})[-1] \xrightarrow[]{\partial} CF^{d_1}((L_{j_1})_{c_2},(K_{j_2}[1])_{c_2}) & \to  HF^{d_1}(L_{j_1},K_{j_2}[1])\to 0\\
    0 \to CF^{d_1+d_2}((L_{j_1})_{c_3},(K_{j_3}[1])_{c_3})[-1] \xrightarrow[]{\partial} CF^{d_1+d_2}((L_{j_1})_{c_1},(K_{j_3}[1])_{c_1}) & \to HF^{d_1+d_2}(L_{j_1},K_{j_3}[1])\to 0\\
    0 \to CF^{d_2}(K_{j_2},K_{j_3}) & \xrightarrow[]{\cong}  HF^{d_2}(K_{j_2}[1],K_{j_3}[1])\to 0
    \end{aligned}
\end{equation}
Then the composition 
\begin{equation}
\begin{aligned}
M^2: HF^{d_2}(K_{j_2}[1],K_{j_3}[1]) \otimes HF^{d_1}(L_{j_1},K_{j_2}[1]) 
\to  HF^{d_1+d_2}(L_{j_1},K_{j_3}[1])
\end{aligned}
\end{equation}
can again be computed by taking representatives for  cohomology classes in the fibers over $c_1,c_2$. For the third morphism, only a triangle can be formed in the base between  $c_0,c_1,c_2$. Thus we compute  
\begin{equation}
\begin{aligned}
    M^2: CF^{d_2}((K_{j_2}[1])_{c_0},(K_{j_3}[1])_{c_0}) \otimes CF^{d_1}((L_{j_1})_{c_2},(K_{j_2}[1])_{c_2}) \to CF^{d_1+d_2}((L_{j_1})_{c_1},(K_{j_3}[1])_{c_1})
    \end{aligned}
\end{equation}
which again, by Lemma \ref{lem:prism}, can be computed in a fiber with linear Lagrangians after rescaling. 
\end{proof}

\section{Proof of main theorem}\label{sec:proof}

Now we prove the main theorem, which we recall here.

\begin{theorem}
The full subcategories $\mathcal{A}_L$ and $\mathcal{A}_K$ of the Fukaya category $H^* FS(Y^0, W_0)$ are equivalent to $D^b_\cL\Coh(H)$ and $D^b_\cL\Coh(V\times \bC)$, respectively. Furthermore, for the full subcategory $\mathcal{A}_{L,K}:= \mathcal{A}_L \cup \mathcal{A}_K$, there is an equivalence  
\begin{equation}
    \mathcal{A}_{L,K} \xrightarrow[]{\cong} D^b_{\cL} \Coh(X).
\end{equation}
In particular, the semi-orthogonality in $D^b_{\cL} Coh(X)$ is respected.
\end{theorem}

First, we have the following

\begin{theorem}[{\cite[\textsection 6]{Ca20}}]\label{thm:thesis}
There exists a commutative diagram as follows, where vertical arrows are fully faithful.
\begin{center}
\begin{tikzcd}[row sep=large, column sep=large]
D^b_{\mathcal{L}} \Coh(V) \arrow[r, "\iota^*"] \arrow[d, hookrightarrow, "\text{HMS on } V"'] & D^b_{\mathcal{L}} \Coh(H) \arrow[d, hookrightarrow, "\text{HMS on } H = \Sigma_2"] \\
H^0 \mathrm{Fuk}(T^4) \arrow[r, "\cup"] & H^0 \mathrm{FS}(Y, -v_0)
\end{tikzcd}
\end{center}
\end{theorem}

\begin{corollary}[HMS on $\mathcal{A}_L$ factor]\label{thm:factor1}
$\mathcal{A}_L \cong D^b_\mathcal{L}\Coh(H)$ under $L_k \mapsto \mathcal{L}^k|_H$. 
\end{corollary}

\begin{proof}
    This is the statement that 
    \begin{equation}
\Ext(\mathcal{L}^{i}|_{H},\mathcal{L}^{j}|_{H}) \cong HF(L_{i},L_{j})
\end{equation}
for all $i,j \in \mathbb{Z}$. Recall Equation \eqref{eq:generalHFses} that $HF^d(L_i,L_j)$ is defined by the cohomology in degree $d$ of
\begin{equation}
  \begin{aligned}
         & CF^d_{b_1}({\ell}_{i+1},{\ell}_{j})[-1] \oplus CF^{d-1}_{b_0}({\ell}_{i},{\ell}_{j})\xrightarrow[]{\dd^{d-1}} CF^{d+1}_{b_1}({\ell}_{i+1},{\ell}_{j})[-1]\oplus CF^d_{b_0}(\ell_{i}, \ell_{j}) \\
         & \xrightarrow[]{\dd^d} CF^{d+2}_{b_1}({\ell}_{i+1},{\ell}_{j})[-1]\oplus CF^{d+1}_{b_0}(\ell_{i}, \ell_{j}).
         \end{aligned}
    \end{equation}

{\bf Case $j>i+1$.} This is covered in the right vertical arrow of Theorem \ref{thm:thesis}, by identifying U-shapes in $(Y,-v_0)$ to U-shapes in $(Y^0,W_0)$ under the affine transformation on the base from $-v_0 \mapsto W_0 = T^\epsilon(v_0-1)$. 

    {\bf Case $j=i+1$.}
For $d=0$:  
\begin{equation}
   \begin{aligned}
         CF^0_{b_1}({\ell}_{i+1},{\ell}_{i+1})[-1] \oplus \cancel{CF^{-1}_{b_0}({\ell}_{i},{\ell}_{i+1})}& \xrightarrow[]{\dd^{-1}} CF^{1}_{b_1}({\ell}_{i+1},{\ell}_{i+1})[-1]\oplus CF^0_{b_0}(\ell_{i}, \ell_{i+1})\xrightarrow[]{\dd^0}\\
         & \hspace{1.5in} CF^{2}_{b_1}({\ell}_{i+1},{\ell}_{i+1})[-1]\oplus \cancel{CF^{1}_{b_0}(\ell_{i}, \ell_{i+1})}
         \end{aligned}
    \end{equation}
where $\text{im}(\dd^{-1}) = CF^0_{b_0}(\ell_{i}, \ell_{i+1})$ because of the bigon between $b_1$ and $b_0$, and the differential on the $b_1$-fiber is 0 by Equation \eqref{eq:fiberdiffl_zero}. And $\dd^0=0$ by Equations \eqref{eq:fiberdiffl_zero} and \eqref{eq:no-opp-bigon}. So 
\begin{equation}\label{eq:deg0_samez}
HF^0(L_i, L_{i+1})=\ker(\dd^0)/\text{im}(\dd^{-1}) = HF^{1}_{b_1}({\ell}_{i+1},{\ell}_{i+1})\cong \mathbb{C}^{2 \choose 1}=H^1(\mathcal{O}_V)=H^0(\cL|_H)
\end{equation}
by Equation \eqref{eq:homs_Bside_forH} with $d=0$ and $l=(i+1)-i=1$. 

For $d=1$: We found above $\dd^0=0$ so
   \begin{equation}
   \begin{aligned}
         CF^1_{b_1}({\ell}_{i+1},{\ell}_{i+1})[-1] \oplus CF^{0}_{b_0}({\ell}_{i},{\ell}_{i+1})& \xrightarrow[]{\dd^{0}=0} CF^{2}_{b_1}({\ell}_{i+1},{\ell}_{i+1})[-1]\oplus \cancel{CF^1_{b_0}(\ell_{i}, \ell_{i+1})}\xrightarrow[]{\dd^1}\\
         & \hspace{1.5in} \cancel{CF^{3}_{b_1}({\ell}_{i+1},{\ell}_{i+1})[-1]}\oplus \cancel{CF^{2}_{b_0}(\ell_{i}, \ell_{i+1})}.
         \end{aligned}
    \end{equation}
Thus the cohomology is 
\begin{equation}
HF^1(L_i,L_{i+1}) = \ker(\dd^1)/\text{im}(\dd^0) = HF^2_{b_1}({\ell}_{i+1},{\ell}_{i+1})=\bC^{2 \choose 2}=H^2(\mathcal{O}_V)=H^1(\cL|_H)
\end{equation}
by Equation \eqref{eq:homs_Bside_forH} for $d=1$ and $l=1$.

{\bf Case $j=i$.}  
For $d=0$: 
\begin{equation}
   \begin{aligned}
         \cancel{CF^0_{b_1}({\ell}_{i+1},{\ell}_{i})[-1]} \oplus \cancel{CF^{-1}_{b_0}({\ell}_{i},{\ell}_{i})}& \xrightarrow[]{\dd^{-1}} \cancel{CF^{1}_{b_1}({\ell}_{i+1},{\ell}_{i})[-1]}\oplus CF^0_{b_0}(\ell_{i}, \ell_{i})\xrightarrow[]{\dd^0}\\
         & \hspace{1.5in} CF^{2}_{b_1}({\ell}_{i+1},{\ell}_{i})[-1]\oplus CF^{1}_{b_0}(\ell_{i}, \ell_{i}).
         \end{aligned}
    \end{equation}
    By Equations \eqref{eq:fiberdiffl_zero} and \eqref{eq:no-opp-bigon}, $\dd^0=0$, and 
\begin{equation}
HF^0(L_{i}, L_{i})=\ker(\dd^0)/\text{im}(\dd^{-1})=HF^0_{b_0}(\ell_{i}, \ell_{i})=\mathbb{C}^{2 \choose 0} = H^0(\cO_V) = H^0(\cO_H)
\end{equation}
by Equation \eqref{eq:homs_Bside_forH} for $d=l=0$. 

For $d=1$: We found above $\dd^0=0$ so
   \begin{equation}
   \begin{aligned}
         \cancel{CF^1_{b_1}({\ell}_{i+1},{\ell}_{i})[-1]} \oplus CF^{0}_{b_0}({\ell}_{i},{\ell}_{i})& \xrightarrow[]{\dd^{0}=0} CF^{2}_{b_1}({\ell}_{i+1},{\ell}_{i})[-1]\oplus CF^1_{b_0}(\ell_{i}, \ell_{i})\xrightarrow[]{\dd^1}\\
         & \hspace{1.5in} \cancel{CF^{3}_{b_1}({\ell}_{i+1},{\ell}_{i})[-1]}\oplus CF^{2}_{b_0}(\ell_{i}, \ell_{i}).
         \end{aligned}
    \end{equation}
We have a pseudo-holomorphic bigon $\dd^1: CF^2_{b_1}({\ell}_{i+1},{\ell}_i)[-1] \cong \bC \xrightarrow[]{\cong} CF^2_{b_0}(\ell_i, \ell_i)\cong \bC$, so $\ker(\dd^{1})=CF^{1}_{b_0}(\ell_i, \ell_i)$ by Equation \eqref{eq:fiberdiffl_zero}. So 
\begin{equation}
HF^{1}(L_{i},L_{i}) =\ker(\dd^{1})/\text{im}(\dd^{0})=HF^{1}_{b_0}(\ell_i, \ell_i)=\mathbb{C}^{2 \choose 1} = H^1(\cO_V) = H^1(\cO_H)
\end{equation}
by Equation \eqref{eq:homs_Bside_forH} with $d=1$ and $l=0$.

For $d=2$: There is something to check since the cochain group in degree 2 is nonzero.    \begin{equation}
   \begin{aligned}
         CF^2_{b_1}({\ell}_{i+1},{\ell}_{i})[-1] \oplus CF^{1}_{b_0}({\ell}_{i},{\ell}_{i})& \xrightarrow[]{\dd^{1}} \cancel{CF^{3}_{b_1}({\ell}_{i+1},{\ell}_{i})[-1]}\oplus CF^2_{b_0}(\ell_{i}, \ell_{i})\xrightarrow[]{\dd^2} 0.
         \end{aligned}
    \end{equation}
So $\dd^{1}$ is surjective because of the bigon $b_1$ to $b_0$, and $\ker(\dd^2) = CF^2_{b_0}(\ell_{i}, \ell_{i})$. Thus
\begin{equation}
HF^2(L_{i},L_{i})=\ker(\dd^{2})/\text{im}(\dd^{1})=0=H^2(\cO_H)
\end{equation}
because there is no cohomology in degree 2 on the genus 2 curve of complex dimension 1.

{\bf Case $j<i$}: In this case,  $CF(\ell_{i+1},\ell_{j})$ and $CF(\ell_{i},\ell_{j})$ are supported in degree $2$ because the slope of the input Lagrangian $\ell_{i+1}$ and $\ell_{i}$ is greater than the slope of the output Lagrangian $\ell_{j}$, as discussed in \cite[Example 4.6]{ACLL_Bfield}. This phenomenon of changing from degree 0 to degree $2$ when changing the order of the slopes realizes an example of Serre duality as stated in \cite[Lemma 4.7]{ACLL_Bfield}. Thus $HF(L_i,L_j)$ is supported in degree 2 in the fibers. On the B-side, Ext$(\cL^i|_H, \cL^j|_H)$ is supported in degree 1 by Equation \eqref{eq:homs_Bside_forH}. This matches the A-side on the total space, as follows.  We use that $K_{H} = \cL\mid_{H}$ by the adjunction formula, and then Serre duality \begin{equation}\label{eq:Serre_duality_Bside}
H^{d}(H, \cL^{j-i}\vert_{H})\cong H^{1-d}(H,\mathcal{L}^{i-j+1}|_H)    
\end{equation}
to obtain that 
    \begin{equation}
    \begin{aligned}
    \Ext^1(\mathcal{L}^{i}|_{H},\mathcal{L}^{j}|_{H}) & \cong H^{1}(H, \cL^{j-i}\vert_{H})\cong H^0(H,\mathcal{L}^{i-j+1}|_H) \cong HF^0(L_j, L_{i+1}) \cong HF^2(L_i,L_j)[-1].
    \end{aligned}
    \end{equation} 
    The third isomorphism follows from the first case above. The last isomorphism follows from applying monodromy to the Lagrangian $L_{i+1}$, which changes the grading in the base by $-1$ and sends $L_{i+1}$ to $L_i$. That is, $HF^0(L_j, L_{i+1}) \cong HF^2( L_{i+1}, L_j) \cong HF^2(L_i,L_j)[-1]$. So in all cases 
\begin{equation}
\Ext(\mathcal{L}^{i}|_{H},\mathcal{L}^{j}|_{H}) \cong HF(L_{i},L_{j}).
\end{equation}
 
\end{proof}

\begin{corollary}[HMS on $\mathcal{A}_K$ factor]\label{corF2}
$\mathcal{A}_K \cong D^b_{\cL}  \Coh(V \times \bC)$ under $K_j \mapsto \pr_V^*\cL^{\otimes j} \otimes \pr_\bC^* \cO_\bC$.
\end{corollary}

\begin{proof} Lemma \ref{lem:morsF2} describes the morphisms on the A-side, which is a product of the morphisms on a $T^4$ fiber and the $\mathbb{C}^*$ base of the LG model:
\begin{equation}
 HF^*(K_i, K_{j}) \cong HF^*(\ell_i,\ell_j)[x].
\end{equation}
By the K\"unneth formula for Ext, we obtain a product as well on the B-side:
\begin{equation}
 \Ext^*(\pr_V^*\cL^{\otimes i} \otimes \pr_\bC^* \cO_\bC,\pr_V^*\cL^{\otimes j} \otimes \pr_\bC^* \cO_\bC)\cong \Ext^*(\cL^{\otimes i}, \cL^{\otimes j}) \otimes \text{Ext}^*(\mathcal{O}_\bC, \mathcal{O}_\bC).
\end{equation}
By HMS for abelian varieties \cite{fuk_abel}, \cite[Theorem 4.21]{ACLL_tori},
\begin{equation}
    \Ext^*(\cL^{\otimes i}, \cL^{\otimes j}) \cong HF^*(\ell_i,\ell_j)
\end{equation}
and we know as well that 
\begin{equation}
\text{Ext}^*(\mathcal{O}_\bC, \mathcal{O}_\bC)= \mathbb{C}[x]. 
\end{equation}
Thus
\begin{equation}
    \Ext^*(\pr_V^*\cL^{\otimes i} \otimes \pr_\bC^* \cO_\bC,\pr_V^*\cL^{\otimes j} \otimes \pr_\bC^* \cO_\bC) \cong  HF^*(K_i, K_{j}).
\end{equation}
\end{proof}

\begin{proof}[Proof of Main Theorem \ref{thm}] The map 
$$
\mathcal{A}_{L,K}=\langle\mathcal{A}_L,\mathcal{A}_K\rangle \to D^b_{\cL} \Coh (X) = \langle j_*(p^* D^b_{\cL}\Coh(H  \times \{0\})\otimes \cO_p(-1)) , \pi^* D^b_{\cL}\Coh(V \times \bC) \rangle  
$$ 
defined by
\begin{equation}
\begin{aligned}
\mathcal{A}_L \ni L_{j_1} &\mapsto  j_*(p^*\cL^{\otimes {j_1}}|_H \otimes \cO_p(-1)) \in j_*(p^* D^b_{\cL}\Coh(H  \times \{0\})\otimes \cO_p(-1))\\
\mathcal{A}_K \ni K_{j_2} &\mapsto \pi^* (\pr_V^*\cL^{\otimes j_2} \otimes \pr_\bC^* \cO_\bC) \in \pi^* D^b_{\cL}\Coh(V \times \bC)
\end{aligned}
\end{equation}
is an equivalence within each factor of the semi-orthogonal decomposition by Corollaries \ref{thm:factor1} and \ref{corF2}. That is, letting 
\begin{equation}
    \mathcal{L}^{j_1,H}:=j_*(p^*\cL^{\otimes {j_1}}|_H \otimes \cO_p(-1)) , \qquad \mathcal{L}^{j_2,V \times \mathbb{C}} :=\pi^* (\pr_V^*\cL^{\otimes j_2} \otimes \pr_\bC^* \cO_\bC),
\end{equation}
then since $j:E \to X$ is a closed immersion and $\otimes \cO_p(-1)$ is an auto-equivalence,
\begin{equation}
\Ext^*(\mathcal{L}^{j_1,H},\mathcal{L}^{j_2,H})=\Ext^*(\cL^{j_1}|_H, \cL^{j_2}|_H) = HF^*(L_{j_1},L_{j_2}),
\end{equation}
\begin{equation}
\Ext^*(\mathcal{L}^{j_1,V \times \mathbb{C}},\mathcal{L}^{j_2,V \times \mathbb{C}})=HF^*(K_{j_1},K_{j_2}),
\end{equation}
and composition is respected within each factor.

For morphisms between factors, it is zero on both sides in one direction, by Equation \eqref{eq:orthoB} on the B-side and Lemma \ref{lem:orthoA} on the A-side. That is,
$$
\Ext(\mathcal{L}^{j_1,V \times \mathbb{C}}, \mathcal{L}^{j_2,H})=0 = HF(K_{j_1}, L_{j_2}).
$$
In the other direction, the equivalence on morphisms reduces to HMS for $H$, as follows. By Equations \eqref{eq:1stpart-nonorthoBside} and \eqref{eq:nonorthoB}
\begin{equation}
\Ext(\mathcal{L}^{j_2,H},\mathcal{L}^{j_1,V \times \mathbb{C}}) \cong \Ext(\cL^{j_2}|_H, \cL^{j_1}|_H)[-1].
\end{equation}
By Corollary \ref{thm:factor1} and Lemma \ref{lem:relate_to_Ca} this is isomorphic to 
\begin{equation}\label{eq:deg_shft_homs}
    HF(L_{j_2},L_{j_1})[-1] \cong HF(L_{j_2}, K_{j_1}).
\end{equation}

Lastly, composition is respected between the two components by Lemma \ref{lem:compn_final} and HMS for abelian varieties \cite[Proposition 4.18]{ACLL_tori}, \cite{fuk_abel}. That is, the following two diagrams commute:
\[\begin{tikzcd}
 HF^{d_2}(L_{j_2},K_{j_3}[1]) \otimes HF^{d_1}(L_{j_1}, L_{j_2}) \arrow[r, "M^2"]\arrow[d, "\cong "'] & HF^{d_1+d_2}(L_{j_1}, K_{j_3}[1])\arrow[d,"\cong "] \\
 \Ext^{d_2}(\cL^{j_2}|_{H},\cL^{j_3}|_{H}) \otimes \Ext^{d_1}(\cL^{j_1}|_{H},\cL^{j_2}|_{H}) \arrow[r, "\otimes"]& \Ext^{d_1+d_2}(\cL^{j_1}|_{H},\cL^{j_3}|_{H})
 \end{tikzcd}
 \]
 and 
 \[\begin{tikzcd}
 HF^{d_2}(K_{j_2}[1],K_{j_3}[1]) \otimes HF^{d_1}(L_{j_1}, K_{j_2}[1]) \arrow[r, "M^2"]\arrow[d, " \cong"'] & HF^{d_1+d_2}(L_{j_1}, K_{j_3}[1])\arrow[d,"\cong "] \\
 \Ext^{d_2}(\cL^{j_2}|_H,\cL^{j_3}|_H) \otimes \Ext^{d_1}(\cL^{j_1}|_{H},\cL^{j_2}|_{H}) \arrow[r, "\otimes"]& \Ext^{d_1+d_2}(\cL^{j_1}|_{H},\cL^{j_3}|_H)
 \end{tikzcd}.
 \]

\end{proof}

\bibliographystyle{amsalpha}
\bibliography{CV}
\end{document}